\documentclass[11pt,a4paper]{article}
\usepackage[utf8]{inputenc}
\usepackage[T1]{fontenc}
\usepackage{amsmath}
\usepackage{amsthm}
\usepackage{amsfonts}
\usepackage{amssymb}
\usepackage{mathrsfs}
\usepackage{url}
\usepackage{mathdots}
\usepackage[english]{babel}
\usepackage{geometry}
\usepackage{verbatim}
\usepackage[bookmarks=false]{hyperref}
\usepackage{float}
\usepackage{authblk}
\usepackage{enumerate}
\usepackage{tikz-cd}
\usepackage{extarrows}
\usepackage{comment}

\hypersetup{
    bookmarks=true,         % show bookmarks bar?
    unicode=false,          % non-Latin characters in Acrobat<E2><80><99>s bookmarks
    pdftoolbar=true,        % show Acrobat<E2><80><99>s toolbar?
    pdfmenubar=true,        % show Acrobat<E2><80><99>s menu?
    pdffitwindow=false,     % window fit to page when opened
    pdfstartview={FitH},    % fits the width of the page to the window
    pdftitle={On global rigid inner forms},    % title
    pdfauthor={Olivier Taïbi},     % author
    colorlinks=true,       % false: boxed links; true: colored links
    linkcolor=blue,          % color of internal links
    citecolor=green,        % color of links to bibliography
    filecolor=green,      % color of file links
    urlcolor=cyan}           % color of external links

\newcommand{\Q}{\mathbb{Q}}
\newcommand{\A}{\mathbb{A}}
\newcommand{\R}{\mathbb{R}}
\providecommand{\C}{\mathbb{C}}
\renewcommand{\C}{\mathbb{C}}

\newcommand{\Z}{\mathbb{Z}}

\newcommand{\Hom}{\mathrm{Hom}}

\newtheorem{theo}{Theorem}[subsection]

\newtheorem{lemm}[theo]{Lemma}

\newtheorem*{prob}{Problem}

\newtheorem{coro}[theo]{Corollary}
\newtheorem{defi}[theo]{Definition}
\newtheorem{prop}[theo]{Proposition}
\newtheorem{rema}[theo]{Remark}

\numberwithin{equation}{subsection}

\begin{document}

\baselineskip=16pt

\author{Olivier Taïbi}

\title{On global rigid inner forms}

\maketitle

\begin{abstract}
We give an explicit construction of global Galois gerbes constructed more
abstractly by Kaletha to define global rigid inner forms. This notion is crucial
to formulate Arthur's multiplicity formula for inner forms of quasi-split
reductive groups. As a corollary, we show that any global rigid inner form is
almost everywhere unramified, and we give an algorithm to compute the resulting
local rigid inner forms at all places in a given finite set. This makes global
rigid inner forms as explicit as global pure inner forms, up to computations in
local and global class field theory.
\end{abstract}

\newpage
\setcounter{tocdepth}{2}
\tableofcontents
\newpage

\section{Introduction}

Let $F$ be a number field, and $G$ a connected reductive group over $F$.
Following seminal work of Labesse-Langlands \cite{LabLan} and Shelstad,
Langlands, Kottwitz and Arthur \cite{ArthurUnip} conjectured a multiplicity
formula for discrete automorphic representations for $G$, in terms of
Arthur-Langlands parameters $\psi : L_F \times \mathrm{SL}_2(\C) \rightarrow
{}^L G$. The formulation of this conjecture on automorphic multiplicities
requires a precise version of the local Arthur-Langlands correspondence for
$G_{F_v} := G \times_F F_v$ at all places $v$ of $F$, describing individual
elements of local packets using the theory of endoscopy. For this it is
necessary to endow each $G_{F_v}$ with a \emph{rigidifying datum}. For places
$v$ such that $G_{F_v}$ is quasi-split, that is for all but finitely many
places of $F$, this can take the form of a Whittaker datum $\mathfrak{w}_v$. If
$G$ is quasi-split, then one can choose a global Whittaker datum
$\mathfrak{w}$, and it is expected that taking localizations $\mathfrak{w}_v$
of $\mathfrak{w}$ yields a coherent family of precise versions of the local
Arthur-Langlands correspondence. This coherence is crucial for the automorphic
multiplicity formula to hold. For example this is the setting used in
\cite{Arthur} and \cite{Mokunitary}. Note that even though a choice of global
Whittaker datum is necessary to express the formula for automorphic
multiplicities, these multiplicities are canonical, as one can easily deduce
from \cite[Theorem 4.3]{Kalgeneric}.

In general the connected reductive group $G$ might not be quasi-split, and $G$
is only an inner form of a unique quasi-split group. Recall (see
\cite{BorelCorvallis}) that two connected reductive groups have isomorphic
Langlands dual groups if and only if they are inner forms of each other. Vogan
\cite{VoganLL} and Kottwitz conjectured a formulation of the local Langlands
correspondence in the case where $G_{F_v}$ is a \emph{pure} inner form of a
quasi-split group. In this case a rigidifying datum is a quadruple $(G^*_v,
\Xi_v, z_v, \mathfrak{w}_v)$ where $G^*_v$ is a connected reductive quasi-split
group over $F_v$, $\Xi_v : (G^*_v)_{\overline{F_v}} \rightarrow
G_{\overline{F_v}}$ is an isomorphism, and $z_v \in Z^1(F_v, G^*_v)$ is such
that for any $\sigma \in \mathrm{Gal}(\overline{F_v}/F_v)$ we have $\Xi_v^{-1}
\sigma(\Xi_v) = \mathrm{Ad}(z_v(\sigma))$. If globally $G$ is a pure inner form
of a quasi-split group, one can choose a similar global quadruple $(G^*, \Xi,
z, \mathfrak{w})$, and localizing at all places of $F$ seems to yield a
coherent family of rigidifying data. Away from a finite set $S$ of places of
$F$, the restriction $z_v$ of $z$ to a decomposition group
$\mathrm{Gal}(\overline{F_v}/F_v)$ is cohomologically trivial, and writing it as
a coboundary yields an isomorphism $\Xi'_v : G^*_{F_v} \simeq G_{F_v}$
well-defined up to conjugation by $G(F_v)$, which endows $G_{F_v}$ with a
Whittaker datum $(\Xi'_v)_*(\mathfrak{w}_v)$ in a canonical way. Furthermore, up
to enlarging $S$ this can be done integrally, that is over a finite étale
extension of $\mathcal{O}(F_v)$, so that $\Xi'_v$ is an isomorphism between the
canonical models of $G^*$ and $G$ over $\mathcal{O}(F_v)$.

Unfortunately not all connected reductive groups can be realized as pure inner
forms of quasi-split groups, due to the fact that $H^1(F, G) \rightarrow H^1(F,
G_{\mathrm{ad}})$ can fail to be surjective. The simplest example is certainly
the group of elements having reduced norm equal to $1$ in a non-split
quaternion algebra, an inner form of $\mathrm{SL}_2$, considered in
\cite{LabLan}. To circumvent this problem, Kaletha defined larger Galois
cohomology groups in \cite{Kalri} for the local case and in \cite{Kalgri} for
the global case. More precisely, he constructed central extensions (Galois
gerbes bound by commutative groups in the terminology of \cite{LanRap})
\[ 1 \rightarrow P_v \rightarrow \mathcal{E}_v \rightarrow
\mathrm{Gal}(\overline{F_v}/F_v) \rightarrow 1 \]
in the local case, $v$ any place of $F$, and
\[ 1 \rightarrow P \rightarrow \mathcal{E} \rightarrow
\mathrm{Gal}(\overline{F}/F) \rightarrow 1 \]
in the global case. Here $P_v$ and $P$ are inverse limits of finite algebraic
groups defined over $F_v$ or $F$, and we have denoted by $P_v \rightarrow
\mathcal{E}_v$ the extension denoted by $u \rightarrow W$ in \cite{Kalri}, to
emphasize the analogy between the local and global cases. The central extensions
are obtained from certain classes $\xi_v \in H^2(F_v, P_v)$, $\xi \in H^2(F,
P)$. Using these central extensions Kaletha defined, for $Z$ a finite central
algebraic subgroup of $G$, certain groups of $1$-cocycles
\[ Z^1(P_v \rightarrow \mathcal{E}_v, Z(\overline{F_v}) \rightarrow
G(\overline{F_v})) \supset Z^1(F_v, G_{F_v}) \]
\[ Z^1(P \rightarrow \mathcal{E}, Z(\overline{F}) \rightarrow G(\overline{F}))
\supset Z^1(F, G). \]
Kaletha also proposed precise formulations of the local Langlands conjecture and
Arthur multiplicity formula, using rigidifying data $(G^*_v, \Xi_v, z_v,
\mathfrak{w}_v)$ (resp.\ $(G^*, \Xi, z, \mathfrak{w})$) where now $z_v$ (resp.\
$z$) belongs to this larger group of $1$-cocycles. For $Z$ large enough, for
example if $Z$ contains the center of the derived subgroup of $G$, the map
\[ H^1(P \rightarrow \mathcal{E}, Z(\overline{F}) \rightarrow G(\overline{F}))
\rightarrow H^1(F, G_{\mathrm{ad}}) \]
is surjective, and so any $G$ can be endowed with such a rigidifying datum
$(G^*, \Xi, z, \mathfrak{w})$. From such a global rigidifying datum, one
obtains local rigidifying data by localization. Each localization $z_v =
\mathrm{loc}_v(z)$ of $z$ is defined by pulling back via a morphism of central
extensions
\begin{equation} \label{e:locdiagintro}
\begin{tikzcd}
1 \arrow[r] & P_v \arrow[r] \arrow[d] & \mathcal{E}_v \arrow[r] \arrow[d] & \mathrm{Gal}(\overline{F_v}/F_v) \arrow[r] \arrow[d] & 1 \\
1 \arrow[r] & P \arrow[r]             & \mathcal{E} \arrow[r]             & \mathrm{Gal}(\overline{F}/F) \arrow[r]               & 1
\end{tikzcd}
\end{equation}
and extending coefficients from $G^*(\overline{F})$ to $G^*(\overline{F_v})$.

In this paper we give an explicit, bottom-up realization of the central
extension
\[ 1 \rightarrow P \rightarrow \mathcal{E} \rightarrow
\mathrm{Gal}(\overline{F}/F) \rightarrow 1 \]
constructed in \cite{Kalgri}. Here ``bottom-up'' means that our construction is
naturally an inverse limit of central extensions of
$\mathrm{Gal}(\overline{F}/F)$ by finite algebraic groups $P_k$, with $P =
\varprojlim_{k \geq 0} P_k$. We also give bottom-up realizations of
localization morphisms \eqref{e:locdiagintro} and generalized Tate-Nakayama
morphisms for tori (\cite[Theorem 3.7.3]{Kalgri}, which generalizes
\cite{Tatecohotori}), as well as compatibilities between them. We also show
(Proposition \ref{p:xiiscanclass}) that our construction recovers the
``canonical class'' defined abstractly in \cite[\S 3.5]{Kalgri}. Apart from
giving alternative proofs of some results in \cite{Kalgri}, our construction
has the benefit that it allows one to compute with global rigid inner forms
``at finite level'', that is using a \emph{finite} Galois extension of the base
field $F$. In particular, we deduce that global rigid inner forms are almost
everywhere unramified (Proposition \ref{p:unrae}), a fact which is obvious for
pure inner forms, but surprisingly not for rigid inner forms. In the future our
construction could be used to prove further properties of Kaletha's canonical
class.

Our direct construction is also useful for explicit applications using Arthur's
formula for automorphic multiplicities. Computing spaces of automorphic forms,
along with action of a Hecke algebra, is much easier for definite groups, which
are not quasi-split. Once such spaces are computed, one would like to interpret
Hecke eigenforms as being related to (ersatz) motives, and Arthur's multiplicity
formula makes this relation precise (see \cite{TaiMult} for some cases for which
rigid inner forms are needed). For this it is necessary to compute localizations
of rigidifying data, more precisely to solve the following problem.

\begin{prob}
Given a connected reductive group $G$ over a number field $F$, find
\begin{itemize}
\item a global rigidifying datum $\mathcal{D} = (G^*, \Xi, z, \mathfrak{w})$,
\item a finite set $S$ of places of $F$ containing all archimedean places and
all non-archimedean places $v$ such that $G_{F_v}$ is ramified,
\item a reductive model of $\underline{G}$ over the ring $\mathcal{O}_{F,S}$ of
$S$-integers in $F$ such that for any $v \not\in S$, the localization
$\mathcal{D}_v$ of $\mathcal{D}$ at $v$ is unramified with respect to the
integral model $\underline{G}_{\mathcal{O}_{F_v}}$ of $G_{F_v}$,
\item for each $v \in S$, an explicit description of the localization
$\mathcal{D}_v$ of $\mathcal{D}$ at $v$.
\end{itemize}
\end{prob}

Above ``unramified'' means that $\mathrm{loc}_v(z) \in B^1(F_v, G)$, and that
the resulting isomorphism $\Xi_v' : G^*_{F_v} \simeq G_{F_v}$, which is
well-defined up to composing with conjugation by an element of $G(F_v)$,
identifies the conjugacy class of $\mathfrak{w}_v$ with a Whittaker datum for
$G_{F_v}$ compatible with the integral model $\underline{G}_{\mathcal{O}(F_v)}$,
in the sense of \cite{CasSha}. At almost all places this is implied by the fact
that $\mathfrak{w}_v$ is compatible with the canonical model of $G^*$ and the
fact that $\mathrm{loc}_v(z) \in Z^1(F_v^{\mathrm{unr}}/F_v, G^*)$, but for
applications it is desirable to keep $S$ as small as possible. For $v \in S$,
the meaning of ``explicit description of $\mathcal{D}_v$'' is somewhat vague. In
the case where $\mathrm{loc}_v(z)$ is cohomologically trivial this simply means
a Whittaker datum for $G_{F_v}$. In general it means describing the localization
$\mathcal{D}_v$ in a purely local fashion, so that it could be compared to a
reference rigidifying datum. We give detailed steps to solve this problem in
section \ref{s:effloc}, reducing the computation of localizations at places in
$S$ to computations in local and global class field theory. We give an example
in section \ref{s:example} in a case where $G$ is a definite inner form of
$\mathrm{SL}_2$ over $F=\Q(\sqrt{3})$ which is split at all finite places, and
for $S$ the set of archimedean places, that is in ``level one''. It can be
generalized effortlessly, and without additional computations, to the analogous
inner forms of $\mathrm{Sp}_{2n}$ over $F$, for arbitrary $n \geq 2$.

Let us explain why this problem does not appear to be directly solvable using
constructions in \cite{Kalgri}, which might be surprising when one considers the
case of pure inner forms, as it is straightforward to restrict a
$1$-cocycle to a decomposition group. For explicit computations one can only
work with finite extensions of $F$, and finite Galois modules. Although the
localization maps \eqref{e:locdiagintro} are canonical, unfortunately they do
not arise from \emph{canonical} morphism of central extensions of Galois groups
by \emph{finite} Galois modules, because of the possible non-vanishing of
$H^1(F_v, P_k)$, where $P = \varprojlim_k P_k$. Similarly, the possible
non-vanishing of $H^1(F, P_k)$ means that inflation morphisms
\begin{equation} \label{e:infdiagintro}
\begin{tikzcd}
1 \arrow[r] & P_{k+1} \arrow[r] \arrow[d] & \mathcal{E}_{k+1} \arrow[r] \arrow[d] & \mathrm{Gal}(\overline{F}/F) \arrow[r] \arrow[d] & 1 \\
1 \arrow[r] & P_k \arrow[r]             & \mathcal{E}_k \arrow[r]             & \mathrm{Gal}(\overline{F}/F) \arrow[r]               & 1
\end{tikzcd}
\end{equation}
are not defined canonically, where $\mathcal{E}_k$ is the central extension
obtained using a $2$-cocycle in the cohomology class of the image of $\xi$ in
$H^2(F, P_k)$. For applications to generalized Tate-Nakayama isomorphisms,
Kaletha shows that these ambiguities are innocuous using a clever indirect
argument (Lemma 3.7.10 in \cite{Kalgri}) in cohomology (but only in cohomology).
Note that in the local case, Kaletha gave an explicit construction of the
inflation maps analogous to \eqref{e:infdiagintro}: see \cite[\S 4.5]{Kalri}.

Our construction is a global analogue. The main difficulty lies in formulating
and proving the analogue of \cite[Lemma 4.4]{Kalri} (which draws on \cite[\S
VI.1]{Langlandsdebuts}) in the global case. First we reinterpret \cite[Lemma
4.4]{Kalri} using a modification $\mathrm{AW}^2$ of the Akizuki-Witt map on
$2$-cocycles (\cite[Ch.\ XV]{ArtinTate}) occurring in the construction of Weil
groups attached to class formations. We study this modification systematically
in section \ref{s:AW}, in particular we observe that it is more flexible while
retaining the interpretation in terms of central extensions. It is not difficult
to establish the analogue of \cite[Lemma 4.4]{Kalri} where local fundamental
cocycles are replaced by global fundamental cocycles. However, in Tate-Nakayama
isomorphisms these global fundamental cocycles control Galois cohomology groups
such as $H^1(E/F, T(\A_E)/T(E))$, where $T$ is a torus over $F$ split by the
finite Galois extension $E/F$, whereas we are interested in cohomology groups
such as $H^1(E/F, T(E))$. These are controlled by \emph{Tate cocycles} defined
by Tate in \cite{Tatecohotori}, essentially as a consequence of the
compatibility between local and global fundamental $2$-cocycles. Unfortunately
these do not seem to have an interpretation using the Akizuki-Witt map, and this
makes the global case more challenging. We give an ad hoc definition of a
certain map $\mathrm{AWES}^2$ in Definition \ref{d:AWES2}, which is compatible
with the corestriction map in Eckmann-Shapiro's lemma for modules which are
\emph{twice} induced. This definition is crucial for the main technical result
of this article, Theorem \ref{t:constructionTate}, constructing a family of Tate
cocycles compatible under $\mathrm{AWES}^2$, as well as local-global
compatibility with local fundamental cocycles. We give a second proof as
preparation for the algorithm in section \ref{s:effloc}. Once this is proved, we
construct Kaletha's generalized Tate-Nakayama morphisms at the level of cocycles
in section \ref{s:genTN}, and prove compatibilities with respect to inflation
and localization. In particular we obtain an explicit version of Kaletha's
localization maps at finite level and for cocycles. Although these explicit
localization maps are not canonical, as they depend on a number of choices
detailed in the paper to form cocycles, they are compatible with inflation and
so yield a localization map between towers of central extensions (see
Proposition \ref{p:towerlocalglobal}).

As mentioned above, a consequence is that global rigid inner forms are
unramified away from a finite set (Proposition \ref{p:unrae}), which is not
obvious from the definition using cohomology classes. After the first version of
this paper was written, we found a short proof of this ramification property
using only Kaletha's characterization of the canonical class in \cite[\S
3.5]{Kalgri}. This proof is included in section \ref{s:ramification}, along
with an example of a ``non-canonical'' class, which does not satisfy this
ramification property.

The author thanks Tasho Kaletha for numerous interesting discussions on rigid
inner forms and for his encouragement.

\section{Notation}
\label{s:notation}

Let $F$ be a number field. We denote by $\A$ the ring of adèles for $F$. Let
$\overline{F}$ be an algebraic closure of $F$. All algebraic extensions of $F$
considered will be subextensions of $\overline{F}$. If $E$ is an algebraic
extension of $F$, let $\mathcal{O}(E)$ be its ring of integers, $\A_E = E
\otimes_F \A$, $I(E) = \A_E^{\times}$ the group of idèles and $C(E) =
I(E)/E^{\times}$ the group of idèle classes. Let $V$ be the set of all places of
$F$. If $S \subset V$ and $E$ is an algebraic extension of $F$, denote by $S_E$
the set of places of $E$ above $S$. If $S$ is a set of places of $F$ or $E$
containing all archimedean places, let $I(E,S)$ be the subgroup of $I(E)$
consisting of idèles which are integral units away from $S$, and
$\mathcal{O}(E,S)$ the ring of $S$-integral elements of $E$. For $S \subset V$
let $\overline{F}_S$ be the maximal subextension of $\overline{F}/F$ unramified
outside $S$, and $\mathcal{O}_S = \mathcal{O}(\overline{F}_S, S)$. For $E$ an
algebraic extension of $F$ and $u \in V_E$, we will denote by $\mathrm{pr}_u$
the projection $\A_E \rightarrow E_u$. For $v \in V$ we will denote by
$\mathrm{pr}_v$ the projection $\A_{\overline{F}} \rightarrow \overline{F}
\otimes_F F_v$.

As in \cite{Kalgri} we fix a tower $(E_k)_{k \geq 0}$ of increasing finite
Galois extensions of $F$, with $E_0 = F$ and $\bigcup_k E_k = \overline{F}$.
Choose an increasing sequence $(S_k)_{k \geq 0}$ of finite subsets of $V$ such
that $S_0$ contains all archimedean places of $F$, $S_k$ contains all
non-archimedean places of $F$ ramifying in $E_k$, and $I(E_k, S_k)$ maps onto
$C(E_k)$. We also fix a set $\dot{V} \subset V_{\overline{F}}$ of
representatives for the action of $\mathrm{Gal}(\overline{F}/F)$, that is
$\dot{V}$ contains a place of $\overline{F}$ above every place of $F$. For $E$ a
Galois extension of $F$ and $S' \subset V$ let $\dot{S}'_E$ be the set of places
of $E$ below $\dot{V}$ and above $S'$, so that $\dot{S}'_E$ is a set of
representatives for the action of $\mathrm{Gal}(E/F)$ on $S'_E$. We can assume
that $\dot{V}$ is chosen so that for any finite Galois extension $E/F$ and
$\sigma \in \mathrm{Gal}(E/F)$, there exists $\dot{v} \in \dot{V}_E$ such that
$\sigma \cdot \dot{v} = \dot{v}$. In fact the sets of representatives
$\dot{V}_{E_k}$, for the action of $\mathrm{Gal}(E_k/F)$ on $V_{E_k}$, can be
chosen sequentially as $k$ increases. For $v \in V$ and $k \geq 0$ we will
denote by $\dot{v}_k$ the unique place in $\dot{V}_{E_k}$ above $v$. These
hypotheses on $(S_k)_{k \geq 0}$ are weaker than Conditions 3.3.1 in
\cite{Kalgri}, which are necessary to obtain surjectivity properties in
cohomology. For $v \in S$ let $\overline{F_v} = \varinjlim_k E_{k, \dot{v}_k}$,
an algebraic closure of $F_v$, so that we have a well-defined inclusion
$\mathrm{Gal}(\overline{F_v}/F_v) \subset \mathrm{Gal}(\overline{F}/F)$.

If $A$ is an abelian group and $N \geq 1$ is an integer, $A[N]$ denotes the
$N$-torsion subgroup of $A$. If $A$ is a finite abelian group, $\exp(A)$ is the
exponent of $A$, i.e.\ the smallest $N \geq 1$ such that $A[N]=A$. We will
denote the group law of most abelian groups multiplicatively, excepti notably
for groups of characters or cocharacters of tori. If $G$ is a group and $A$ a
$G$-module, $A^G \subset A$ is the subgroup of $G$-invariants. If in addition $G
= \mathrm{Gal}(E/F)$, we will write $N_{E/F}$ for the norm map, and
$A^{N_{E/F}}$ for the subgroup of elements killed by $N_{E/F}$.

\section{Preliminaries}
\subsection{A modification of the Akizuki-Witt map}
\label{s:AW}

Consider $G$ a finite group, $N$ a normal subgroup.  If $s : G/N \rightarrow G$
is a section such that $s(1)=1$ and $A$ is a $G$-module, with group law written
multiplicatively, then for $\alpha \in Z^2(G, A)$,
\begin{equation} \label{e:trueAW}
\widetilde{\mathrm{AW}}(\alpha) : (\sigma, \tau) \mapsto \prod_{n \in N} \frac{n\left(\alpha(s(\sigma), s(\tau))\right) \times \alpha(n, s(\sigma) s(\tau))}{\alpha(n, s(\sigma \tau))}
\end{equation}
defines an element of $Z^2(G/N, A^N)$, the cohomology class of which only
depends on that of $\alpha$ \cite[Ch.\ XIII, \S 3]{ArtinTate}, so that
$\widetilde{\mathrm{AW}}$ descends to a map $H^2(G,A) \rightarrow
H^2(G/N,A^N)$. We refer to \cite[Ch.\ XIII, \S 3]{ArtinTate} for the natural
interpretation of $\widetilde{\mathrm{AW}}$ in terms of central group
extensions. Using the $2$-cocycle relation for $\alpha$ at $(n, s(\sigma),
s(\tau))$ we can express \eqref{e:trueAW} as
\[ \prod_{n \in N} \frac{\alpha(n, s(\sigma)) \times \alpha(n s(\sigma), s(\tau))}{\alpha(n, s(\sigma \tau))} = \prod_{n \in N} \frac{\alpha(n, s(\sigma)) \times \alpha(\tilde{\sigma} n, s(\tau))}{\alpha(n, s(\sigma \tau))} \]
where $\tilde{\sigma} \in G$ is any lift of $\sigma$, not necessarily equal to
$s(\sigma)$.  Using the $2$-cocycle relation for $\alpha$ at $(\tilde{\sigma},
n, s(\tau))$ we can also rewrite this as
\[ \widetilde{\mathrm{AW}}(\alpha)(\sigma, \tau) = \prod_{n \in N} \left(\frac{\alpha(n, s(\sigma)) \times \tilde{\sigma} \left(\alpha(n, s(\tau)) \right)}{\alpha(n, s(\sigma \tau))} \times \frac{\alpha(\tilde{\sigma}, n s(\tau))}{\alpha(\tilde{\sigma}, n)} \right). \]

The following shows that with an appropriate choice of $\alpha$ in its
cohomology class, this expression simplifies.
\begin{lemm} \label{l:goodcocycle}
In any cohomology class in $H^2(G, A)$, there is a $2$-cocycle $\alpha$ such that for all $n \in N$ and $\sigma \in G/N$, $\alpha(n, s(\sigma)) = 1$.
\end{lemm}
\begin{proof}
It is well-known that any cohomology class contains a $2$-cocycle $\alpha$ such that for all $\sigma \in G$, $\alpha(\sigma, 1) = 1 = \alpha(1, \sigma)$.
We choose such an $\alpha$, and we will construct $\beta : G \rightarrow A$ such that $\alpha \mathrm{d}(\beta)$ satisfies the required property.
Let $\beta(1)=1$, and choose the values of $\beta$ on $N \smallsetminus \{1\}$ and $s(G/N \smallsetminus \{1\})$ arbitrarily.
For $n \in N$ and $\sigma \in G/N$,
\[ \mathrm{d} \beta (n, s(\sigma)) = \frac{\beta(n) \times n (\beta(s(\sigma)))}{\beta(n s(\sigma))}, \]
and we are led to define $\beta(n s(\sigma)) = \alpha(n, s(\sigma)) \times
\beta(n) \times n \left( \beta(s(\sigma)) \right)$ for $n \in N \smallsetminus \{1\}$ and $\sigma \in G/N \smallsetminus \{1\}$.
Note that this equality also holds when $n=1$ or $\sigma=1$.
\end{proof}

This motivates to the following modification $\mathrm{AW}^2$ of the Akizuki-Witt
map $\widetilde{\mathrm{AW}}$.
\begin{defi}
Let $\Gamma$ be an extension of $G$, i.e.\ $\Gamma$ is a group endowed with a surjective morphism $\Gamma \rightarrow G$.
Let $A$ be a commutative group, with group law written multiplicatively.
For $\alpha : \Gamma \times G \rightarrow A$, define $\mathrm{AW}^2(\alpha) : \Gamma \times G/N \rightarrow A$ by
\[ \mathrm{AW}^2(\alpha)(\sigma, \tau) = \prod_{n \in N} \frac{\alpha(\sigma, n \tilde{\tau})}{\alpha(\sigma, n)} \]
where $\sigma \in \Gamma$, $\tau \in G/N$ and $\tilde{\tau} \in G$ is any lift of $\tau$.
\end{defi}
Although this coincides with the original Akizuki-Witt map a priori only for
classes $\alpha$ as in Lemma \ref{l:goodcocycle} (for $A$ a $G$-module and
$\Gamma = G$), this definition has the advantage that it does not require a
choice of section $s$, and will be more convenient for taking cup-products.
Moreover it is defined in a slightly more general setting, since it does not
involve an action of $G$ on $A$. This property will make ``extracting $N$-th
roots'' in section \ref{s:genTN} almost harmless. The definition has the
disadvantage that, even when $A$ is a $G$-module, $\Gamma = G$ and $\alpha \in
Z^2(G, A)$, it is not automatic that $\mathrm{AW}^2(\alpha)$ factors through
$G/N \times G/N$ or takes values in $A^N$.

\begin{prop} \label{p:propAW}
Let $\Gamma$ be an extension of $G$, i.e.\ $\Gamma$ is a group endowed with a
surjective morphism $\Gamma \rightarrow G$.  Let $A$ be a $\Gamma$-module.  For
$\alpha \in Z^{2,1}(\Gamma, G, A)$ (notation as in \cite[\S 4.3]{Kalri}), we
have $\mathrm{AW}^2(\alpha) \in Z^{2,1}(\Gamma, G/N, A)$.

Moreover, if $\Gamma = G$ then $\sigma \mapsto \prod_{n \in N} \alpha(n,
\sigma)$ descends to a map $G/N \rightarrow A/A^N$ mapping $1$ to $1$, and the
following are equivalent:
\begin{enumerate}
\item $\mathrm{AW}^2(\alpha)$ factors through $G/N \times G/N$,
\item for all $n \in N$ and $\tau \in G/N$, $\mathrm{AW}^2(\alpha)(n, \tau) = 1$,
\item for all $\sigma \in G$, $\prod_{n \in N} \alpha(n, \sigma) \in A^N$.
\end{enumerate}
If these conditions are satisfied, then $\mathrm{AW}^2(\alpha) \in
Z^2(G/N, A^N)$ belongs to the same  cohomology class as
$\widetilde{\mathrm{AW}}(\alpha)$ and we have a morphism of central extensions
\begin{alignat*}{2}
A \underset{\alpha}{\boxtimes} G & \longrightarrow & A^N
\underset{\mathrm{AW}^2(\alpha)}{\boxtimes} G/N \\
x \boxtimes \sigma & \longmapsto & \left( \prod_{n \in N} n(x) \alpha(n, \sigma)
\right) \boxtimes \overline{\sigma}.
\end{alignat*}
\end{prop}
We omit the proof, since we will not use this result in the paper.

In order to investigate the effect on $\mathrm{AW}^2(\alpha)$ of the choice of
$\alpha$ in its cohomology class, let us define a second map $\mathrm{AW}^1$ on
$1$-cochains.

\begin{defi}
Let $A$ be a commutative group.
For $\beta : G \rightarrow A$, define $\mathrm{AW}^1(\beta) : G/N \rightarrow A$ by the formula $\mathrm{AW}^1(\beta)(\sigma) = \prod_{n \in N} \beta(n \tilde{\sigma}) / \beta(n)$, where $\tilde{\sigma} \in G$ is any lift of $\sigma \in G/N$.
\end{defi}

\begin{prop} \label{p:dAWcomm}
Suppose $\Gamma$ is an extension of $G$, and $A$ is a $\Gamma$-module.  For any
$\beta : G \rightarrow A$, we have $\mathrm{d}(\mathrm{AW}^1(\beta)) =
\mathrm{AW}^2(\mathrm{d}(\beta))$ in $Z^{2,1}(\Gamma, G/N, A)$.
\end{prop}
\begin{proof}
For $\sigma \in \Gamma$ and $\tau \in G/N$, denoting $\bar{\sigma}$ the image
of $\sigma$ in $G$, we have
\begin{align*}
\mathrm{d}(\mathrm{AW}^1(\beta))(\sigma, \tau) =& \prod_{n \in N} \frac{\beta(n \bar{\sigma})}{\beta(n)} \frac{\sigma\left( \beta(n \tilde{\tau}) \right)}{\sigma\left( \beta(n) \right)} \frac{\beta(n)}{\beta(n \bar{\sigma} \tilde{\tau})} \\
=& \prod_{n \in N} \frac{\beta(n \bar{\sigma}) \sigma\left( \beta(n \tilde{\tau}) \right)}{\beta(n \bar{\sigma} \tilde{\tau}) \sigma\left( \beta(n) \right)} \\
\end{align*}
and
\begin{align*}
\mathrm{AW}^2\left( \mathrm{d}(\beta) \right)(\sigma, \tau) =& \prod_{n \in N} \frac{\beta(\bar{\sigma}) \sigma\left( \beta(n \tilde{\tau}) \right)}{\beta(\bar{\sigma} n \tilde{\tau})} \frac{\beta(\bar{\sigma} n)}{\beta(\bar{\sigma}) \sigma\left( \beta(n) \right)} \\
=& \prod_{n \in N} \frac{\sigma\left( \beta(n \tilde{\tau}) \right)}{\beta(\bar{\sigma} n \tilde{\tau})} \frac{\beta(\bar{\sigma} n)}{\sigma\left( \beta(n) \right)}.
\end{align*}
\end{proof}

\begin{lemm} \label{l:surjAW}
The maps
\[ \left\{ \beta : G \rightarrow A\,\middle|\, \beta(1)=1 \right\} \rightarrow \left\{ \beta : G/N \rightarrow A\,\middle|\, \beta(1)=1 \right\} \]
induced by $\mathrm{AW}^1$ and
\[ \left\{ \alpha : \Gamma \times G \rightarrow A\,\middle|\,\alpha(\sigma,1)=1 \text{ for all } \sigma \in \Gamma \right\} \rightarrow \left\{ \alpha : \Gamma \times G/N \rightarrow A\,\middle|\,\alpha(\sigma,1)=1 \text{ for all } \sigma \in \Gamma \right\} \]
induced by $\mathrm{AW}^2$ are both surjective.
\end{lemm}
\begin{proof}
Let $s : G/N \rightarrow G$ be a section such that $s(1)=1$.
Restricting $\mathrm{AW}^1$ to the set of $\beta : G \rightarrow A$ such that $\beta|_N = 1$ and $\beta(n s(\sigma))=1$ for $\sigma \in G/N \smallsetminus \left\{ 1 \right\}$ and $n \in N \smallsetminus \left\{ 1 \right\}$ yields a bijective map onto $\left\{ \beta : G/N \rightarrow A\,\middle|\, \beta(1)=1 \right\}$.

Similarly, restricting $\mathrm{AW}^2$ to the set of $\alpha : \Gamma \times G \rightarrow A$ such that
\begin{itemize}
\item for all $\sigma \in \Gamma$ and $n \in N$, $\alpha(\sigma, n) = 1$,
\item for all $\sigma \in \Gamma$, $n \in N \smallsetminus \{1\}$ and $\tau \in G/N \smallsetminus \{1\}$, $\alpha(\sigma, n s(\tau))=1$,
\end{itemize}
yields a bijective map onto $\left\{ \alpha : \Gamma \times G/N \rightarrow A\,\middle|\,\alpha(\sigma,1)=1 \text{ for all } \sigma \in \Gamma \right\}$.
\end{proof}

The following corollary is readily deduced from the previous Lemma and
Proposition \ref{p:dAWcomm}.

\begin{coro} \label{c:surjAW}
Suppose that $A$ is a $G$-module. Consider $c \in H^2(G, A)$, and let $\alpha_N
\in Z^2(G/N, A^N)$ be in the cohomology class of the image of $c$ under
$\widetilde{\mathrm{AW}}$. Assume that $\alpha_N(1,1)=1$.  Then there exists
$\alpha \in c$ such that $\alpha(1,1)=1$ and $\mathrm{AW}^2(\alpha) =
\alpha_N$.
\end{coro}
Note that we have not imposed that $\alpha$ should satisfy the property in Lemma
\ref{l:goodcocycle}, and indeed it can happen that no such $\alpha$ maps to
$\alpha_N$ under $\mathrm{AW}^2$. A simple computation shows that if we fix a
section $s : G/N \rightarrow G$ as above, then for $\alpha, \alpha' \in c$ as in
Lemma \ref{l:goodcocycle}, $\mathrm{AW}^2(\alpha/\alpha') \in B^2(G/N, N_N(A))$
where
\[ N_N(A) = \{ \prod_{n \in N} n(x) \,|\, x \in A \}. \]

\subsection{Explicit Eckmann-Shapiro}
\label{s:ES}

Let $G$ be a finite group acting transitively on the left on a set $X$.
Choose $x_0 \in X$ and let $H$ be the stabilizer of $x_0$, so that $X \simeq G/H$.

Let $A$ be a left $H$-module.
Define
\[ \mathrm{ind}_H^G(A) = \{ f : G \rightarrow A \,|\, \forall h \in H, g \in G,\,f(hg)=h \cdot f(g) \}. \]
It is naturally a left $G$-module by defining $(g_1 \cdot f)(g_2) = f(g_2 g_1)$.
We can identify $A$ with the sub-$H$-module of $\mathrm{ind}_H^G(A)$ consisting of all functions whose support is contained in $H$, by evaluating such functions at $1$.
Choose $R$ a set of representatives for $G/H$.
Then $\mathrm{ind}_H^G(A) = \bigoplus_{r \in R} r \cdot A$.
For simplicity we assume that $1 \in R$.

If $A$ happens to be a $G$-module, then
\begin{equation} \label{e:GAmodiso}
f \mapsto (g \mapsto g \cdot f(g^{-1}))
\end{equation}
defines an isomorphism of $G$-modules $\phi^G_H$ between $\mathrm{ind}_H^G(A)$ and $\mathrm{Maps}(X, A)$.
The sub-$H$-module $A$ of $\mathrm{ind}_H^G(A)$ corresponds to functions supported on $x_0$ under this isomorphism.
\begin{prop} \label{p:explicitES}
For $c \in Z^2(H, A)$, the following formula defines a $2$-cocycle $\mathrm{ES}^2_R(c)$ for the $G$-module $\mathrm{ind}_H^G(A)$:
\[ \mathrm{ES}^2_R(c)(r_1 h_1 r_2^{-1}, r_2 h_2 r_3^{-1})(h_3 r_1^{-1}) = h_3 \left( c(h_1, h_2) \right) \]
where $r_1, r_2, r_3 \in R$ and $h_1, h_2, h_3 \in H$.
The morphism $\mathrm{ES}^2_R : Z^2(H, A) \rightarrow Z^2(G, \mathrm{ind}_H^G(A))$ induces an isomorphism $H^2(H, A) \rightarrow H^2(G, \mathrm{ind}_H^G(A))$.

If $A$ happens to be a $G$-module, then using the identification \eqref{e:GAmodiso} we can write
\begin{equation} \label{e:explicitESphi}
\phi^G_H\left(\mathrm{ES}^2_R(c)(r_1 h_1 r_2^{-1}, r_2 h_2 r_3^{-1})\right)(r_1 \cdot x_0) = r_1 \left( c(h_1, h_2) \right).
\end{equation}
\end{prop}
This isomorphism in cohomology is well-known.
We will need these explicit formulas for cocycles in the sequel.

\begin{prop}
\label{p:explicitESdegree1}
For $c \in C^1(H, A)$, define $\mathrm{ES}_R^1(c) \in C^1(G, \mathrm{ind}_H^G(A))$ by
\[ \mathrm{ES}_R^1(c)(r_1 h_1 r_2^{-1})(h_2 r_1^{-1}) = h_2 \left( c(h_1) \right) \]
for $r_1, r_2 \in R$ and $h_1, h_2 \in H$.
Then $\mathrm{d}(\mathrm{ES}_R^1(c)) = \mathrm{ES}_R^2(\mathrm{d}(c))$.
\end{prop}

\section{Construction of Tate cocycles in a tower}

Recall that Tate \cite{Tatecohotori} constructed a Tate-Nakayama isomorphism,
which in particular describes $H^1(E_k/F, T)$ in elementary terms, for any $k
\geq 0$ and any torus $T$ defined over $F$ and split by $E_k$.  This isomorphism
is obtained by taking cup-products with the Tate class, called $\alpha$ in
\cite{Tatecohotori}.

Our first goal is to construct a \emph{compatible} family of Tate
\emph{cocycles}
\[ \alpha_k \in Z^2(E_k/F, \mathrm{Maps}(V_{E_k}, I(E_k))) \]
for the Galois extensions $E_k/F$.  We will give a precise meaning to technical
notion of ``compatibility'' in Theorem \ref{t:constructionTate}. For now
we simply mention that this compatibility is the global analogue of \cite[Lemma
4.4]{Kalri}.

Recall that for $k$ fixed, Tate cocycles $\alpha_k$ are obtained as follows.
\begin{enumerate}
\item
Choose a representative $\overline{\alpha}_k \in Z^2(E_k/F, C(E_k))$ of the
fundamental class for $E_k/F$.
\item
For each place $v$ of $F$, choose a representative $\alpha_{k,v} \in Z^2(
E_{k,\dot{v}} / F_v, E_{k,\dot{v}}^{\times})$ of the fundamental class for
$E_{k,\dot{v}}/F_v$. Let $\alpha_k' \in
Z^2(E_k/F, \mathrm{Maps}(V_{E_k}, I(E_k)))$ be such that for any $v \in V$, the
$2$-cocycle
\begin{align*}
\mathrm{Gal}(E_{k,\dot{v}}/F_v)^2 \longrightarrow & I(E_k) \\
(\sigma, \tau) \longmapsto & \alpha_k'(\sigma, \tau)(\dot{v}_k)
\end{align*}
is cohomologous to $\alpha_{k,v}$ composed with $j_{k,v} :
E_{k,\dot{v}}^{\times} \hookrightarrow I(E_k)$.  Explicitly, $\alpha_k'$ can be
obtained from $(\alpha_{k,v})_{v \in V}$ using \eqref{e:explicitESphi}.
\item
Choose $\overline{\beta}_k \in C^1( E_k/F , \mathrm{Maps}(V_{E_k}, C(E_k)) )$
such that $\overline{\alpha}_k / \overline{\alpha'_k} =
\mathrm{d}(\overline{\beta}_k)$, where $\overline{\alpha}_k$ is seen as taking
values in the set of constant maps $V_{E_k} \rightarrow C(E_k)$ and
$\overline{\alpha'_k}$ is the composition of $\alpha'_k$ with the natural map
$\mathrm{Maps}(V_{E_k}, I(E_k)) \rightarrow \mathrm{Maps}(V_{E_k}, C(E_k))$.
\item
Lift $\overline{\beta}_k$ to $\beta_k \in C^1( E_k/F , \mathrm{Maps}(V_{E_k},
I(E_k)) )$ arbitrarily, and define $\alpha_k = \alpha'_k \times
\mathrm{d}(\beta_k)$.
\end{enumerate}
In this section we will show that each step can be done compatibly with
Akizuki-Witt-like maps.  For cocycles $\alpha_{k,v}$ this was done in
\cite[Lemma 4.4]{Kalri}, we will however give a slightly different
construction, using Corollary \ref{c:surjAW}. The case of
$\overline{\alpha}_k$ is very similar. A key point of the construction will be
the definition (see \ref{d:AWES2}) of an ``Akizuki-Witt-Eckmann-Shapiro'' map
relating the maps $\mathrm{AW}$ for local and global Galois groups, and formula
\eqref{e:explicitESphi} (see Lemma \ref{l:AWadelic}).

\subsection{Global fundamental cocycles}
\label{s:globfundcocy}

For any $k \geq 0$, the image of the fundamental class in $H^2(E_{k+1}/F,
C(E_{k+1}))$ under the Akizuki-Witt map \eqref{e:trueAW} (for the normal
subgroup $\mathrm{Gal}(E_{k+1}/E_k)$, and any choice of section) is the
fundamental class in $H^2(E_k/F, C(E_k))$. This is a direct consequence of
\cite[Ch.\ XIII, Theorem 6]{ArtinTate}. For $i \in \{1,2\}$ write
$\mathrm{AW}^i_k$ for the maps $\mathrm{AW}^i$ defined in section \ref{s:AW},
for the normal subgroup $\mathrm{Gal}(E_{k+1}/E_k)$ of
$\mathrm{Gal}(E_{k+1}/F)$. Using Corollary \ref{c:surjAW} we see that there
exists a family $\left( \overline{\alpha}_k \right)_{k \geq 0}$ where each
$\overline{\alpha}_k \in Z^2(E_k/F, C(E_k))$ represents the fundamental class,
and such that for all $k \geq 0$ we have $\overline{\alpha}_k =
\mathrm{AW}^2_k(\overline{\alpha}_{k+1})$.

\begin{rema} \label{r:altAW}
Alternatively, one could construct such a family using a method similar to
\cite[\S 4.4]{Kalri} (and so \cite[\S VI.1]{Langlandsdebuts}), that is by
choosing sections $\mathrm{Gal}(E_{k+1} / E_k) \rightarrow W_{E_k}$, where
$W_{E_k}$ is the Weil group of $E_k$, and multiplying them to produce sections
$\mathrm{Gal}(E_k / F) \rightarrow W_{E_k / F}$, yielding fundamental cocycles
compatible with $\mathrm{AW}^2_k$.

A third way would be to use a compacity argument and Lemma \ref{l:goodcocycle},
as in the proof of Theorem \ref{t:constructionTate} (using $2$-cochains
instead of $1$-cochains). The details for this last alternative are left to the
reader.
\end{rema}

\subsection{Local and adèlic fundamental classes}
\label{s:locfundcocyc}

Fix $v \in V$. For $i \in \{1,2\}$ write $\mathrm{AW}^i_{k,v}$ for the maps
$\mathrm{AW}^i$ defined in section \ref{s:AW}, for the normal subgroup
$\mathrm{Gal}(E_{k+1, \dot{v}}/E_{k, \dot{v}})$ of $\mathrm{Gal}(E_{k+1,
\dot{v}}/F_v)$. As in the global case we can use Corollary \ref{c:surjAW}
inductively to produce a family $\left( \alpha_{k,v} \right)_{k \geq 0}$ where
$\alpha_{k,v} \in Z^2(E_{k, \dot{v}}/F_v, E_{k, \dot{v}}^{\times})$ represents
the fundamental class and for all $k \geq 0$, we have $\alpha_{k,v} =
\mathrm{AW}^2_{k,v}(\alpha_{k+1,v})$. Alternatively we could simply use
\cite[Lemma 4.4]{Kalri}: see Remark \ref{r:altAW}.

For each $k \geq 1$, choose representatives for $\mathrm{Gal}(E_k/E_{k-1}) /
\mathrm{Gal}(E_{k,\dot{v}}/E_{k-1, \dot{v}})$, and choose lifts of these
representatives in $\mathrm{Gal}(\overline{F} / E_{k-1})$ to obtain a finite set
$R_{k, v} \subset \mathrm{Gal}(\overline{F} / E_{k-1})$.  We can and do assume
that $1 \in R_{k, v}$.  For convenience we also define $R_{0,v} = \{1\} \subset
\mathrm{Gal}(\overline{F} / F)$.  For any $k \geq 0$, $R'_{k,v} := R_{0,v}
R_{1,v} \dots R_{k,v} \subset \mathrm{Gal}(\overline{F} / F)$ projects to a set
of representatives for $\mathrm{Gal}(E_k/F) / \mathrm{Gal}(E_{k,\dot{v}}/F_v)$.
For $v \in V$ and $k \geq 0$ let $\zeta_{k,v} : \{v\}_{E_k} \rightarrow
\{v\}_{E_{k+1}}$ be the section such that for all $r \in R'_{k, v}$,
$\zeta_{k,v}(r \cdot \dot{v}_k) = r \cdot \dot{v}_{k+1}$.  Define $j_{k,v} :
E_{k, \dot{v}}^{\times} \hookrightarrow I(E_k)$ by $(j_{k,v}(x))_{\dot{v}_k} =
x$ and $(j_{k,v}(x))_w = 1$ for $w \neq \dot{v}_k$.  We have natural inclusions
$E_{k, \dot{v}}^{\times} \subset E_{k+1, \dot{v}}^{\times}$ and for $x \in E_{k,
\dot{v}}^{\times}$ we have
\[ j_{k,v}(x) = \prod_{r \in R_{k+1,v}} r \left( j_{k+1,v}(x) \right). \]

Following Proposition \ref{p:explicitES} define, for all $k \geq 0$, $\alpha'_k
\in Z^2(E_k/F, \mathrm{Maps}(V_{E_k}, I(E_k)))$ by
\[ \alpha'_k(r_1 \sigma r_2^{-1}, r_2 \tau r_3^{-1})(r_1 \cdot \dot{v}_k) = r_1
\left( j_{k, v} (\alpha_{k,v}(\sigma, \tau)) \right) \]
for $v \in V$, $\sigma, \tau \in \mathrm{Gal}(E_{k, \dot{v}}/F_v)$ and $r_1,
r_2, r_3 \in R'_{k,v}$.
That is, $\alpha'_k$ is obtained by aggregating
\[ \phi_{ \mathrm{Gal}( E_{k, \dot{v}}/F_v) }^{ \mathrm{Gal}(E_k/F) } (
\mathrm{ES}^2_{R'_{k,v}}( j_{k,v}(\alpha_{k,v}))) \in Z^2(E_k/F, \mathrm{Maps}(
\{v\}_{E_k}, I(E_k))) \]
for $v \in V$.

\begin{defi} \label{d:AWES2}
Suppose that $A$ is a commutative group. For $k \geq 0$ and $\alpha :
\mathrm{Gal}(\overline{F} / F) \times \mathrm{Gal}(E_{k+1} / F) \rightarrow
\mathrm{Maps}(V_{E_{k+1}}, A)$, define $\mathrm{AWES}^2_k(\alpha) :
\mathrm{Gal}(\overline{F} / F) \times \mathrm{Gal}(E_k / F) \rightarrow
\mathrm{Maps}(V_{E_k},A)$ by
\[ \mathrm{AWES}^2_k(\alpha)(\sigma, \tau)(\sigma_k \tau \cdot w) := \prod_{n
\in \mathrm{Gal}(E_{k+1}/E_k)} \frac{\alpha(\sigma, n \tilde{\tau})(\sigma_{k+1}
n \tilde{\tau} \cdot \zeta_{k,v}(w)) }{\alpha(\sigma, n)(\sigma_{k+1} n \cdot
\zeta_{k,v}(\tau \cdot w))}. \]
In this formula $\sigma \in \mathrm{Gal}(\overline{F} / F)$ has image
$\sigma_{k+1}$ in $\mathrm{Gal}(E_{k+1} / F)$ and $\sigma_k$ in
$\mathrm{Gal}(E_k / F)$, $\tau \in \mathrm{Gal}(E_k / F)$ and $\tilde{\tau} \in
\mathrm{Gal}(E_{k+1} / F)$ is any lift of $\tau$, $v \in V$ and $w \in
\{v\}_{E_k}$.
\end{defi}

Note that $\mathrm{AWES}^2_k$ depends on the choice of representatives
$R'_{k,v}$ only via $\zeta_{k,v}$.

\begin{lemm} \label{l:AWadelic}
For all $k \geq 0$ we have $\mathrm{AWES}^2_k(\alpha'_{k+1}) = \alpha'_k$.
\end{lemm}
Note that a priori the left hand side is only a map $\mathrm{Gal}(E_{k+1} / F)
\times \mathrm{Gal}(E_k / F) \rightarrow \mathrm{Maps}(V_{E_k}, I(E_{k+1}))$.
The lemma implies that is is inflated from a map $\mathrm{Gal}(E_k / F)^2
\rightarrow \mathrm{Maps}(V_{E_k}, I(E_k))$.
\begin{proof}
Fix $\sigma \in \mathrm{Gal}(E_{k+1}/F)$, $\tau \in \mathrm{Gal}(E_k/F)$ and
$\gamma \in R'_{k,v}$.  In $\mathrm{Gal}(E_k/F)$ write $\tau \gamma = r_2 g_2$
where $r_2 \in R'_{k,v}$ and $g_2 \in \mathrm{Gal}(E_{k,\dot{v}}/F_v)$.  Let
$\widetilde{g_2} \in \mathrm{Gal}(E_{k+1, \dot{v}} / F_v)$ be any lift of $g_2$.
Note that
\begin{align*}
 \{ n \tilde{\tau} \,|\, n \in \mathrm{Gal}(E_{k+1}/E_k) \} =& \{ r_2 u n_v \widetilde{g_2} \gamma^{-1} \,|\, u \in R_{k+1,v}, \, n_v \in \mathrm{Gal}(E_{k+1,\dot{v}}/E_{k,\dot{v}}) \}, \\
\mathrm{Gal}(E_{k+1}/E_k) =& \{ r_2 u n_v r_2^{-1} \,|\, u \in R_{k+1,v}, \, n_v \in \mathrm{Gal}(E_{k+1,\dot{v}}/E_{k,\dot{v}}) \}.
\end{align*}
In $\mathrm{Gal}(E_k/F)$ write $\sigma_k r_2 = r_1 g_1$ where $r_1 \in R'_{k,v}$
and $g_1 \in \mathrm{Gal}(E_{k,\dot{v}}/F_v)$.  Choose $\widetilde{g_1} \in
\mathrm{Gal}(E_{k+1,\dot{v}}/F_v)$ lifting $g_1$.  For every $u \in R_{k+1,v}$
we can decompose $\sigma r_2 u \in \mathrm{Gal}(E_{k+1}/F)$ as follows: $\sigma
r_2 u = r_1 u' \widetilde{g_1} x_v$ where $u' \in R_{k+1,v}$ and $x_v \in
\mathrm{Gal}(E_{k+1,\dot{v}}/E_{k,\dot{v}})$ depend on $u$.  Moreover $u \mapsto
u'$ realises a bijection from $R_{k+1,v}$ to itself: $r_1^{-1} \sigma r_2 \in
\mathrm{Gal}(E_{k+1} / E_k)$ induces a permutation of the set of places of
$E_{k+1}$ lying above $\dot{v}_k$.
\begin{align*}
\mathrm{AWES}^2_k(\alpha'_{k+1})(\sigma, \tau)(\sigma_k \tau \gamma \cdot \dot{v}_k) =& \prod_{n \in \mathrm{Gal}(E_{k+1}/E_k)} \frac{\alpha'_{k+1}(\sigma, n \tilde{\tau})(\sigma n \tilde{\tau} \gamma \cdot \dot{v}_{k+1}) }{\alpha'_{k+1}(\sigma, n)(\sigma n r_2 \cdot \dot{v}_{k+1})} \\
=& \prod_{u, n_v} \frac{\alpha'_{k+1}(r_1 u' \widetilde{g_1} x_v (r_2 u)^{-1}, r_2 u n_v \widetilde{g_2} \gamma^{-1})(r_1 u' \cdot \dot{v}_{k+1})}{\alpha'_{k+1}(r_1 u' \widetilde{g_1} x_v (r_2 u)^{-1}, r_2 u n_v r_2^{-1})(r_1 u' \cdot \dot{v}_{k+1})} \\
=& \prod_u r_1u'\left( \prod_{n_v} \frac{j_{k+1,v}\left(\alpha_{k+1,v}(\widetilde{g_1} x_v, n_v \widetilde{g_2})\right)}{j_{k+1,v}\left(\alpha_{k+1,v}(\widetilde{g_1} x_v, n_v) \right)} \right) \\
=& \prod_u r_1u'\left( j_{k+1, v}\left(\alpha_{k,v}(g_1,g_2)\right) \right) \\
=& r_1\left( j_{k,v}(\alpha_{k,v}(g_1,g_2)) \right) \\
=& \alpha_k'(r_1 g_1 r_2^{-1}, r_2 g_2 \gamma^{-1})(r_1 \cdot \dot{v}_k) \\
=& \alpha_k'(\sigma, \tau)(\sigma \tau \gamma \cdot \dot{v}_k).
\end{align*}
\end{proof}

\begin{rema}
One could define $\mathrm{AWES}^2$ axiomatically, as we did for $\mathrm{AW^2}$
in Section \ref{s:AW}, for general quadruples $(G, N, H, R_{G/N}, R_N)$ where
$G$ is a finite group, $N$ a normal subgroup of $G$, $H$ a subgroup of $G$,
$R_{G/N} \subset G$ a set of representatives for $G/HN = (G/N) / (HN/N)$ such
that $1 \in R_{G/N}$, and $R_N \subset N$ a set of representatives for $N/(N
\cap H)$ such that $1 \in R_N$. One could also state the generalization of Lemma
\ref{l:AWadelic} in this context, with an identical proof. Note that it would
apply to $2$-cocycles $\alpha'$ taking values in a \emph{twice} induced module,
that is $\Z[G/H] \otimes_{\Z} \mathrm{ind}_H^G(A)$ for some $H$-module $A$. We
will not need this generality, however.
\end{rema}

\subsection{Properties of $\mathrm{AWES}^2_k$}

To establish the analogue of Proposition \ref{p:dAWcomm}, we introduce variants
of $\mathrm{AWES}^2_k$ in degrees $0$ and $1$.
\begin{defi} \label{d:AWES01}
Fix $k \geq 0$.
\begin{enumerate}
\item
Suppose that $A$ is a commutative group.  For $\beta : \mathrm{Gal}(E_{k+1}/F)
\rightarrow \mathrm{Maps}(V_{E_{k+1}}, A)$, define $\mathrm{AWES}^1_k(\beta) :
\mathrm{Gal}(E_k/F) \rightarrow \mathrm{Maps}(V_{E_k}, A)$ by
\[ \mathrm{AWES}^1_k(\beta)(\sigma)(\sigma \cdot w) = \prod_{n \in
\mathrm{Gal}(E_{k+1}/E_k)} \frac{\beta(n \tilde{\sigma})(n \tilde{\sigma} \cdot
\zeta_{k,v}(w)) }{\beta(n)(n \cdot \zeta_{k,v}(\sigma \cdot w))} \]
for $\sigma \in \mathrm{Gal}(E_k/F)$ and $w \in \{v\}_{E_k}$.  In this formula
$\tilde{\sigma} \in \mathrm{Gal}(E_{k+1}/F)$ is any lift of $\sigma$.
\item
Suppose that $A$ is a $\mathrm{Gal}(E_{k+1}/E_k)$-module.
For $\beta \in \mathrm{Maps}(V_{E_{k+1}}, A)$ define $\mathrm{AWES}^0_k(\beta) \in \mathrm{Maps}(V_{E_k}, A^{\mathrm{Gal}(E_{k+1}/E_k)})$ by
\[ \mathrm{AWES}^0_k(\beta)(w) = N_{E_{k+1}/E_k}(\beta(\zeta_{k,v}(w)) \]
for $w \in \{v\}_{E_k}$.
\end{enumerate}
\end{defi}

\begin{lemm} \label{l:dAWES}
Fix $k \geq 0$.
\begin{enumerate}
\item
Suppose that $A$ is a $\mathrm{Gal}(\overline{F} / F)$-module. For $\beta :
\mathrm{Gal}(E_{k+1}/F) \rightarrow \mathrm{Maps}(V_{E_{k+1}}, A)$, we have the
equality of maps $\mathrm{Gal}(\overline{F}/F) \times \mathrm{Gal}(E_k/F)
\rightarrow \mathrm{Maps}(V_{E_k}, A)$
\[ \mathrm{AWES}^2_k(\mathrm{d}(\beta)) = \mathrm{d}(\mathrm{AWES}^1_k(\beta)).
\]
\item
Suppose that $A$ is a $\mathrm{Gal}(E_{k+1} / F)$-module. For $\beta \in
\mathrm{Maps}(V_{E_{k+1}}, A)$, we have the equality of maps
$\mathrm{Gal}(E_{k+1}/F) \rightarrow \mathrm{Maps}(V_{E_k}, A)$
\[ \mathrm{AWES}^1_k(\mathrm{d}(\beta)) = \mathrm{d}(\mathrm{AWES}^0_k(\beta)).
\]
The right hand side is a map $\mathrm{Gal}(E_k/F) \rightarrow \mathrm{Maps}(
V_{E_k}, N_{E_{k+1}/E_k}(A) )$.
\end{enumerate}
\end{lemm}
\begin{proof}
\begin{enumerate}
\item
Let $v \in S$, $w \in \{w\}_k$, $\sigma \in \mathrm{Gal}(E_{k+1}/F)$ and $\tau
\in \mathrm{Gal}(E_k/F)$. Let $\overline{\sigma}$ be the image of $\sigma$ in
$\mathrm{Gal}(E_k/F)$, and fix $\tilde{\tau} \in \mathrm{Gal}(E_{k+1}/F)$
lifting $\tau$. We have
\begin{align*}
\mathrm{d}(\mathrm{AWES}^1_k(\beta))(\sigma, \tau)(\overline{\sigma} \tau \cdot w) & = \frac{\mathrm{AWES}^1_k(\beta)(\overline{\sigma})(\overline{\sigma} \tau \cdot w) \sigma(\mathrm{AWES}^1_k(\beta)(\tau))(\overline{\sigma} \tau \cdot w)}{\mathrm{AWES}^1_k(\beta)(\overline{\sigma} \tau)(\overline{\sigma} \tau \cdot w)} \\
& = \prod_{n \in \mathrm{Gal}(E_{k+1}/E_k)} \frac{\beta(n \sigma)(n \sigma \cdot \zeta_{k,v}(\tau \cdot w)) }{\beta(n)(n \cdot \zeta_{k,v}(\overline{\sigma} \tau \cdot w))} \\
&\ \ \ \ \ \ \ \ \times \sigma\left(\frac{\beta(n \tilde{\tau})(n \tilde{\tau} \cdot \zeta_{k,v}(w)) }{\beta(n)(n \cdot \zeta_{k,v}(\tau \cdot w))}\right) \\
&\ \ \ \ \ \ \ \ \times \frac{\beta(n)(n \cdot \zeta_{k,v}(\overline{\sigma} \tau \cdot w)) }{\beta(n \sigma \tilde{\tau})(n \sigma \tilde{\tau} \cdot \zeta_{k,v}(w))} \\
& = \prod_{n \in \mathrm{Gal}(E_{k+1}/E_k)} \frac{\sigma\left(\beta(n \tilde{\tau})(n \tilde{\tau} \cdot \zeta_{k,v}(w))\right)}{\beta(n \sigma \tilde{\tau})(n \sigma \tilde{\tau} \cdot \zeta_{k,v}(w))} \\
&\ \ \ \ \ \ \ \ \times \frac{\beta(n \sigma)(n \sigma \cdot \zeta_{k,v}(\tau \cdot w))}{\sigma\left(\beta(n)(n \cdot \zeta_{k,v}(\tau \cdot w)) \right)} \\
& = \prod_{n \in \mathrm{Gal}(E_{k+1}/E_k)} \frac{\sigma\left(\beta(n \tilde{\tau})(n \tilde{\tau} \cdot \zeta_{k,v}(w))\right)}{\beta(\sigma n \tilde{\tau})(\sigma n \tilde{\tau} \cdot \zeta_{k,v}(w))} \\
&\ \ \ \ \ \ \ \ \times \frac{\beta(\sigma n)(\sigma n \cdot \zeta_{k,v}(\tau \cdot w))}{\sigma\left(\beta(n)(n \cdot \zeta_{k,v}(\tau \cdot w)) \right)} \\
& = \prod_{n \in \mathrm{Gal}(E_{k+1}/E_k)} \frac{\mathrm{d} \beta(\sigma, n \tilde{\tau})(\sigma n \tilde{\tau} \cdot \zeta_{k,v}(w))}{\beta(\sigma)(\sigma n \tilde{\tau} \cdot \zeta_{k,v}(w))} \\
&\ \ \ \ \ \ \ \ \times \frac{\beta(\sigma)(\sigma n \cdot \zeta_{k,v}(\tau \cdot w)))}{\mathrm{d} \beta (\sigma, n)(\sigma n \cdot \zeta_{k,v}(\tau \cdot w))} \\
& = \prod_{n \in \mathrm{Gal}(E_{k+1}/E_k)} \frac{\mathrm{d} \beta(\sigma, n \tilde{\tau})(\sigma n \tilde{\tau} \cdot \zeta_{k,v}(w))}{\mathrm{d} \beta (\sigma, n)(\sigma n \cdot \zeta_{k,v}(\tau \cdot w))} \\
& = \mathrm{AWES}^2_k(d\beta)(\sigma, \tau)(\overline{\sigma} \tau \cdot w).
\end{align*}
We have used the fact that for any $u \in \{v\}_{E_{k+1}}$,
\[ \mathrm{card} \left\{ n \in \mathrm{Gal}(E_{k+1}/E_k) \,\middle|\, n
\tilde{\tau} \cdot \zeta_{k,v}(w) = u \right\} = \mathrm{card} \left\{ n \in
\mathrm{Gal}(E_{k+1}/E_k) \,\middle|\, n \cdot \zeta_{k,v}(\tau \cdot w)) = u
\right\} \]
that implies
\[ \prod_{n \in \mathrm{Gal}(E_{k+1}/E_k)} \beta(\sigma)(\sigma n \tilde{\tau}
\cdot \zeta_{k,v}(w)) = \prod_{n \in \mathrm{Gal}(E_{k+1}/E_k)}
\beta(\sigma)(\sigma n \cdot \zeta_{k,v}(\tau \cdot w))). \]
\item
Let $v \in S$ and $w \in \{v\}_{E_k}$. Let $\sigma \in \mathrm{Gal}(E_k/F)$ and
fix $\tilde{\sigma} \in \mathrm{Gal}(E_{k+1}/F)$ lifting $\sigma$.
\begin{align*}
\mathrm{d}(\mathrm{AWES}^0_k(\beta))(\sigma)(\sigma \cdot w) & = \frac{\sigma (\mathrm{AWES}^0_k(\beta))(\sigma \cdot w)}{\mathrm{AWES}^0_k(\beta)(\sigma  \cdot w)} \\
& = \prod_{n \in \mathrm{Gal}(E_{k+1}/E_k)} \frac{\tilde{\sigma} n (\beta(\zeta_{k,v}(w)))}{n(\beta(\zeta_{k,v}(\sigma \cdot w)))} \\
& = \prod_{n \in \mathrm{Gal}(E_{k+1}/E_k)} \frac{n \tilde{\sigma} (\beta(\zeta_{k,v}(w)))}{n(\beta(\zeta_{k,v}(\sigma \cdot w)))} \\
& = \prod_{n \in \mathrm{Gal}(E_{k+1}/E_k)} \frac{\mathrm{d} \beta(n \tilde{\sigma})(n \tilde{\sigma} \cdot \zeta_{k,v}(w)) \times \beta(n \tilde{\sigma} \cdot \zeta_{k,v}(w))}{\mathrm{d} \beta(n)(n \cdot \zeta_{k,v}(\sigma \cdot w)) \times \beta(n \cdot \zeta_{k,v}(\sigma \cdot w))} \\
& = \prod_{n \in \mathrm{Gal}(E_{k+1}/E_k)} \frac{\mathrm{d} \beta(n \tilde{\sigma})(n \tilde{\sigma} \cdot \zeta_{k,v}(w))}{\mathrm{d} \beta(n)(n \cdot \zeta_{k,v}(\sigma \cdot w))} \\
& = \mathrm{AWES}^1_k(\mathrm{d} \beta)(\sigma)(\sigma \cdot w).
\end{align*}
Again we have used the fact that for any $u \in \{v\}_{E_{k+1}}$,
\[ \mathrm{card} \left\{ n \in \mathrm{Gal}(E_{k+1}/E_k) \,\middle|\, n
\tilde{\sigma} \cdot \zeta_{k,v}(w) = u \right\} = \mathrm{card} \left\{ n \in
\mathrm{Gal}(E_{k+1}/E_k) \,\middle|\, n \cdot \zeta_{k,v}(\sigma \cdot w)) = u
\right\} \]
that implies
\[ \prod_{n \in \mathrm{Gal}(E_{k+1}/E_k)} \beta(n \tilde{\sigma} \cdot
\zeta_{k,v}(w)) = \prod_{n \in \mathrm{Gal}(E_{k+1}/E_k)} \beta(n \cdot
\zeta_{k,v}(\sigma \cdot w))). \]
\end{enumerate}
\end{proof}

\begin{coro} \label{c:AWESfactor}
Fix $k \geq 0$, and suppose that $A$ is a $\mathrm{Gal}(E_{k+1} / F)$-module.
\begin{enumerate}
\item
Let $\beta : \mathrm{Gal}(E_{k+1} / F) \rightarrow \mathrm{Maps}(V_{E_{k+1}},
A)$ be such that $\mathrm{AWES}^1_k(\mathrm{d}(\beta))$ factors through
$\mathrm{Gal}(E_k / F)^2$. Then $\mathrm{AWES}^1_k(\beta)$ takes values in
$\mathrm{Maps}\left(V_{E_k}, A^{\mathrm{Gal}(E_{k+1} / E_k)}\right)$.
\item
If $\beta \in Z^1\left( \mathrm{Gal}(E_{k+1} / F), \mathrm{Maps}(V_{E_{k+1}},
A)\right)$ then
\[ \mathrm{AWES}^1_k(\beta) \in Z^1\left(\mathrm{Gal}(E_k / F),
\mathrm{Maps}\left(V_{E_k}, A^{\mathrm{Gal}(E_{k+1} / E_k)}\right)\right) . \]
\end{enumerate}
\end{coro}
\begin{proof}
\begin{enumerate}
\item
Recall that a priori $\mathrm{AWES}^1_k(\beta) : \mathrm{Gal}(E_k / F)
\rightarrow \mathrm{Maps}(V_{E_k}, A)$. By the previous lemma, for all $w \in
V_{E_k}$, $\sigma \in \mathrm{Gal}(E_{k+1} / F)$ and $\tau \in \mathrm{Gal}(E_k
/ F)$, the quotient
\[ \frac{\mathrm{AWES}^1_k(\beta)(\overline{\sigma})(\overline{\sigma} \tau
\cdot w) \times \sigma\left(\mathrm{AWES}^1_k(\beta)(\tau)(\tau \cdot w)\right)
}{ \mathrm{AWES}^1_k(\beta)(\overline{\sigma} \tau)(\overline{\sigma} \tau \cdot
w) } \]
depends on $\sigma$ only via its image $\overline{\sigma} \in
\mathrm{Gal}(E_k/F)$. Taking $\sigma \in \mathrm{Gal}(E_{k+1}/E_k)$ shows that
$\mathrm{AWES}^1_k(\beta)(\tau)(\tau \cdot w)$ is invariant under
$\mathrm{Gal}(E_{k+1}/E_k)$. \item This follows directly from the first point
and a second application of the previous lemma.
\end{enumerate}
\end{proof}

We now establish the analogue of Lemma \ref{l:surjAW} for $\mathrm{AWES}^1_k$
and $\mathrm{AWES}^2_k$.

\begin{lemm} \label{l:surjAWES}
Let $k \geq 0$. Suppose that $A$ is a commutative group.
\begin{enumerate}
\item
The map
\begin{multline*}
\left\{ \beta : \mathrm{Gal}(E_{k+1}/F) \rightarrow \mathrm{Maps}(V_{E_{k+1}}, A) \,\middle|\, \beta(1)=1 \right\} \\
\rightarrow \left\{ \beta : \mathrm{Gal}(E_k/F) \rightarrow \mathrm{Maps}(V_{E_k}, A) \,\middle|\, \beta(1)=1 \right\}
\end{multline*}
induced by $\mathrm{AWES}^1_k$ is surjective.
\item
Let $K \subset \overline{F}$ be a Galois extension of $F$ containing $E_{k+1}$.
The map
\begin{multline*}
\left\{ \alpha : \mathrm{Gal}(K/F) \times \mathrm{Gal}(E_{k+1}/F) \rightarrow \mathrm{Maps}(V_{E_{k+1}}, A) \,\middle|\, \forall \sigma \in \mathrm{Gal}(K/F) , \, \alpha(\sigma,1)=1 \right\} \\
\rightarrow \left\{ \alpha : \mathrm{Gal}(K/F) \times \mathrm{Gal}(E_k/F) \rightarrow \mathrm{Maps}(V_{E_k}, A) \,\middle|\, \forall \sigma \in \mathrm{Gal}(K/F) , \, \alpha(\sigma,1)=1 \right\}
\end{multline*}
induced by $\mathrm{AWES}^2_k$ is surjective.
\end{enumerate}
\end{lemm}
\begin{proof}
As in the proof of Lemma \ref{l:surjAW}, in each case we exhibit a subset of the source such that restricting to this subset yields a bijection.
Choose a section $s : \mathrm{Gal}(E_k / F) \rightarrow \mathrm{Gal}(E_{k+1} / F)$ such that $s(1)=1$.
\begin{enumerate}
\item
Restrict to the set of $\beta : \mathrm{Gal}(E_{k+1}/F) \rightarrow \mathrm{Maps}(V_{E_{k+1}}, A)$ such that for $n \in \mathrm{Gal}(E_{k+1}/E_k)$, $\sigma \in \mathrm{Gal}(E_k / F)$, $v \in V$ and $u \in \{v\}_{E_{k+1}}$, $\beta(n s(\sigma))(n s(\sigma) \cdot u) = 1$ unless $n=1$, $\sigma \neq 1$ and $u$ belongs to the image of $\zeta_{k,v} : \{v\}_{E_k} \rightarrow \{v\}_{E_{k+1}}$.
\item
Restrict to the set of $\alpha : \mathrm{Gal}(K/F) \times \mathrm{Gal}(E_{k+1}/F) \rightarrow \mathrm{Maps}(V_{E_{k+1}}, A)$ such that for $\sigma \in \mathrm{Gal}(K/F)$, $n \in \mathrm{Gal}(E_{k+1}/E_k)$, $\tau \in \mathrm{Gal}(E_k / F)$, $v \in V$ and $u \in \{v\}_{E_{k+1}}$, $\alpha(\sigma, n s(\tau))(n s(\sigma) \cdot u) = 1$ unless $n=1$, $\tau \neq 1$ and $u$ belongs to the image of $\zeta_{k,v} : \{v\}_{E_k} \rightarrow \{v\}_{E_{k+1}}$.
\end{enumerate}
\end{proof}

\subsection{Tate cocycles}
\label{s:Tatecocycles}

Recall that for every $k \geq 0$ the kernel $C(E_k)^1$ of the surjective norm
map $\lVert\cdot\rVert_k : C(E_k) \rightarrow \R_{>0}$ is compact, and that
these norm maps commute with the norm maps for the Galois action
$N_{E_{k+1}/E_k} : C(E_{k+1}) \rightarrow C(E_k)$, that is $\lVert x
\rVert_{k+1} = \lVert N_{E_{k+1}/E_k}(x) \rVert_k$ for all $x \in C(E_{k+1})$.
In this section we will see the fundamental cocycles $\overline{\alpha}_k \in
\Z^2(E_k/F, C(E_k))$ defined in Section \ref{s:globfundcocy} as taking values in
$\mathrm{Maps}(V_{E_k}, C(E_k))$, by seeing elements of $C(E_k)$ as constant
functions $V_{E_k} \rightarrow C(E_k)$.

\begin{lemm} \label{l:constructionTatemod}
There exists a family $\left(\overline{\beta}_k^{(0)} \right)_{k \geq 0}$, where
$\overline{\beta}_k^{(0)} : \mathrm{Gal}(E_k/F) \rightarrow
\mathrm{Maps}(V_{E_k}, C(E_k))$, such that:
\begin{enumerate}
\item For any $k \geq 0$ we have $\overline{\alpha}_k / \overline{\alpha'_k} =
\mathrm{d}\left( \overline{\beta}_k^{(0)} \right)$, where $\overline{\alpha'_k}
:= \alpha'_k \mod E_k^{\times}$.
\item For any $k \geq 0$ we have
\[ \mathrm{AWES}^1_k\left( \overline{\beta}_{k+1}^{(0)} \right) \in
\mathrm{Maps}(\mathrm{Gal}(E_k/F), \mathrm{Maps}(V_{E_k}, C(E_k))). \]
\item For any $k \geq 0$ we have $\left\lVert \mathrm{AWES}^1_k\left(
\overline{\beta}_{k+1}^{(0)} \right) \right\rVert_k = \left\lVert
\overline{\beta}_k^{(0)} \right\rVert_k$, as functions $\mathrm{Gal}(E_k/F)
\times V_{E_k} \rightarrow \R_{>0}$.
\end{enumerate}
\end{lemm}
\begin{proof}
For a given $k$, the existence of $\overline{\beta}_k^{(0)}$ satisfying the
first condition is a consequence of compatibility between local and global
fundamental classes (see \cite{Tatecohotori}). Note that if
$\overline{\beta}_{k+1}^{(0)}$ is such that $\overline{\alpha}_{k+1} /
\overline{\alpha'_{k+1}} = \mathrm{d} \left( \overline{\beta}_{k+1}^{(0)}
\right)$, then by Lemma \ref{l:dAWES}
\begin{equation} \label{e:dAWTate}
\mathrm{d} \left( \mathrm{AWES}^1_k \left( \overline{\beta}_{k+1}^{(0)} \right)
\right) = \mathrm{AWES}^2_k \left( \mathrm{d} \left(
\overline{\beta}_{k+1}^{(0)} \right) \right) =
\mathrm{AWES}^2_k(\overline{\alpha}_{k+1}) /
\overline{\mathrm{AWES}^2_k(\alpha'_{k+1})} = \overline{\alpha}_k /
\overline{\alpha'_k}
\end{equation}
factors through $\mathrm{Gal}(E_k/F)^2$, and by Corollary \ref{c:AWESfactor}
$\mathrm{AWES}^1_k \left( \overline{\beta}_{k+1}^{(0)} \right)$ takes values in
$\mathrm{Maps}(V_{E_k}, C(E_k))$. So the second condition in the lemma is a
consequence of the first one.

Let us start with a family $\left(\overline{\beta}_k^{(0)} \right)_{k \geq 0}$
satisfying the first condition, and show that we can inductively multiply
$\overline{\beta}_k^{(0)}$, $k \geq 1$, by a $1$-coboundary so that the third
condition is also satisfied. By \eqref{e:dAWTate} we know that
\[ \mathrm{AWES}^1_k \left( \overline{\beta}_{k+1}^{(0)} \right) /
\overline{\beta}_k^{(0)} \in Z^1(\mathrm{Gal}(E_k/F), \mathrm{Maps}(V_{E_k},
C(E_k)))  \]
and by vanishing of $H^1(\mathrm{Gal}(E_k/F), \mathrm{Maps}(V_{E_k}, C(E_k)))$
there exists $b_k : V_{E_k} \rightarrow C(E_k)$ such that $\mathrm{AWES}^1_k
\left( \overline{\beta}_{k+1}^{(0)} \right) / \overline{\beta}_k^{(0)} =
\mathrm{d}( b_k )$.  Choose $\tilde{b}_k : V_{E_{k+1}} \rightarrow C(E_{k+1})$
such that for any $\tau \in \mathrm{Gal}(E_k/F)$, $\left\lVert
\tilde{b}_k(s_{k,v}(\tau) \cdot \dot{v}_{k+1}) \right\rVert_{k+1} = \lVert
b_k(\tau \cdot \dot{v}_k) \rVert_k$.  Equivalently, $\left\lVert
\mathrm{AWES}^0_k (\tilde{b}_k) \right\rVert_k = \lVert b_k \rVert_k$.
Substituting $\overline{\beta}_{k+1}^{(0)} \times \mathrm{d}(\tilde{b}_k)$ for
$\overline{\beta}_{k+1}^{(0)}$, the third condition is satisfied.
\end{proof}

\begin{theo} \label{t:constructionTate}
There exists a family $\left(\beta_k \right)_{k \geq 0}$ with $\beta_k \in C^1(
E_k/F, \mathrm{Maps}(V_{E_k}, I(E_k, S_k)) )$ such that
\begin{enumerate}
\item For any $k \geq 0$ we have $\overline{\alpha}_k / \overline{\alpha'_k} =
\mathrm{d}\left( \overline{\beta_k} \right)$.
\item For any $k \geq 0$ we have $\mathrm{AWES}^1_k(\beta_{k+1}) = \beta_k$.
\end{enumerate}
Therefore, the family $\left( \alpha_k \right)_{k \geq 0}$ defined by $\alpha_k
= \alpha'_k \times \mathrm{d}(\beta_k)$ is a family of Tate cocycles, compatible
in the sense that $\mathrm{AWES}^2_k(\alpha_{k+1}) = \alpha_k$ for all $k \geq
0$.
\end{theo}
\begin{proof}
Let $\left(\overline{\beta}_k^{(0)} \right)_{k \geq 0}$ be a family as in the
previous Lemma. The space
\[ X_k := \left\{ \overline{\beta}_k : \mathrm{Gal}(E_k/F) \rightarrow
\mathrm{Maps}(V_{E_k}, C(E_k)) \,\middle|\, \lVert \overline{\beta}_k \rVert_k =
\lVert \overline{\beta}_k^{(0)} \rVert_k \text{ and } \overline{\alpha}_k /
\overline{\alpha'_k} = \mathrm{d}\left( \overline{\beta}_k \right) \right\} \]
is compact for the topology induced by the product topology on
\[ \mathrm{Maps}(\mathrm{Gal}(E_k/F), \mathrm{Maps}(V_{E_k}, C(E_k))). \]
Moreover $\overline{\beta}_k^{(0)} \in X_k$. The inverse system $\left( (X_k)_{k
\geq 0}, \left(\mathrm{AWES}^1_k : X_{k+1} \rightarrow X_k \right)_{k \geq 0}
\right)$ consists of non-empty compact topological spaces and continuous maps
between them, therefore $\varprojlim_{k \geq 0} X_k \neq \emptyset$. Choose
$\left(\overline{\beta}_k \right)_k \in \varprojlim X_k$. Such a family
satisfies the two conditions in the proposition, but note that
$\overline{\beta}_k$ takes values in $C(E_k)$.

Let us inductively choose lifts $\beta_k$ of $\overline{\beta}_k$ such that
$\mathrm{AWES}^1_k(\beta_{k+1}) = \beta_k$. Note that this imposes $\beta_k(1) =
1$ for all $k$. Choose any $\beta_0$ lifting $\overline{\beta}_0$ such that
$\beta_0(1)=1$. Suppose that $\beta_k$ is given. If $\beta$ is any lift of
$\overline{\beta}_{k+1}$ such that $\beta(1)=1$, then $\beta_k /
\mathrm{AWES}^1_k(\beta)$ is a mapping $\mathrm{Gal}(E_{k}/F)\rightarrow
\mathrm{Maps}(V_{E_k}, \mathcal{O}(E_{k+1}, S_{k+1}))$. By Lemma
\ref{l:surjAWES}, there exists $\nu : \mathrm{Gal}(E_{k+1}/F) \rightarrow
\mathrm{Maps}(V_{E_{k+1}}, \mathcal{O}(E_{k+1}, S_{k+1}))$ such that $\nu(1)=1$
and $\beta_k / \mathrm{AWES}^1_k(\beta) = \mathrm{AWES}^1_k(\nu)$, and we let
$\beta_{k+1} = \beta \times \nu$.
\end{proof}

\begin{rema}
This result solves two problems at once:
\begin{enumerate}
\item Constructing a family of Tate cocycles $(\alpha_k)_{k \geq 0}$ compatible
with respect to $\mathrm{AWES}^2_k$, which will be useful to compare
(generalized) Tate-Nakayama isomorphisms in the tower $(E_k)_{k \geq 0}$, by
taking cup-products (Lemma \ref{l:cupproductglobal} and Proposition
\ref{p:towerglobal}).
\item Constructing a family $(\beta_k)_{k \geq 0}$ compatible with respect to
$\mathrm{AWES}^1_k$ and realizing local-global compatibility, which will be
useful to compare local and global (generalized) Tate-Nakayama isomorphisms
(Lemmas \ref{l:cupproductlocalglobal} and \ref{l:cupproductdeltabeta} and
Propositions \ref{p:klocalglobal} and \ref{p:towerlocalglobal}).
\end{enumerate}
The proof suggests that it is not possible to solve the first problem
separately from the second.  One can show that if families $(\alpha_{k,v})_{k
\geq 0, v \in V}$, $(R_{k,v})_{k \geq 0, v \in V}$ and
$(\overline{\alpha}_k)_{k \geq 0}$ as above are fixed, then
$(\overline{\beta_k})_{k \geq 0}$ is determined up to
\[ B^1\left(\mathrm{Gal}(\overline{F}/F),\ \varprojlim_{k \geq 0} C(E_k)^0
\right) \]
where $C(E_k)^0$ is the connected component of $1$ in $C(E_k)$, i.e.\ the
closure of $(\R \otimes_{\Q} E_k)^{\times,0}$ in $C(E_k)$, where $(\R
\otimes_{\Q} E_k)^{\times, 0}$ is the connected component of $1$ in $(\R
\otimes_{\Q} E_k)^{\times}$.
\end{rema}

Note that while $\alpha_{k,v}$, $\alpha_k$ and $R_{k,v}$ can simply be chosen
sequentially as $k$ grows, the existence of a family $(\beta_k)_{k \geq 0}$ in
Theorem \ref{t:constructionTate} follows from a compacity argument. Let us give
an alternative, constructive but more intricate argument for the existence of
$(\beta_k)_{k \geq 0}$. For simplicity we assume that for any $k \geq 0$,
$E_{k+1}$ contains the narrow Hilbert class field of $E_k$, i.e.\
$N_{E_{k+1}/E_k}(C(E_{k+1}))$ is contained in the image of $(\R \otimes_{\Q}
E_k)^{\times, 0} \times \widehat{\mathcal{O}(E_k)}^{\times}$ in $C(E_k)$. This
can be achieved by discarding some of the $E_k$'s.  Choose
$\overline{\beta}_1^{(0)}$ such that $\mathrm{d} \left(
\overline{\beta}_1^{(0)} \right) = \overline{\alpha}_1 / \overline{\alpha'_1}$.
Note that $\overline{\beta}_0^{(1)} :=
\mathrm{AWES}_0^1(\overline{\beta}_1^{(0)}) = 1$.  For good measure let
$\beta_0^{(1)} = 1$ and $\alpha_0 = 1$.
We now proceed to inductively construct $\overline{\beta}_{k+1}^{(0)}$,
$\beta_k^{(1)}$ and $\zeta_{k-1}$ for $k \geq 1$, satisfying the following
properties.
\begin{enumerate}
\item $\overline{\beta}_{k+1}^{(0)} : \mathrm{Gal}(E_{k+1}/F) \rightarrow
\mathrm{Maps}(V_{E_{k+1}}, C(E_{k+1}))$ is such that $\overline{\alpha'_{k+1}}
\times \mathrm{d}\! \left( \overline{\beta}_{k+1}^{(0)} \right) =
\overline{\alpha}_{k+1}$.
\item $\beta_k^{(1)} : \mathrm{Gal}(E_k/F) \rightarrow \mathrm{Maps}(V_{E_k},
I(E_k, S_k))$ is a lift of
$\mathrm{AWES}^1_k\left(\overline{\beta}_{k+1}^{(0)}\right)$ such that
$\beta_k^{(1)}(1) = 1$.
\item $\zeta_k \in \mathrm{Maps}(V_{E_k}, \widehat{\mathcal{O}(E_k)}^{\times})$
is such that $\mathrm{AWES}^1_{k-1}(\beta_k^{(1)}) = \beta_{k-1}^{(1)}
\mathrm{d}(\zeta_{k-1})$.
\end{enumerate}
Let $k \geq 0$, assume that $\overline{\beta}_{k+1}^{(0)}$ and $\beta_k^{(1)}$
are constructed.  First choose any $\overline{\beta}_{k+2}^{(0)} :
\mathrm{Gal}(E_{k+2}/F) \rightarrow \mathrm{Maps}(V_{E_{k+2}}, C(E_{k+2}))$
such that $\overline{\alpha'_{k+2}} \times \mathrm{d} \left(
\overline{\beta}_{k+2}^{(0)} \right) = \overline{\alpha}_{k+2}$.  As we saw in
the proof of Lemma \ref{l:constructionTatemod}, there exists
$\overline{z}_{k+1} \in \mathrm{Maps}(V_{E_{k+1}}, C(E_{k+1}))$ such that
$\mathrm{AWES}^1_{k+1}(\overline{\beta}_{k+2}^{(0)}) =
\overline{\beta}_{k+1}^{(0)} \times \mathrm{d}(\overline{z}_{i+1})$.  Applying
$\mathrm{AWES}^1_k$, we get 
\[ \mathrm{AWES}^1_k \circ \mathrm{AWES}^1_{k+1} \left(
\overline{\beta}_{k+2}^{(0)} \right) = \mathrm{AWES}^1_k\left(
\overline{\beta}_{k+1}^{(0)} \right) \times \mathrm{d}\left(
\mathrm{AWES}^0_k\left( \overline{z}_{k+1} \right) \right) \]
and we would like to let $\zeta_k \in \mathrm{Maps}(V_{E_k}, (\R \otimes_{\Q}
E_k)^{\times, 0} \times \widehat{\mathcal{O}(E_k)}^{\times})$ be a lift of
$\mathrm{AWES}^0_k\left( \overline{z}_{k+1} \right)$, which exists thanks to
the hypothesis that $E_{k+1}$ contains the narrow Hilbert class field of $E_k$.
This is not quite right, since we want $\zeta_k \in \mathrm{Maps}(V_{E_k},
\widehat{\mathcal{O}(E_k)}^{\times})$.  By surjectivity of
\[ \mathrm{AWES}^0_k \circ \mathrm{AWES}^0_{k+1} : \mathrm{Maps}(V_{E_{k+2}},
(\R \otimes_{\Q} E_{k+2})^{\times, 0}) \rightarrow \mathrm{Maps}(V_{E_k}, (\R
\otimes_{\Q} E_k)^{\times, 0}) \]
we see that up to dividing $\overline{\beta}_{k+2}^{(0)}$ by an element of
$B^1(\mathrm{Gal}(E_{k+2}/F), \mathrm{Maps}(V_{E_{k+2}}, (\R \otimes_{\Q}
E_{k+2})^{\times, 0}))$, we can find $\zeta_k \in \mathrm{Maps}(V_{E_k},
\widehat{\mathcal{O}(E_k)}^{\times})$.  Now let $\beta_k^{(2)} = \beta_k^{(1)}
\times \mathrm{d}(\zeta_k)$, and as we saw in the proof of Theorem
\ref{t:constructionTate}, there exists $\beta_{k+1}^{(1)} :
\mathrm{Gal}(E_{k+1}/F) \rightarrow \mathrm{Maps}(V_{E_{k+1}}, I(E_{k+1},
S_{k+1}))$ a lift of
$\mathrm{AWES}^1_{k+1}\left(\overline{\beta}_{k+2}^{(0)}\right)$ such that
$\beta_{k+1}^{(1)}(1) = 1$ and $\mathrm{AWES}^1_k(\beta_{k+1}^{(1)}) =
\beta_k^{(2)}$.  This concludes the construction of $\left(
\overline{\beta}_{k+2}^{(0)}, \beta_{k+1}^{(1)}, \zeta_k \right)$.

Define inductively $\beta_k^{(i+1)} = \mathrm{AWES}^1_k \left(
\beta_{k+1}^{(i)} \right)$ for $i \geq 0$. Then for all $i > k \geq 0$, we
have
\[ \beta_k^{(i+2-k)} = \beta_k^{(i+1-k)} \times \mathrm{d}\left(
\mathrm{AWES}^0_k \circ \dots \circ \mathrm{AWES}^0_{i-1}(\zeta_i) \right) \]
and since $\mathrm{AWES}^0_k \circ \dots \circ \mathrm{AWES}^0_{i-1}(\zeta_i)
\in \mathrm{Maps}\left(V_{E_k}, N_{E_i/E_k} \left(
\widehat{\mathcal{O}(E_i)}^{\times} \right) \right)$, by the existence theorem
in local class field theory and Krasner's lemma the sequences
$(\beta_k^{(i)})_{i>0}$ converge and we can define $\beta_k = \lim_{i
\rightarrow + \infty} \beta_k^{(i)}$.

\section{Generalized Tate-Nakayama morphisms}
\label{s:genTN}

In this section we will construct $N$-th roots of the cochains
$(\alpha_{k,v})_{v \in V}$, $\alpha'_k$, $\beta_k$ and $\alpha_k$ for all $N
\geq 1$ and $k \geq 0$.  This is necessary to establish the global analogue of
\cite[\S 4.5]{Kalri}, i.e.\ to make explicit the morphism $\iota_{\dot{V}}$ of
\cite[Theorem 3.7.3]{Kalgri} for the tower $(E_k)_{k \geq 0}$, and to study the
localization map \cite[(3.19)]{Kalgri}.

\subsection{Choice of $N$-th roots}

\begin{prop} \label{p:Nthlocal}
For any $v \in V$, there exists a family $\left(\sqrt[N]{\alpha_{k,v}}
\right)_{N \geq 1, k \geq 0}$ where $\sqrt[N]{\alpha_{k,v}} :
\mathrm{Gal}(E_{k,\dot{v}}/F_v)^2 \rightarrow \overline{F_v}^{\times}$ such that
\begin{enumerate}
\item for all $k \geq 0$, $\sqrt[1]{\alpha_{k,v}} = \alpha_{k,v}$,
\item for all $k \geq 0$ and $N,N' \geq 1$ such that $N$ divides $N'$,
 $\sqrt[N']{\alpha_{k,v}}^{N'/N} = \sqrt[N]{\alpha_{k,v}}$,
\item for all $k \geq 0$ and $N \geq 1$,
 $\mathrm{AW}^2_{k,v}(\sqrt[N]{\alpha_{k+1,v}}) = \sqrt[N]{\alpha_{k,v}}$.
\end{enumerate}
\end{prop}
\begin{proof}
Using Bézout identities, we see that it is enough to construct families
$\left(\sqrt[\ell^m]{\alpha_{k,v}} \right)_{m \geq 0, k \geq 0}$ for all primes
$\ell$. So fix a prime number $\ell$. For a fixed $k \geq 0$, there exists a
family $\left(\sqrt[\ell^m]{\alpha_{k,v}} \right)_{m \geq 0}$ satisfying the
first two conditions in the proposition, and such that for all $m \geq 0$ and
$\sigma \in \mathrm{Gal}(E_{k,\dot{v}}/F_v)$,
$\sqrt[\ell^m]{\alpha_{k,v}}(\sigma, 1)=1$. If we choose two such families for
$k$ and $k+1$, the last condition might not be satisfied, i.e.\ 
\[ \frac{ \mathrm{AW}^2_{k,v}(\sqrt[\ell^m]{\alpha_{k+1,v}}) }{
	\sqrt[\ell^m]{\alpha_{k,v}} } : \mathrm{Gal}(E_{k+1,\dot{v}}/F_v) \times
\mathrm{Gal}(E_{k,\dot{v}}/F_v) \rightarrow \mu_{\ell^m} \]
could be non-trivial.
Thanks to the fact that $\mathrm{AW}^2_{k,v}(\alpha_{k+1,v}) = \alpha_{k,v}$,
for a fixed $m \geq 0$ Lemma \ref{l:surjAW} applied with $A=\mu_{\ell^m}$
implies that $\sqrt[\ell^m]{\alpha_{k+1,v}}$ can be chosen so that
$\mathrm{AW}^2_{k,v}(\sqrt[\ell^m]{\alpha_{k+1,v}}) =
\sqrt[\ell^m]{\alpha_{k,v}}$. The proof of Lemma \ref{l:surjAW} shows that this
can be done uniformly in $m \geq 0$, by using the same section $s :
\mathrm{Gal}(E_{k,\dot{v}}/F_v) \rightarrow \mathrm{Gal}(E_{k+1,\dot{v}}/F_v)$.
\end{proof}

Fix such a family for each $v \in V$. Note that even if $v(N)=0$ and
$E_{k,\dot{v}} / F_v$ is unramified, $\sqrt[N]{\alpha_{k,v}}$ \emph{cannot} be
chosen to take values in an unramified extension of $F_v$.

Recall the embedding $j_{k, v} : E_{k,\dot{v}}^{\times} \hookrightarrow I(E_k)$.
We now want to extend to $j_{k,v} : \overline{F_v}^{\times} \hookrightarrow
I(\overline{F})$. For $x \in \overline{F_v}^{\times}$, there exists $i \geq 0$
such that $x \in E_{k+i, \dot{v}}^{\times}$. Define \[ j_{k,v}(x) = \prod_{r \in
R_{k+1,v} \dots R_{k+i,v}} r(j_{k+i,v}(x)) \] which does not depend on the
choice of a big enough $i$. These extended embeddings $j_{k,v}$ also satisfy a
compatibility formula: for any $x \in \overline{F_v}^{\times}$ we have
\begin{equation} \label{e:compatj}
j_{k,v}(x) = \prod_{r \in R_{k+1,v}} r \left( j_{k+1,v}(x) \right).
\end{equation}
For $N \geq 1$ define $\sqrt[N]{\alpha'_k} : \mathrm{Gal}(E_k/F)^2 \rightarrow
\mathrm{Maps}(V_{E_k}, I(\overline{F}))$ by
\[ \sqrt[N]{\alpha'_k}(r_1 \sigma r_2^{-1}, r_2 \tau r_3^{-1})(r_1 \cdot
\dot{v}_k) = r_1 \left( j_{k, v} (\sqrt[N]{\alpha_{k,v}}(\sigma, \tau)) \right)
\]
for $r_1,r_2,r_3 \in R'_{k,v}$ and $\sigma, \tau \in \mathrm{Gal}(E_{k, \dot{v}}
/ F_v)$. Obviously $\sqrt[1]{\alpha'_k} = \alpha'_k$ and whenever $N$ divides
$N'$, $\sqrt[N']{\alpha'_k}^{N'/N} = \sqrt[N]{\alpha'_k}$. By the same proof as
Lemma \ref{l:AWadelic}, thanks to \eqref{e:compatj}, we have
\[ \mathrm{AWES}^2_k\left(\sqrt[N]{\alpha'_{k+1}}\right) = \sqrt[N]{\alpha'_k}. \]
Note that for any $k \geq 0$ and $v \in V$, there exists $i \geq 0$ such that
$\sqrt[N]{\alpha_{k,v}}$ takes values in $E_{k+i, \dot{v}}^{\times}$ and so for
any $w \in \{v\}_{E_k}$, $\sqrt[N]{\alpha'_k}(-,-)(w)$ takes values in
$\A_{E_{k+i}}^{\times}$.

We now want to construct $N$-th roots $\sqrt[N]{\alpha_k}$ of the Tate classes
$\alpha_k$ constructed in Section \ref{s:Tatecocycles}.  For this it is
necessary to take $N$-th roots of id\`eles, which may not be id\`eles.  For $S'$
a finite subset of $V$, let $\mathcal{I}(F, S') \subset \prod_{v \in V} \left(
\overline{F} \otimes_F F_v \right)^{\times}$ be the set of families $(x_v)_v$
such that for any $v \not\in S'$, there exists a finite Galois extension $K/F$
unramified above $v$ such that $x_v \in \left( \mathcal{O}_K
\otimes_{\mathcal{O}_F} \mathcal{O}_{F_v} \right)^{\times} = \prod_{w |v}
\mathcal{O}_{K_w}^{\times}$.  Let $\mathcal{I}(F) = \varinjlim_{S'}
\mathcal{I}(F, S')$.  Recall (Theorem \ref{t:constructionTate}) that $\alpha_k :
\mathrm{Gal}(E_k/F)^2 \rightarrow \mathrm{Maps}(V_{E_k}, I(E_k))$ has the
following properties:
\begin{itemize}
\item for all $\sigma, \tau \in \mathrm{Gal}(E_k/F)$ and $w_1,w_2 \in V_{E_k}$,
$\alpha_k(\sigma, \tau)(w_1) / \alpha_k(\sigma, \tau)(w_2) \in E_k^{\times}$,
\item for all $\sigma, \tau \in \mathrm{Gal}(E_k/F)$, $v \in V$ and $w \in
\{v\}_{E_k}$, $\alpha_k(\sigma, \tau)(w) \in I(E_k)$ is a unit away from
$S_{k,E_k} \cup \{v\}_{E_k}$.
\end{itemize}
It is crucial for $\sqrt[N]{\alpha_k}$ to enjoy similar properties.
\begin{prop} \label{p:NthTate}
There exists a family $\left(\sqrt[N]{\alpha_k} \right)_{N \geq 1, k \geq 0}$
where $\sqrt[N]{\alpha_k} : \mathrm{Gal}(E_k/F)^2 \rightarrow \mathcal{I}(F)$
such that
\begin{enumerate}
\item for all $k \geq 0$, $\sqrt[1]{\alpha_k} = \alpha_k$,
\item for all $k \geq 0$ and $N,N' \geq 1$ such that $N$ divides $N'$, $\sqrt[N']{\alpha_k}^{N'/N} = \sqrt[N]{\alpha_k}$,
\item for all $k \geq 0$ and $N \geq 1$, $\mathrm{AWES}^2_k(\sqrt[N]{\alpha_{k+1}}) = \sqrt[N]{\alpha_k}$,
\item for all $k \geq 0$, $N \geq 1$, $\sigma, \tau \in \mathrm{Gal}(E_k/F)$ and $w_1,w_2 \in V_{E_k}$, $\sqrt[N]{\alpha_k}(\sigma, \tau)(w_1) / \sqrt[N]{\alpha_k}(\sigma, \tau)(w_2) \in \overline{F}^{\times}$,
\item for all $k \geq 0$, $N \geq 1$, $\sigma, \tau \in \mathrm{Gal}(E_k/F)$, $v \in V$ and $w \in \{v\}_{E_k}$, $\sqrt[N]{\alpha_k}(\sigma, \tau)(w) \in \mathcal{I}(F, S_k \cup \{v\} \cup N)$.
\end{enumerate}
\end{prop}
\begin{proof}
It will be convenient to fix an archimedean place $u$ of $F$, so that in particular $\dot{u}_k \in S_{k,E_k}$ for all $k \geq 0$.
As in the proof of Proposition \ref{p:Nthlocal} it is enough to restrict to powers of a fixed prime $\ell$.

First we show how to construct a family $\left(\sqrt[\ell^m]{\alpha_k}\right)_{m \geq 0}$ for a fixed $k \geq 0$.
For $m \geq 0$ and $\sigma, \tau \in \mathrm{Gal}(E_k/F)$ choose roots $\sqrt[\ell^m]{\alpha_k}(\sigma, \tau)(\sigma \tau \cdot \dot{u}_k) \in \mathcal{I}(F, S_k \cup \ell)$ such that $\sqrt[\ell^{m+1}]{\alpha_k}(\sigma, \tau)(\sigma \tau \cdot \dot{u}_k)^{\ell} = \sqrt[\ell^m]{\alpha_k}(\sigma, \tau)(\sigma \tau \cdot \dot{u}_k)$.
We can further impose that $\sqrt[\ell^m]{\alpha_k}(\sigma, 1)(\sigma \cdot \dot{u}_k) = 1$ for all $\sigma \in \mathrm{Gal}(E_k / F)$.
Then choose, for $\sigma, \tau \in \mathrm{Gal}(E_k/F)$, $v \in V$ and $w \in \{v\}_{E_k} \smallsetminus \{ \sigma \tau \cdot \dot{u}_k \}$, $\ell^m$-th roots of $\alpha_k(\sigma, \tau)(w) / \alpha_k(\sigma, \tau)(\sigma \tau \cdot \dot{u}_k)$ in $(\overline{F}_{S_k \cup \{v\} \cup \ell})^{\times}$, and define $\sqrt[\ell^m]{\alpha_k}(\sigma, \tau)(w)$ as the products of these $\ell^m$-th roots with $\sqrt[\ell^m]{\alpha_k}(\sigma, \tau)(\sigma \tau \cdot \dot{u}_k)$.
This can be done compatibly as $m$ varies.
Again we can impose $\sqrt[\ell^m]{\alpha_k}(\sigma, 1)(w) = 1$ for all $\sigma \in \mathrm{Gal}(E_k / F)$.
We obtain a family $\left(\sqrt[\ell^m]{\alpha_k}\right)_{m \geq 0}$ satisfying all conditions in the proposition except for the third one.

The fact that these choices can be made compatibly as $k$ varies, i.e.\ in such a way that the third condition is also satisfied, can be proved as in Proposition \ref{p:Nthlocal}, using the fact that $\mathrm{AWES}^2_k(\alpha_{k+1}) = \alpha_k$ and (the proof of) Lemma \ref{l:surjAWES} instead of Lemma \ref{l:surjAW}.
\end{proof}
Fix a family $\left(\sqrt[N]{\alpha_k} \right)_{N \geq 1, k \geq 0}$ as in the proposition.
We want to compare $\sqrt[N]{\alpha_k'}$ and $\sqrt[N]{\alpha_k}$.
Recall (Theorem \ref{t:constructionTate}) that $\alpha_k = \alpha_k' \mathrm{d}(\beta_k)$, where $\beta_k : \mathrm{Gal}(E_k / F) \rightarrow \mathrm{Maps}(V_{E_k}, I(E_k, S_k))$.
\begin{prop} \label{p:Nthbeta}
There exists a family $\left(\sqrt[N]{\beta_k} \right)_{N \geq 1, k \geq 0}$
where $\sqrt[N]{\beta_k} : \mathrm{Gal}(E_k/F) \rightarrow \mathrm{Maps}
\left(V_{E_k}, \mathcal{I}(F, S_k \cup N) \right)$ such that
\begin{enumerate}
\item for all $k \geq 0$, $\sqrt[1]{\beta_k} = \beta_k$,
\item for all $k \geq 0$ and $N,N' \geq 1$ such that $N$ divides $N'$, $\sqrt[N']{\beta_k}^{N'/N} = \sqrt[N]{\beta_k}$,
\item for all $k \geq 0$ and $N \geq 1$, $\mathrm{AWES}^1_k(\sqrt[N]{\beta_{k+1}}) = \sqrt[N]{\beta_k}$.
\end{enumerate}
\end{prop}
\begin{proof}
Only the third condition is non-trivial, and the proof proceeds as in Propositions \ref{p:Nthlocal} and \ref{p:NthTate}.
\end{proof}
Fix a family $\left(\sqrt[N]{\beta_k} \right)_{N \geq 1, k \geq 0}$ as in the proposition.
Note that $\mathrm{d}\left(\!\!\sqrt[N]{\beta_k} \right) : \mathrm{Gal}(\overline{F}_{S_k \cup N} / F) \times \mathrm{Gal}(E_k/F) \rightarrow \mathrm{Maps}\left(V_{E_k}, \mathcal{I}(F, S_k \cup N) \right)$.

\begin{defi} \label{d:delta}
For $k \geq 0$ and $N \geq 1$, let
\[ \delta_k(N) = \frac{ \sqrt[N]{\alpha_k} }{ \sqrt[N]{\alpha_k'} \  \mathrm{d}\left(\!\!\sqrt[N]{\beta_k} \right)} : \mathrm{Gal}(\overline{F}_{S_k \cup N} / F) \times \mathrm{Gal}(E_k/F) \rightarrow \mathrm{Maps}\left(V_{E_k}, \mathcal{I}(F)[N] \right) \]
where $\mathcal{I}(F)[N]$ is the subgroup of $N$-torsion elements in $\mathcal{I}(F)$.
\end{defi}
By construction, we have:
\begin{itemize}
\item For all $k \geq 0$, $N \geq 1$ and $w \in V_{E_k}$, there exists a finite Galois extension $K$ of $F$ containing $E_k$ such that $\delta_k(N)$ factors through $\mathrm{Gal}(K/F) \times \mathrm{Gal}(E_k/F)$.
\item
For all $k \geq 0$, $N \geq 1$, $\sigma \in \mathrm{Gal}(\overline{F}_{S_k \cup N} / F)$, $\tau \in \mathrm{Gal}(E_k / F)$, $v \in V$ and $w \in \{v\}_{E_k}$,
\[ \delta_k(N)(\sigma, \tau)(w) \in \mathcal{I}(F, S_k \cup \{v\} \cup N)[N]. \]
\item
For all $k \geq 0$ and $N, N' \geq 1$ such that $N$ divides $N'$, we have $\delta_k(N')^{N'/N} = \delta_k(N)$.
\item
For all $k \geq 0$ and $N \geq 1$, $\mathrm{AWES}^2_k(\delta_{k+1}(N)) = \delta_k(N)$.
\end{itemize}

\subsection{Generalized Tate-Nakayama morphism for the global tower}

Using the compatible families of cochains constructed in the previous section,
we now want to recast several of Kaletha's constructions in cohomology, but for
actual cochains. First we describe the extension $P_{\dot{V}} \rightarrow
\mathcal{E}_{\dot{V}} \rightarrow \mathrm{Gal}(\overline{F} / F)$ explicitly as
a projective limit of extensions $P(E_k, \dot{S}'_{E_k}, N) \rightarrow
\mathcal{E}_k(S',N) \rightarrow \mathrm{Gal}(\overline{F}_{S' \cup N}/F)$
constructed using $\sqrt[N]{\alpha_k}$, for varying $k, S', N$.  This is the
global analogue of \cite[\S 4.5]{Kalri}. Then we make explicit the morphism
$\iota_{\dot{V}}$ of \cite[Theorem 3.7.3]{Kalgri} using this projective limit.
To avoid repeating similar calculations we deduce these two constructions from
Lemma \ref{l:cupproductglobal} below.

Let us recall notation from Kaletha's paper.  Suppose that $S' \subset V$.  If
$M$ is an abelian group, define $!_k : M[S'_{E_k}] \rightarrow M[S'_{E_{k+1}}]$
by $!_k(\Lambda)(\zeta_{k,v}(w)) = \Lambda(w)$ for $v \in S'$ and $w \in
\{v\}_{E_k}$, and $!_k(\Lambda)(u) = 0$ if $u \not\in \left\{ \zeta_{k,v}(w)
\,\middle|\, v \in S',\, w \in \{v\}_{E_k} \right\}$.

\begin{lemm} \label{l:cupproductglobal}
Let $T$ be a torus defined over $F$.  Denote $Y = X_*(T)$. Let $k$ be big enough
so that $E_k$ splits $T$. Let $N \geq 1$ be an integer. Let $S'$ be a finite
subset of $V$ containing $S_{k+1}$. Let $\Lambda \in Y[S'_{E_k}]_0^{N_{E_k/F}}$.
Then we have an equality of maps $\mathrm{Gal}(\overline{F}_{S' \cup N} / F)
\rightarrow T(\mathcal{O}_{S' \cup N})$:
\[ \sqrt[N]{\alpha_k} \underset{E_k / F}{\sqcup} \Lambda  = \sqrt[N]{\alpha_{k+1}} \underset{E_{k+1} / F}{\sqcup} !_k(\Lambda). \]
\end{lemm}
Note that if $S_k \subset S'' \subset S'$ and the support of $\Lambda$ is
contained in $S''_{E_k}$, then the left hand side is inflated from a map
$\mathrm{Gal}(\overline{F}_{S'' \cup N} / F) \rightarrow T(\mathcal{O}_{S'' \cup
N})$.
\begin{proof}
For $\sigma \in \mathrm{Gal}(\overline{F}_{S' \cup N} / F)$ we have
\begin{align*}
\left( \sqrt[N]{\alpha_k} \sqcup_{E_k / F} \Lambda_k \right)(\sigma) =& \prod_{\tau \in \mathrm{Gal}(E_k / F)} \sqrt[N]{\alpha_k}(\sigma, \tau) \otimes \sigma \tau ( \Lambda ) \\
=& \prod_{\tau \in \mathrm{Gal}(E_k / F)} \prod_{w \in S'_{E_k}} \sqrt[N]{\alpha_k}(\sigma, \tau)(w) \otimes \sigma \tau (\Lambda)(w).
\end{align*}
Note that in this last expression, the tensor products land in $T(\mathcal{I}(F,
S' \cup N))$, but the product over $S'_{E_k}$ belongs to $T(\mathcal{O}_{S' \cup
N})$ because $\sum_{w \in S'_{E_k}} \Lambda(w) = 0$, using the third condition
in Proposition \ref{p:NthTate}. Compare with the pairing \cite[(3.24)]{Kalgri}.
Recall that $\sqrt[N]{\alpha_k} = \mathrm{AWES}^2_k(\sqrt[N]{\alpha_{k+1}})$,
so that
\begin{multline*}
\left( \sqrt[N]{\alpha_{k+1}} \underset{E_{k+1} / F}{\sqcup} !_k(\Lambda) \right)(\sigma) / \left( \sqrt[N]{\alpha_k} \underset{E_k / F}{\sqcup} \Lambda \right)(\sigma) = \\
\prod_{\tau \in \mathrm{Gal}(E_k / F)} \prod_{v \in S'} \prod_{w \in \{v\}_{E_k}} \prod_{n \in \mathrm{Gal}(E_{k+1} / E_k)} \sqrt[N]{\alpha_{k+1}}(\sigma, n)(\sigma n \cdot \zeta_{k,v}(\tau \cdot w)) \otimes \sigma \tau \left( \Lambda(w) \right).
\end{multline*}
In this expression we use the change of variable $u = \tau \cdot w$ to get
\[ \prod_{\substack{v \in S' \\ n \in \mathrm{Gal}(E_{k+1} / E_k)}} \prod_{u \in \{v\}_{E_k}} \sqrt[N]{\alpha_{k+1}}(\sigma, n)(\sigma n \cdot \zeta_{k,v}(u)) \otimes \sigma \left( \sum_{\tau \in \mathrm{Gal}(E_k / F)} \tau \left(\Lambda(\tau^{-1} \cdot u)\right) \right) \]
and the sum vanishes since $N_{E_k / F}(\Lambda)=0$.
\end{proof}

Let $k \geq 0$ and $N \geq 1$, and let $S'$ be a finite subset of $V$ containing
$S_k$. Recall the finite sub-$\mathrm{Gal}(E_k / F)$-module $M(E_k,
\dot{S}'_{E_k}, N)$ of $\mathrm{Maps}(\mathrm{Gal}(E_k / F) \times S'_{E_k},
\frac{1}{N} \Z / \Z)$ defined in \cite[\S 3.3]{Kalgri}, and the finite
diagonalizable group $P(E_k, \dot{S}'_{E_k}, N)$ such that $X^*(P(E_k,
\dot{S}'_{E_k}, N)) = M(E_k, \dot{S}'_{E_k}, N)$. For any finite algebraic group
$Z$ over $F$ such that $\exp(Z) | N$ and the Galois action on $A := X^*(Z)$
factors through $\mathrm{Gal}(E_k/F)$, we have an identification $\Psi(E_k,S',N)
: \mathrm{Hom}(P(E_k, \dot{S}'_{E_k}, N), Z) \simeq A^{\vee}[ \dot{S}'_{E_k}
]_0^{N_{E_k / F}}$ (see \cite[Lemma 3.3.2]{Kalgri}). Recall also the
$2$-cocycle $\xi_k \in Z^2(\mathrm{Gal}(\overline{F}_{S' \cup N} / F), P(E_k,
\dot{S}'_{E_k}, N))$ from \cite[(3.5)]{Kalgri}, defined using an unbalanced
cup-product:
\begin{equation} \label{e:defxi} \xi_k(S', N) = \mathrm{d}\left(\!\!\sqrt[N]{\alpha_k}  \right) \underset{E_k / F}{\sqcup} c_{\mathrm{univ}}(E_k, S', N) \end{equation}
where $c_{\mathrm{univ}}(E_k, S', N) \in M(E_k, \dot{S}'_{E_k},
N)^{\vee}[\dot{S}'_{E_k}]_0^{N_{E_k / F}}$ is the image of $\mathrm{Id}_{M(E_k,
\dot{S}'_{E_k}, N)}$ under $\Psi(E_k, S', N)$. Explicitly, for $v \in S'$ and $f
\in M(E_k, \dot{S}'_{E_k}, N)$, $c_{\mathrm{univ}}(E_k, S', N)(w)(f) = f(1,w)$.
The restriction of $\mathrm{d}\left(\!\!\sqrt[N]{\alpha_k}  \right)$ to
$S'_{E_k}$ is a $3$-cocycle
\[ \mathrm{Gal}(\overline{F}_{S' \cup N } / F) \times \mathrm{Gal}(E_k / F)^2
\rightarrow \mathrm{Maps}(S'_{E_k}, \mathcal{I}(F, S' \cup N)[N]) \]
such that
\[ \frac{ \mathrm{d}\left(\!\!\sqrt[N]{\alpha_k}  \right)(\sigma_1, \sigma_2,
\sigma_3)(w_1) }{ \mathrm{d}\left(\!\!\sqrt[N]{\alpha_k}  \right)(\sigma_1,
\sigma_2, \sigma_3)(w_2) } \in \mu_N(\overline{F}) \subset \mathcal{I}(F, S'
\cup N)[N]. \]
This property allows to pair $\mathrm{d}\left(\!\!\sqrt[N]{\alpha_k}
\right)(\sigma_1, \sigma_2, \sigma_3)$ with an element of $M(E_k,
\dot{S}'_{E_k}, N)^{\vee}[\dot{S}'_{E_k}]_0$ to get an element of $P(E_k,
\dot{S}'_{E_k}, N)$, as in \cite[Fact 3.2.3]{Kalgri}. This is the pairing used
in the definition of $\xi_k(S', N)$ \eqref{e:defxi}.  The $2$-cocycle
$\xi_k(S', N)$ is universal in the sense that for any morphism of algebraic
groups $f : P(E_k, \dot{S}'_{E_k}, N) \rightarrow Z$ over $F$ we have
\begin{equation} \label{e:univxik}
f_*(\xi_k(S', N)) = \mathrm{d}\left(\!\!\sqrt[N]{\alpha_k}  \right) \underset{E_k / F}{\sqcup} \Psi(E_k,S',N)(f).
\end{equation}

\begin{defi} \label{d:extk}
Let $k \geq 0$ and $N \geq 1$, and let $S'$ be a finite subset of $V$ containing
$S_k$. Define $\mathcal{E}_k(S', N)$ as the central extension $P(E_k,
\dot{S}'_{E_k}, N) \underset{\xi_k(S', N)}{\boxtimes} \mathrm{Gal}(
\overline{F}_{S' \cup N} / F)$.
\end{defi}
Recall that set-theoretically this is $P(E_k, \dot{S}'_{E_k}, N) \times
\mathrm{Gal}(\overline{F}_{S' \cup N} / F)$, with group law
\[ (x \boxtimes \sigma)(y \boxtimes \tau) = x \sigma(y) \xi_k(S', N)(\sigma,
\tau) \boxtimes \sigma \tau. \]

Suppose $Z \hookrightarrow T$ is an injective morphism of algebraic groups over
$F$ with $Z$ finite, $\exp(Z)|N$ and $T$ a torus split by $E_k$.  Denote $Y =
X_*(T)$ and $\overline{Y} = X_*(T/Z)$.  As shown in \cite[Proposition
3.7.8]{Kalgri}, we have a morphism
\begin{align*}
\iota_k(S',N) : \overline{Y}[S'_{E_k}, \dot{S}'_{E_k}]_0^{N_{E_k / F}}
\longrightarrow & Z^1(P(E_k, \dot{S}'_{E_k}, N) \rightarrow \mathcal{E}_k(S',
N), Z \rightarrow T(\mathcal{O}_{S' \cup N})) \\
\Lambda \longmapsto & \left( x \boxtimes \sigma \mapsto \Psi(E_k, S',
N)^{-1}([\Lambda])(x) \times \Big(\sqrt[N]{\alpha_k} \underset{E_k / F}{\sqcup}
N \Lambda \Big)(\sigma) \right)
\end{align*}
where $[\Lambda]$ is the image of $\Lambda$ in $A^{\vee}[ \dot{S}'_{E_k}
]_0^{N_{E_k / F}}$. As explained in the proof of \cite[Proposition
3.7.8]{Kalgri}, the fact that $\iota_k(S',N)(\Lambda)$ is a $1$-cocycle is
essentially equivalent to
\begin{equation} \label{e:cupTZ}
\mathrm{d}\left(\!\!\sqrt[N]{\alpha_k}  \right) \underset{E_k / F}{\sqcup} N
\Lambda = \mathrm{d}\left(\!\!\sqrt[N]{\alpha_k}  \right) \underset{E_k /
F}{\sqcup} [\Lambda].
\end{equation}
Note that different pairings are used to form cup-products in this equality:
\cite[(3.24)]{Kalgri} on the left, \cite[(3.3)]{Kalgri} on the right. To be
rigorous we should point out that \cite[Proposition 3.7.8]{Kalgri} is stated
with additional assumptions on $S'$, but it is easy to check that the first
point in this proposition does not use these assumptions.

As $N$ and $S'$ vary, there are natural morphisms between the extensions
$\mathcal{E}_k(S', N)$, compatible with $\iota_k(S',N)$. Verifying this is
purely formal, so we omit this verification.

The more challenging and interesting compatibility is when $k$ varies. This is
the main goal of this paper, and we can finally harvest the fruit of our
labour. Assume that $S'$ also contains $S_{k+1}$. Recall (\cite[(3.7)]{Kalgri})
the natural injection $M(E_k, \dot{S}'_{E_k}, N) \hookrightarrow M(E_{k+1},
\dot{S}'_{E_{k+1}}, N)$ mapping $f$ to
\[ (\sigma, w) \mapsto
\begin{cases} f(\sigma, w) & \text{ if } \sigma^{-1} \cdot w \in
\dot{V}_{E_{k+1}}, \\ 0 & \text{  otherwise.} \end{cases} \]
and the dual surjective morphism $\rho_k(S', N) : P(E_{k+1},
\dot{S}'_{E_{k+1}}, N) \rightarrow P(E_k, \dot{S}'_{E_k}, N)$.

It is formal to check that for any finite algebraic group $Z$ over $F$ such that
$\exp(Z) | N$ and the Galois action on $A := X^*(Z)$ factors through
$\mathrm{Gal}(E_k / F)$ and any finite $s' \subset V$, the following diagram is
commutative.
\begin{equation} \label{e:commPsiglobal}
\begin{tikzcd}[column sep=6em]
\mathrm{Hom}(P(E_k, \dot{S}'_{E_k}, N), Z) \arrow[r, "{\Psi(E_k, S', N)}"]
\arrow[d, "{\rho_k(S', N)^*}"] & A^{\vee}[\dot{S}'_{E_k}]_0^{N_{E_k / F}}
\arrow[d, "!_k"] \\
\mathrm{Hom}(P(E_{k+1}, \dot{S}'_{E_{k+1}}, N), Z) \arrow[r, "{\Psi(E_{k+1}, S',
N)}"] & A^{\vee}[\dot{S}'_{E_{k+1}}]_0^{N_{E_{k+1} / F}}
\end{tikzcd}
\end{equation}

\begin{prop} \label{p:towerglobal}
Let $k \geq 0$ and $N \geq 1$, and let $S'$ be a finite subset of $V$ containing
$S_{k+1}$.
\begin{enumerate}
\item Composition with $\rho_k(S', N)$ maps $\xi_{k+1}(S', N)$ to $\xi_k(S', N)$.
In particular, we have a natural morphism of extensions
\begin{alignat*}{2}
\mathcal{E}_{k+1}(S', N) & \longrightarrow & \mathcal{E}_k(S', N) \\
x \boxtimes \sigma & \longmapsto & \rho_k(S', N)(x) \boxtimes \sigma
\end{alignat*}
\item
Let $Z \hookrightarrow T$ be an injective morphism of algebraic groups over $F$
with $Z$ finite and $T$ a torus split by $E_k$. Assume that $\exp(Z)|N$. Let $Y
= X_*(T)$ and $\overline{Y} = X_*(T/Z)$. Then the following diagram commutes

\begin{tikzcd}[column sep=5em]
\overline{Y}[S'_{E_k}, \dot{S}'_{E_k}]_0^{N_{E_k / F}} \arrow[r, "{\iota_k(S',
N)}"] \arrow[d, "!_k"] & Z^1(P(E_k, \dot{S}'_{E_k}, N) \rightarrow
\mathcal{E}_k(S', N), Z \rightarrow T(\mathcal{O}_{S' \cup N})) \arrow[d] \\
\overline{Y}[S'_{E_{k+1}}, \dot{S}'_{E_{k+1}}]_0^{N_{E_{k+1} / F}} \arrow[r,
"{\iota_{k+1}(S', N)}"] & Z^1(P(E_{k+1}, \dot{S}'_{E_{k+1}}, N) \rightarrow
\mathcal{E}_{k+1}(S', N), Z \rightarrow T(\mathcal{O}_{S' \cup N}))
\end{tikzcd}

where the right vertical map is the inflation map induced by the morphism of extensions defined above.
\end{enumerate}
\end{prop}
\begin{proof}
\begin{enumerate}
\item
We use an argument similar to the proof of \cite[Lemma 3.2.8]{Kalgri}. We will
apply Lemma \ref{l:cupproductglobal}. This way we avoid explicit computations
with $3$-cocycles $\mathrm{d}\left(\!\!\sqrt[N]{\alpha_k} \right)$. Denote $Z =
P(E_k, \dot{S}'_{E_k}, N)$ and $A = X^*(Z)$.  Fix a surjective morphism $X
\rightarrow A$ where $X$ is a free $\Z[\mathrm{Gal}(E_k / F)]$-module, and let
$\overline{X}$ be the kernel.  Associated to $X, \overline{X}$ are tori $T,
\overline{T}$ and a short exact sequence $1 \rightarrow Z \rightarrow T
\rightarrow \overline{T} \rightarrow 1$.  Let $Y = X_*(T) =
\mathrm{Hom}_{\Z}(X, \Z)$ and $\overline{Y} = X_*(\overline{T}) =
\mathrm{Hom}_{\Z}(\overline{X}, \Z)$.  We have a short exact sequence $ 0
\rightarrow Y[S'_{E_k}]_0 \rightarrow \overline{Y}[S'_{E_k}]_0 \rightarrow
A^{\vee}[S'_{E_k}]_0 \rightarrow 0 $, where $A = \mathrm{Hom}(X^*(Z), \Q /
\Z)$. The $\mathrm{Gal}(E_k / F)$-modules $Y$ and $Y[S'_{E_k}]$ are
cohomologically trivial (for Tate cohomology) and we have a short exact
sequence $0 \rightarrow Y[S'_{E_k}]_0 \rightarrow Y[S'_{E_k}] \rightarrow Y
\rightarrow 0$, therefore $Y[S'_{E_k}]_0$ is also cohomologically trivial.
This implies in particular that there exists $\Lambda \in
\overline{Y}[S'_{E_k}]_0^{N_{E_k / F}}$ mapping to the class of
$c_{\mathrm{univ}}(E_k, S', N)$ in $A^{\vee}[S'_{E_k}]_0^{N_{E_k / F}} / I_{E_k
/ F} \left( A^{\vee}[S'_{E_k}]_0 \right)$.  Since $I_{E_k / F} \left(
\overline{Y}[S'_{E_k}]_0 \right)$ surjects to $I_{E_k / F} \left(
A^{\vee}[S'_{E_k}]_0 \right)$, we can even assume that the image $[\Lambda]$ of
$\Lambda$ in $A^{\vee}[S'_{E_k}]_0^{N_{E_k / F}}$ equals
$c_{\mathrm{univ}}(E_k, S', N)$.  Then $\Lambda \in \overline{Y}[S'_{E_k},
\dot{S}'_{E_k}]_0^{N_{E_k / F}}$, and applying Lemma \ref{l:cupproductglobal}
to $N \Lambda \in Y[S_{E_k}']_0^{N_{E_k/F}}$ and taking the coboundary, we
obtain the identity between $2$-cocycles taking values in $Z$
\[ \mathrm{d}\left(\!\!\sqrt[N]{\alpha_k}\right) \underset{E_k / F}{\sqcup} N \Lambda  = \mathrm{d}\left(\!\!\sqrt[N]{\alpha_{k+1}}\right) \underset{E_{k+1} / F}{\sqcup} !_k(N\Lambda). \]
Using identity \eqref{e:cupTZ} on both sides, we obtain
\[ \xi_k(S', N) = \mathrm{d}\left(\!\!\sqrt[N]{\alpha_{k+1}}\right) \underset{E_{k+1} / F}{\sqcup} [!_k(\Lambda)]. \]
Moreover $[!_k(\Lambda)] = !_k([\Lambda]) = !_k(c_{\mathrm{univ}}(E_k, S', N)) =
!_k \left( \Psi(E_k, S', N) \left( \mathrm{Id}_{P(E_k, \dot{S}'_{E_k}, N)}
\right)\right)$ equals $\Psi(E_{k+1}, S', N) \left( \rho_k(S', N) \right)$ by
commutativity of diagram \eqref{e:commPsiglobal}. Therefore
\begin{align*}
\xi_k(S', N) =& \mathrm{d}\left(\!\!\sqrt[N]{\alpha_{k+1}}\right) \underset{E_{k+1} / F}{\sqcup} \Psi(E_{k+1}, S', N)(\rho_k(S', N)) \\
=& \rho_k(S', N)_*(\xi_{k+1}(S', N)).
\end{align*}
\item
This is a direct consequence of Lemma \ref{l:cupproductglobal} applied to $N
\Lambda$, using also the commutative diagram \eqref{e:commPsiglobal} for
$[\Lambda]$ and \eqref{e:univxik}.
\end{enumerate}
\end{proof}

Let $Z \hookrightarrow T$ be an injective morphism of algebraic groups over $F$ with $Z$ finite and $T$ a torus.
Let $Y = X_*(T)$ and $\overline{Y} = X_*(T/Z)$, and denote
\[ \overline{Y}[V_{\overline{F}}, \dot{V}]_0^{N_{/F}} = \varinjlim_{k,S'} \overline{Y}[S'_{E_k}, \dot{S}'_{E_k}]_0^{N_{E_k/F}} \]
where the limit is over pairs $k,S'$ such that $E_k$ splits $T$ and $S' \supset S_k$.

\begin{coro} \label{c:spliceglobal}
Let $Z \hookrightarrow T$ be an injective morphism of algebraic groups over $F$
with $Z$ finite and $T$ a torus.  Let $\overline{T} = T / Z$ and let $Y =
X_*(T)$, $\overline{Y} = X_*(\overline{T})$.  Then the morphisms
$\left(\iota_k(S',N)\right)_{k,S',N}$, for $k,S',N$ such that $E_k$ splits $T$,
$\exp(Z) | N$ and $S' \supset S_k$, splice into a morphism
\[ \iota : \overline{Y}[V_{\overline{F}}, \dot{V}]_0^{N_{/F}} \rightarrow Z^1(P \rightarrow \mathcal{E}, Z \rightarrow T(\overline{F})) \]
lifting the cohomological morphism $\iota_{\dot{V}}$ of \cite[Theorem
3.7.3]{Kalgri}.
\end{coro}

\subsection{Generalized Tate-Nakayama morphism for the local towers}

In this section we fix $v \in V$.  We want to study the relation of the map
$\iota$ defined in Corollary \ref{c:spliceglobal} with the localization map
$\mathrm{loc}_v$ defined in \cite[\S 3.6]{Kalgri}. This will necessitate
defining $\mathrm{loc}_v$ (for varying $k,S',N$) for cochains rather than in
cohomology. The first step is to recall several constructions from
\cite{Kalri}. We choose notation similar to the global case instead of notation
used in \cite{Kalri}. For $k \geq 0$ and $N \geq 1$, we have a central
extension
\[ P(E_{k, \dot{v}}, N) \rightarrow \mathcal{E}_{k,v}(N) \rightarrow
\mathrm{Gal}(\overline{F_v} / F_v) \]
where $P(E_{k, \dot{v}}, N) := \mathrm{Res}_{E_{k, \dot{v}} / F_v} (\mu_N) /
\mu_N$.  In particular $M(E_{k, \dot{v}}, N) := X^*(P(E_{k, \dot{v}}, N))$ can
be identified with $\Z/N\Z[\mathrm{Gal}(E_{k, \dot{v}}/F_v)]_0$.  The central
extension
\[ \mathcal{E}_{k,v}(N) := P(E_{k, \dot{v}}, N)
\underset{\xi_{k,v}(N)}{\boxtimes} \mathrm{Gal}(\overline{F_v} / F_v) \]
is defined using the $2$-cocycle
\[ \xi_{k,v}(N) := \mathrm{d}\left(\!\!\sqrt[N]{\alpha_{k,v}}\right)
\underset{E_{k, \dot{v}} / F_v}{\sqcup} c_{\mathrm{univ}}(E_{k, \dot{v}}, N) \]
where $c_{\mathrm{univ}}(E_{k, \dot{v}}, N) \in X^*(P(E_{k, \dot{v}},
N))^{\vee}$ is killed by $N_{E_{k, \dot{v}} / F_v}$, and is defined as $f
\mapsto f(1)$.

Suppose $Z \hookrightarrow T$ is an injective morphism of algebraic groups over
$F_v$ with $Z$ finite, $\exp(Z)|N$ and $T$ a torus split by $E_{k,\dot{v}}$.
Denote $Y = X_*(T)$ and $\overline{Y} = X_*(T/Z)$.  We have a morphism
\begin{align*}
\iota_{k,v}(N) : \overline{Y}^{N_{E_{k,\dot{v}}/F_v}} \longrightarrow &
Z^1(P(E_{k, \dot{v}}, N) \rightarrow \mathcal{E}_{k,v}(N), Z \rightarrow
T(\overline{F_v})) \\
\Lambda \longmapsto & \left( x \boxtimes \sigma \mapsto \Psi(E_{k,\dot{v}},
N)^{-1}([\Lambda])(x) \times \Big(\sqrt[N]{\alpha_{k,v}} \underset{E_{k,\dot{v}}
/ F_v}{\sqcup} N \Lambda \Big)(\sigma) \right)
\end{align*}

The following lemma and proposition, using a formulation analogous to Lemma
\ref{l:cupproductglobal} and Proposition \ref{p:towerglobal}, are essentially
proved in \cite[Lemma 4.5 and Lemma 4.7]{Kalri}. Note that we have arranged for
the $1$-cochain denoted $\alpha_k$ in \cite[Lemma 4.5]{Kalri} to be trivial.
This slightly simplifies formulae.

\begin{lemm} \label{l:cupproductlocal}
Let $T$ be a torus defined over $F_v$.
Denote $Y = X_*(T)$.
Let $k$ be big enough so that $E_{k,\dot{v}}$ splits $T$.
Let $N \geq 1$ be an integer.
Let $\Lambda \in Y^{N_{E_{k,\dot{v}}/F_v}}$.
Then we have an equality of maps $\mathrm{Gal}(\overline{F_v} / F_v) \rightarrow T(\overline{F_v})$:
\[ \sqrt[N]{\alpha_{k,v}} \underset{E_{k,\dot{v}} / F_v}{\sqcup} \Lambda  = \sqrt[N]{\alpha_{k+1,v}} \underset{E_{k+1,\dot{v}} / F_v}{\sqcup} \Lambda. \]
\end{lemm}

As in the global case, there are natural morphisms $\rho_{k,v}(N) : P(E_{k+1, \dot{v}}, N) \rightarrow P(E_{k,\dot{v}}, N)$, denoted $p$ in \cite[(3.2)]{Kalri}.
There are also natural morphisms as $N$ varies, which we do not bother to name.
As in the global case \eqref{e:commPsiglobal}, for any finite algebraic group $Z$ over $F_v$ such that $\exp(Z) | N$ and the Galois action on $A := X^*(Z)$ factors through $\mathrm{Gal}(E_{k,\dot{v}} / F_v)$, we have a commutative diagram:
\begin{equation} \label{e:commPsilocal}
\begin{tikzcd}[column sep=6em]
\mathrm{Hom}(P(E_{k,\dot{v}}, N), Z) \arrow[r, "{\Psi(E_{k,\dot{v}}, N)}"] \arrow[d, "{\rho_{k,v}(N)^*}"] & (A^{\vee})^{N_{E_{k,\dot{v}} / F_v}} \arrow[d, hook] \\
\mathrm{Hom}(P(E_{k+1,\dot{v}}, N), Z) \arrow[r, "{\Psi(E_{k+1, \dot{v}}, N)}"] & (A^{\vee})^{N_{E_{k+1,\dot{v}} / F_v}}
\end{tikzcd}
\end{equation}

\begin{prop} \label{p:towerlocal}
Let $k \geq 0$ and $N \geq 1$.
\begin{enumerate}
\item
Composition with $\rho_{k,v}(N)$ maps $\xi_{k+1,v}(N)$ to $\xi_{k,v}(N)$.
In particular, we have a natural morphism of extensions
\begin{align*}
\mathcal{E}_{k+1,v}(N) \longrightarrow & \mathcal{E}_{k,v}(N) \\
x \boxtimes \sigma \longmapsto & \rho_{k,v}(N)(x) \boxtimes \sigma.
\end{align*}
\item
Let $Z \hookrightarrow T$ be an injective morphism of algebraic groups over $F_v$ with $Z$ finite and $T$ a torus split by $E_{k,\dot{v}}$.
Assume that $\exp(Z)|N$.
Let $Y = X_*(T)$ and $\overline{Y} = X_*(T/Z)$.
Then the following diagram commutes
\[
\begin{tikzcd}[column sep=4em]
\overline{Y}^{N_{E_{k,\dot{v}} / F_v}} \arrow[r, "{\iota_{k,v}(N)}"] \arrow[d, equal] & Z^1(P(E_{k,\dot{v}}, N) \rightarrow \mathcal{E}_{k,v}(N), Z \rightarrow T(\overline{F_v})) \arrow[d] \\
\overline{Y}^{N_{E_{k+1,\dot{v}} / F_v}} \arrow[r, "{\iota_{k+1,v}(N)}"] & Z^1(P(E_{k+1,v}, N) \rightarrow \mathcal{E}_{k+1,v}(N), Z \rightarrow T(\overline{F_v}))
\end{tikzcd}
\]
where the right vertical map is inflation for the morphism of extensions defined above.
\end{enumerate}
\end{prop}
\begin{proof}
The proof is similar to that of Proposition \ref{p:towerglobal}, in fact slightly easier, so we omit it.
\end{proof}

Let $Z \hookrightarrow T$ be an injective morphism of algebraic groups over $F_v$ with $Z$ finite and $T$ a torus.
Let $Y = X_*(T)$ and $\overline{Y} = X_*(T/Z)$.
Denote $\overline{Y}^{N_{/F_v}} = \overline{Y}^{N_{E_{k, \dot{v}}/F_v}}$ for any $k$ such that $E_{k,\dot{v}}$ splits $T$.

\begin{coro} \label{c:splicelocal}
Let $Z \hookrightarrow T$ be an injective morphism of algebraic groups over $F_v$ with $Z$ finite and $T$ a torus.
Let $Y = X_*(T)$ and $\overline{Y} = X_*(T/Z)$.
Then the morphisms $\left(\iota_{k,v}(N)\right)_{k,N}$, for $k,N$ such that $E_{k,\dot{v}}$ splits $T$ and $\exp(Z) | N$, splice into a morphism
\[ \iota_v : \overline{Y}^{N_{/F_v}} \rightarrow Z^1(P_v \rightarrow \mathcal{E}_v, Z \rightarrow T(\overline{F_v})) \]
lifting the morphism in cohomology of \cite[Theorem 4.8]{Kalri}.
\end{coro}

\subsection{Localization}

In this section $v \in V$ is fixed. We want to study the relationship between
$\iota$ (Corollary \ref{c:spliceglobal}), $\iota_v$ (Corollary
\ref{c:splicelocal}) and $\mathrm{loc}_v$ (\cite[\S 3.6]{Kalgri}). We study it
for fixed $k \geq 0$ first.

Recall (\cite[(3.11)]{Kalgri}) the morphisms $\mathrm{loc}_{k,v}(S',N) :
P(E_{k, \dot{v}}, N) \rightarrow P(E_k, \dot{S}'_{E_k}, N)$. If $v \in S'$ it
is dual to $f \mapsto (\sigma \mapsto f(\sigma, \dot{v}))$ if $v \in S'$. We
define it to be trivial if $v \not\in S'$. It is $\mathrm{Gal}(E_{k, \dot{v}} /
F_v)$-equivariant, and there are obvious commuting diagrams as $S'$ and $N$
vary.

For $M$ a $\mathrm{Gal}(E_k/F)$-module, recall the morphism $l_{k,v} :
M[S'_{E_k}]^{N_{E_k / F}} \rightarrow M^{N_{E_{k, \dot{v}} / F_v}}$ (denoted
$l_v^k$ in \cite[Lemma 3.7.2]{Kalgri}) defined by
\[ l_{k,v}(\Lambda) = \sum_{r \in R'_{k,v}} r^{-1} \left( \Lambda(r \cdot \dot{v}_k) \right) \]
if $v \in S'$, and zero otherwise.

\begin{lemm} \label{l:cupproductlocalglobal}
Let $T$ be a torus defined over $F$. Denote $Y = X_*(T)$. Let $k$ be big
enough so that $E_k$ splits $T$. Let $N \geq 1$ be an integer.  Let $S'$ be a
finite subset of $V$ containing $S_k$. Let $\Lambda \in
Y[S'_{E_k}]_0^{N_{E_k/F}}$.

Let $i \geq 0$ be big enough so that $\sqrt[N]{\alpha_{k,v}}$ takes values in
$E_{k+i, \dot{v}}^{\times}$. Then we have an equality of maps $\mathrm{Gal}(
\overline{F} / F) \rightarrow T(\overline{F} \otimes_F F_v)$:
\begin{multline*}
\mathrm{pr}_v \left(\sqrt[N]{\alpha_k} \underset{E_k / F}{\sqcup} \Lambda
\right) =
\mathrm{ES}^1_{R'_{k+i,v}} \left( \sqrt[N]{\alpha_{k,v}} \underset{E_{k,
\dot{v}} / F_v}{\sqcup} l_{k,v}(\Lambda) \right)
\times \mathrm{d}\left(\mathrm{pr}_v(\sqrt[N]{\beta_k}) \underset{E_k /
F}{\sqcup} \Lambda \right) \\
\times \left( \mathrm{pr}_v(\delta_k(N)) \underset{E_k / F}{\sqcup} \Lambda
\right).
\end{multline*}

In particular, upon restriction to $\mathrm{Gal}(\overline{F_v} / F_v)$ and
projection to $T(\overline{F_v})$:
\begin{multline*}
\mathrm{pr}_{\dot{v}} \left(\sqrt[N]{\alpha_k} \underset{E_k / F}{\sqcup}
\Lambda \right) =
\left( \sqrt[N]{\alpha_{k,v}} \underset{E_{k, \dot{v}} / F_v}{\sqcup}
l_{k,v}(\Lambda) \right)
\times \mathrm{d}\left(\mathrm{pr}_{\dot{v}}(\sqrt[N]{\beta_k}) \underset{E_k /
F}{\sqcup} \Lambda \right) \\
\times \left( \mathrm{pr}_{\dot{v}}(\delta_k(N)) \underset{E_k / F}{\sqcup}
\Lambda \right).
\end{multline*}
\end{lemm}
Note that the first equality implicitly uses the identification
\begin{align*}
\mathrm{ind}_{\mathrm{Gal}(E_{k+i,
\dot{v}}/F_v)}^{\mathrm{Gal}(E_{k+i}/F)}(E_{k+i, \dot{v}}^{\times}) &
\xlongrightarrow{\sim} (E_{k+i, \dot{v}} \otimes_F F_v)^{\times} \\
f & \longmapsto \prod_{g \in \mathrm{Gal}(E_{k+i, \dot{v}}/F_v)) \backslash
\mathrm{Gal}(E_{k+i}/F)} g^{-1}(f(g))
\end{align*}
to see $\mathrm{ES}^1_{R'_{k+i,v}} \left( \sqrt[N]{\alpha_{k,v}} \underset{E_{k,
\dot{v}} / F_v}{\sqcup} l_{k,v}(\Lambda) \right)$ as a map
$\mathrm{Gal}(E_{k+i}/F) \rightarrow T(E_{k+i, \dot{v}} \otimes_F F_v)$.
\begin{proof}
Recall that by definition of $\delta_k(N)$, we have $\sqrt[N]{\alpha_k} =
\sqrt[N]{\alpha'_k} \mathrm{d}(\!\!\sqrt[N]{\beta_k}) \delta_k(N)$, and we
compute unbalanced cup-products with these three terms separately. In the case
of $\delta_k(N)$ there is nothing to prove, so we first consider
$\mathrm{d}(\!\!\sqrt[N]{\beta_k})$. By \cite[Fact 4.3]{Kalri} we have
\[ \mathrm{d}(\!\!\sqrt[N]{\beta_k}) \underset{E_k / F}{\sqcup} \Lambda =
\mathrm{d}\left(\!\! \sqrt[N]{\beta_k} \underset{E_k / F}{\sqcup} \Lambda
\right) \]
and thus upon restriction to $\mathrm{Gal}(\overline{F_v} / F_v)$,
\[ \mathrm{pr}_{\dot{v}} \left( \mathrm{d}(\!\!\sqrt[N]{\beta_k}) \underset{E_k
/ F}{\sqcup} \Lambda \right) = \mathrm{d}\left(\mathrm{pr}_{\dot{v}} \left(\!\!
\sqrt[N]{\beta_k} \underset{E_k / F}{\sqcup} \Lambda \right) \right). \]

Let us now consider $\sqrt[N]{\alpha'_k}$. For $\sigma \in \mathrm{Gal}(E_{k} /
F)$ we have
\[ \mathrm{pr}_v \left( \left( \sqrt[N]{\alpha'_k} \sqcup_{E_k / F} \Lambda
\right)(\sigma) \right) = \prod_{\gamma \in R'_{k,v}} \prod_{\tau \in
\mathrm{Gal}(E_k / F)} \sqrt[N]{\alpha'_k}(\sigma, \tau)(\sigma \tau \gamma
\cdot \dot{v}_k) \otimes \sigma \tau \left( \Lambda(\gamma \cdot \dot{v}_k)
\right). \]
Write $\tau \gamma = r \tau'$ and $\sigma r = r' \sigma'$ where $r,r' \in
R'_{k,v}$ and $\tau', \sigma' \in \mathrm{Gal}(E_{k, \dot{v}} / F_v)$ are
functions of $(\sigma, \gamma, \tau)$. For $\sigma$ and $\gamma$ fixed the map
$\tau \mapsto (r, \tau')$ is bijective onto $R'_{k,v} \times \mathrm{Gal}(E_{k,
\dot{v}} / F_v)$. We obtain
\begin{align*}
& \mathrm{pr}_v \left( \left( \sqrt[N]{\alpha'_k} \sqcup_{E_k / F} \Lambda
\right)(\sigma) \right) \\
=& \prod_{\gamma \in R'_{k,v}} \prod_{r \in R'_{k,v}} \prod_{\tau' \in
\mathrm{Gal}(E_{k,\dot{v}} / F_v)} \sqrt[N]{\alpha'_k}(r' \sigma' r^{-1}, r
\tau' \gamma^{-1})(r' \cdot \dot{v}_k) \otimes r' \sigma' \tau' \gamma^{-1}
\left( \Lambda(\gamma \cdot \dot{v}_k) \right) \\
\end{align*}
where $r' \sigma' = \sigma r$, $r' \in R'_{k, v}$ and $\sigma' \in
\mathrm{Gal}(E_{k,\dot{v}} / F_v)$ being functions of $r$. Recall that by
definition,
\[ \sqrt[N]{\alpha'_k}(r' \sigma' r^{-1}, r \tau' \gamma^{-1})(r' \cdot
\dot{v}_k) = r' \left( j_{k,v} \left( \sqrt[N]{\alpha_{k,v}}(\sigma', \tau')
\right) \right). \]
Therefore
\begin{align*}
& \mathrm{pr}_v \left( \left( \sqrt[N]{\alpha'_k} \underset{E_k/F}{\sqcup}
\Lambda \right)(\sigma) \right) \\
=& \prod_{\gamma \in R'_{k,v}} \prod_{r \in R'_{k,v}} \prod_{\tau' \in
\mathrm{Gal}(E_{k,\dot{v}} / F_v)} r' \left( j_{k,v} \left(
\sqrt[N]{\alpha_{k,v}}(\sigma', \tau') \right) \otimes \sigma' \tau' \gamma^{-1}
\left( \Lambda(\gamma \cdot \dot{v}_k) \right) \right) \\
=& \prod_{r \in R'_{k,v}} r' \left( \prod_{\tau' \in \mathrm{Gal}(E_{k,\dot{v}}
/ F_v)} j_{k,v} \left( \sqrt[N]{\alpha_{k,v}}(\sigma', \tau') \right) \otimes
\sigma' \tau' \left( l_{k,v}(\Lambda) \right) \right).
\end{align*}
The map $r \mapsto r'$ from $R'_{k,v}$ to itself is bijective, so we can write
this as
\[ \prod_{r' \in R'_{k,v}} r' \left( \prod_{\tau' \in \mathrm{Gal}(E_{k,\dot{v}}
/ F_v)} j_{k,v} \left( \sqrt[N]{\alpha_{k,v}}(\sigma', \tau') \right) \otimes
\sigma' \tau' \left(l_{k,v}(\Lambda) \right) \right) \]
where $\sigma'$ depends on $r'$ and is the unique element of
$\mathrm{Gal}(E_{k,\dot{v}} / F_v)$ such that $\sigma^{-1} r' \sigma' \in
R'_{k,v}$.
\[ \prod_{r' \in R'_{k+i,v}} r' \left( \prod_{\tau' \in
\mathrm{Gal}(E_{k,\dot{v}} / F_v)} j_{k+i,v} \left(
\sqrt[N]{\alpha_{k,v}}(\sigma', \tau') \right) \otimes \sigma' \tau'
\left(l_{k,v}(\Lambda) \right) \right) \]
and it is easy to check that this is equal to
$\mathrm{ES}^1_{R'_{k+i,v}}\left(\sqrt[N]{\alpha_{k,v}} \underset{E_k/F}{\sqcup}
l_{k,v}(\Lambda)\right)(\sigma)$.
\end{proof}

It is formal to check that for any finite algebraic group $Z$ over $F$ such that
$\exp(Z) | N$ and the Galois action on $A := X^*(Z)$ factors through
$\mathrm{Gal}(E_k / F)$, and any finite set of places $S'$ of $F$ such that $S'
\supset S_k$, the following diagram is commutative.
\begin{equation} \label{e:commPsilocalglobal}
\begin{tikzcd}[row sep=3em, column sep=6em]
\mathrm{Hom}(P(E_k, \dot{S}'_{E_k}, N), Z) \arrow[r, "{\Psi(E_k, S', N)}"] \arrow[d, "{\left( \mathrm{loc}_{k,v}(S', N) \right)^*}"] & A^{\vee}[\dot{S}'_{E_k}]_0^{N_{E_k / F}} \arrow[d, "{l_{k,v}}"] \\
\mathrm{Hom}(P(E_{k, \dot{v}}, N), Z) \arrow[r, "{\Psi(E_{k, \dot{v}}, S', N)}"] & (A^{\vee})^{N_{E_{k, \dot{v}} / F_v}}
\end{tikzcd}
\end{equation}

\begin{defi} \label{d:eta}
For $k \geq 0$, $N \geq 1$ and $S'$ a finite subset of $V$ containing $S_k$, let
$\eta_{k,v}(S', N) : \mathrm{Gal}(\overline{F_v} / F_v) \rightarrow P(E_k,
\dot{S}'_{E_k}, N)$ be the restriction of $\mathrm{pr}_{\dot{v}}(\delta_k(N))
\underset{E_k / F}{\sqcup} c_{\mathrm{univ}}(E_k, S', N)$ to
$\mathrm{Gal}(\overline{F_v} / F_v)$.
\end{defi}

\begin{prop} \label{p:klocalglobal}
Let $k \geq0$, $N \geq 1$ and $S'$ a finite subset of $V$ containing $S_k$.
\begin{enumerate}
\item
The restriction of the $2$-cocycle $\xi_k(S', N)$ to
$\mathrm{Gal}(\overline{F_v} / F_v)$ equals
\[ \left( \mathrm{loc}_{k,v}(S', N) \right)_*(\xi_{k,v}(N)) \times
\mathrm{d}\left( \eta_{k,v}(S', N) \right) \]
and so the morphism $\mathrm{loc}_{k,v}(S', N) : P(E_{k,\dot{v}}, N) \rightarrow
P(E_k, \dot{S}'_{E_k}, N)$ can be extended to a morphism of extensions
\begin{align*}
\mathrm{loc}_{k,v}(S', N) : \mathcal{E}_{k, v}(N) & \longrightarrow
\mathcal{E}_k(S', N) \\
x \boxtimes \sigma & \longmapsto \frac{\mathrm{loc}_{k,v}(S', N)(x) }{
\eta_{k,v}(S', N)(\sigma)} \boxtimes \sigma.
\end{align*}
\item
Let $Z \hookrightarrow T$ be an injective morphism of algebraic groups over $F$
with $Z$ finite and $T$ a torus split by $E_k$. Assume that $\exp(Z)|N$. Let $Y
= X_*(T)$ and $\overline{Y} = X_*(T/Z)$. Then for any $\Lambda \in
\overline{Y}[S'_{E_k}, \dot{S}'_{E_k}]_0^{N_{E_k/F}}$, the following identity
holds in $Z^1(P(E_{k,\dot{v}}, N) \rightarrow \mathcal{E}_{k,v}(N), Z
\rightarrow T(\overline{F_v}))$:
\begin{equation} \label{e:lociotak}
\mathrm{pr}_{\dot{v}} \left( \iota_k(S', N)(\Lambda) \circ
\mathrm{loc}_{k,v}(S', N) \right) = \iota_{k,v}(N)(l_{k,v}(\Lambda)) \times
\mathrm{d}\left( \mathrm{pr}_{\dot{v}}(\sqrt[N]{\beta_k})
\underset{E_k/F}{\sqcup} N \Lambda \right).
\end{equation}
\end{enumerate}
\end{prop}
\begin{proof}
The proof is similar to that of Proposition \ref{p:towerglobal}, and we will be
more concise.
\begin{enumerate}
\item
Let $Z = P(E_k, \dot{S}'_{E_k}, N)$ and $A = X^*(Z)$. As in the proof of
Proposition \ref{p:towerglobal} we can find an embedding $Z \hookrightarrow T$
where $T$ is a torus over $F$, split over $E_k$ and such that $Y := X_*(T)$ is a
free $\Z[\mathrm{Gal}(E_k/F)]$-module. Let $\overline{Y} = X_*(T/Z)$. There
exists $\Lambda \in \overline{Y}[S'_{E_k}, \dot{S}'_{E_k}]_0^{N_{E_k / F}}$ such
that its image $[\Lambda]$ in $A^{\vee}[\dot{S}'_{E_k}]_0^{N_{E_k / F}}$ equals
$c_{\mathrm{univ}}(E_k, S', N)$. Applying Lemma \ref{l:cupproductlocalglobal} to
$N \Lambda \in Y$ and taking the coboundary, we obtain the identity between
$2$-cocycles $\mathrm{Gal}(\overline{F_v}/F_v)^2 \rightarrow T(\overline{F_v})$
\[ \mathrm{d}\left(\!\!\sqrt[N]{\alpha_k}\right) \underset{E_k / F}{\sqcup} N
\Lambda = \left( \mathrm{d}\left(\!\!\sqrt[N]{\alpha_{k,v}}\right)
\underset{E_{k,\dot{v}} / F_v}{\sqcup} Nl_{k,v}(\Lambda) \right) \times
\mathrm{d}\left( \mathrm{pr}_{\dot{v}}(\delta_k(N)) \underset{E_k/F}{\sqcup} N
\Lambda \right). \]
Since $\mathrm{d}\left(\!\!\sqrt[N]{\alpha_k}\right)^N=1$,
$\mathrm{d}\left(\!\!\sqrt[N]{\alpha_{k,v}}\right)^N=1$ and $\delta_k(N)^N=1$
all three terms take values in $Z \subset T(\overline{F_v})$ and the equality
can be written
\[ \mathrm{d}\left(\!\!\sqrt[N]{\alpha_k}\right) \underset{E_k / F}{\sqcup}
[\Lambda] = \left( \mathrm{d}\left(\!\!\sqrt[N]{\alpha_{k,v}}\right)
\underset{E_{k,\dot{v}} / F_v}{\sqcup} l_{k,v}([\Lambda]) \right) \times
\mathrm{d}\left( \mathrm{pr}_{\dot{v}}(\delta_k(N)) \underset{E_k/F}{\sqcup}
[\Lambda] \right) \]
using the pairing $\mu_N \times A^{\vee} \rightarrow Z$. Using the fact that
\[ l_{k,v}(c_{\mathrm{univ}(E_k, S', N)}) = \Psi(E_{k, \dot{v}, S',
N})(\mathrm{loc}_{k,v}(S',N)) \]
thanks to \eqref{e:commPsilocalglobal}, we obtain the desired equality.
\item
This is a direct consequence of Lemma \ref{l:cupproductlocalglobal} applied to
$N \Lambda$, using also the commutative diagram \eqref{e:commPsilocalglobal}
with $[\Lambda]$ in the top right corner.
\end{enumerate}
\end{proof}

\begin{lemm} \label{l:cupproductdeltabeta}
Let $T$ be a torus defined over $F$. Denote $Y = X_*(T)$. Let $k$ be big enough
so that $E_k$ splits $T$. Let $N \geq 1$ be an integer. Let $S'$ be a finite
subset of $V$ containing $S_{k+1}$. Let $\Lambda \in Y[S'_{E_k}]_0^{N_{E_k/F}}$.
Then we have an equality of maps $\mathrm{Gal}(\overline{F}_{S' \cup N} / F)
\rightarrow Y \otimes_{\Z} \mathcal{I}(F, S' \cup N)[N]$:
\begin{equation} \label{e:deltacup}
\delta_k(N) \underset{E_k / F}{\sqcup} \Lambda = \delta_{k+1}(N)
\underset{E_{k+1} / F}{\sqcup} !_k(\Lambda)
\end{equation}
and an equality in $Y \otimes_{\Z} \mathcal{I}(F, S' \cup N)$:
\begin{equation} \label{e:betacup}
\sqrt[N]{\beta_k} \underset{E_k / F}{\sqcup} \Lambda = \sqrt[N]{\beta_{k+1}}
\underset{E_{k+1} / F}{\sqcup} !_k(\Lambda).
\end{equation}
\end{lemm}
Note that in \eqref{e:betacup} the left hand side belongs to $Y \otimes_{\Z}
\mathcal{I}(F, S_k \cup N)$.
\begin{proof}
For \eqref{e:deltacup} the proof is identical to that of Lemma
\ref{l:cupproductglobal}. For \eqref{e:betacup} the proof is similar and easier,
so we omit it.
\end{proof}

The localization maps $l_{k,v}$ are compatible with increasing $k$, i.e.\
$l_{k+1, v} \,\circ\, !_k = l_{k,v}$. This is proved in \cite[Lemma
3.7.2]{Kalgri}. Thus for any embedding $Z \hookrightarrow T$ of algebraic
groups over $F$ with $Z$ finite and $T$ a torus, they splice into
\[ l_v : \overline{Y}[V_{\overline{F}}, \dot{V}]_0^{N_{/F}} \rightarrow
\overline{Y}^{N_{/F_v}} \]
where $\overline{Y} = X_*(T/Z)$.

The localization morphisms $\mathrm{loc}_{k,v}(S',N) : P(E_{k,\dot{v}},N)
\rightarrow P(E_k, \dot{S}'_{E_k}, N)$ are also compatible with varying $k$. We
formulate this compatibility, together with \eqref{e:commPsiglobal},
\eqref{e:commPsilocal} and \eqref{e:commPsilocalglobal}, using a commutative
cubic diagram below. For any finite algebraic group $Z$ over $F$ such that
$\exp(Z) | N$ and the Galois action on $A := X^*(Z)$ factors through
$\mathrm{Gal}(E_k / F)$, and any finite set of places $S'$ of $F$ such that $S'
\supset S_{k+1}$, the following cubic diagram is commutative.
\begin{equation} \label{e:commPsi}
\begin{tikzcd}
& \mathrm{Hom}(P(E_k, \dot{S}'_{E_k}, N), Z) \arrow[rr, "{\Psi(E_k,S',N)}"] \arrow[dd, "{\rho_k(S',N)^*}" near start] \arrow[dl, "{\mathrm{loc}_{k,v}(S',N)^*}"'] & & A^{\vee}[\dot{S}'_{E_k}]_0^{N_{E_k/F}} \arrow[dd, "!_k"] \arrow[dl, "{l_{k,v}}"] \\
   \mathrm{Hom}(P(E_{k,\dot{v}},N), Z) \arrow[rr, crossing over, "{\Psi(E_{k,\dot{v}},N)}"' near start] \arrow[dd, "{\rho_{k,v}(N)^*}"'] & & (A^{\vee})^{N_{E_{k,\dot{v}}/F_v}} \\
    & \mathrm{Hom}(P(E_{k+1}, \dot{S}'_{E_{k+1}}, N), Z)  \arrow[rr, "{\Psi(E_{k+1},S',N)}" near start] \arrow[dl, "{\mathrm{loc}_{k+1,v}(S',N)^*}"' near start] & & A^{\vee}[\dot{S}'_{E_{k+1}}]_0^{N_{E_{k+1}/F}} \arrow[dl, "{l_{k+1,v}}"] \\
    \mathrm{Hom}(P(E_{k+1,\dot{v}},N), Z) \arrow[rr, "{\Psi(E_{k+1,\dot{v}},N)}"] && (A^{\vee})^{N_{E_{k+1,\dot{v}}/F_v}} \ar[from=uu,crossing over,hook]
\end{tikzcd}
\end{equation}
In fact the commutativity of the left face follows from the commutativity of the
other faces and the fact that the morphisms $\Psi$ are isomorphisms.

\begin{prop} \label{p:towerlocalglobal}
\begin{enumerate}
\item
For any $k \geq 0$, $N \geq 1$ and $S'$ a finite subset of $V$ containing
$S_{k+1}$ we have $\eta_{k,v}(S', N) = \rho_k(S', N)_* \left( \eta_{k+1,v}(S',N)
\right)$, and a commutative diagram of central extensions
\begin{equation} \label{e:commextlocalglobal}
\begin{tikzcd}[column sep=6em]
\mathcal{E}_{k+1,v}(N) \arrow[r, "{\mathrm{loc}_{k+1,v}(S', N)}"] \arrow[d] &
\mathcal{E}_{k+1}(S', N) \arrow[d] \\
\mathcal{E}_{k,v}(N) \arrow[r, "{\mathrm{loc}_{k,v}(S', N)}"] &
\mathcal{E}_k(S', N)
\end{tikzcd}
\end{equation}
Therefore as $k,S',N$ vary, the morphisms $\mathrm{loc}_{k,v}(S', N)$ yield
$\mathrm{loc}_v : \mathcal{E}_v \rightarrow \mathcal{E}$.
\item
Let $Z \hookrightarrow T$ be an injective morphism of algebraic groups over $F$
with $Z$ finite and $T$ a torus. Let $Y = X_*(T)$ and $\overline{Y} = X_*(T/Z)$.
Let $\Lambda \in \overline{Y}[V_{\overline{F}}, \dot{V}]_0^{N_{/F}}$. For
$k,S',N$ such that $E_k$ splits $T$, $N \geq 1$ is divisible by $\exp(Z)$, $S'$
contains $S_k$ and $\Lambda$ comes from an element $\Lambda_k \in
\overline{Y}[S'_k, \dot{S}'_{E_k}]_0^{N_{E_k/F}}$, let $\kappa_v(\Lambda) =
\mathrm{pr}_{\dot{v}}(\sqrt[N]{\beta_k}) \underset{E_k/F}{\sqcup} N \Lambda_k
\in T(\overline{F_v})$. As the notation suggests, it does not depend on the
choice of $k,S',N$. Then the following identity holds in $Z^1(P_v \rightarrow
\mathcal{E}_v, Z \rightarrow T(\overline{F_v}))$:
\begin{equation} \label{e:lociotatower}
\mathrm{pr}_{\dot{v}} \left( \iota(\Lambda) \circ \mathrm{loc}_v \right) =
\iota_v(l_v(\Lambda)) \times \mathrm{d}\left( \kappa_v(\Lambda) \right).
\end{equation}
\end{enumerate}
\end{prop}
\begin{proof}
\begin{enumerate}
\item
The equality $\eta_{k,v}(S', N) = \rho_k(S', N)_* \left( \eta_{k+1,v}(S',N)
\right)$ follows from \eqref{e:deltacup} in Lemma \ref{l:cupproductdeltabeta},
using the same argument as in the proof of Proposition \ref{p:towerglobal}.
Commutativity of diagram \eqref{e:commextlocalglobal} follows from this equality
and the equality $\mathrm{loc}_{k,v}(S',N) \circ \rho_{k,v}(N) = \rho_k(S', N)
\circ \mathrm{loc}_{k+1,v}(S', N)$, which is equivalent to commutativity of the
left face of \eqref{e:commPsi} for $Z = P(E_k, \dot{S}'_{E_k}, N)$.
\item
The fact that $\kappa_v(\Lambda)$ does not depend on the choice of $k, S', N$
follows from \eqref{e:betacup} in Lemma \ref{l:cupproductdeltabeta}, and
\eqref{e:lociotatower} is \eqref{e:lociotak} in Proposition
\ref{p:klocalglobal}.
\end{enumerate}
\end{proof}

\subsection{Comparison with Kaletha's canonical class}

As in \cite{Kalgri} we consider the profinite
$\mathrm{Gal}(\overline{F}/F)$-module $P = \varprojlim_{k, S', N} P(E_k,
\dot{S}'_{E_k}, N)$. It is perhaps more natural to view $P$ as an inverse limit
of finite diagonalizable algebraic groups over $F$, but in any case we will only
use $P = P(\overline{F}) = P(\overline{F_v})$ (for any $v \in V$). The
$2$-cocycles $\xi_k(S', N)$ are compatible by Proposition \ref{p:towerglobal},
and we obtain a (continuous) $2$-cocycle $\xi \in Z^2(F, P)$. Let us check that
$\xi$ represents the canonical class in $H^2(\mathrm{Gal}(\overline{F}/F), P)$
defined in \cite[\S 3.5]{Kalgri}.

As in \cite[\S 3.3]{Kalgri}, fix a cofinal sequence $(N_k)_{k \geq 0}$ in
$\Z_{>0}$ (for the partial order defined by divisibility) with $N_0 = 1$ and
such that for any $k \geq 0$, $S_k$ contains all places dividing $N_k$ (this is
possible up to enlarging the finite sets $S_k$). To simplify notation we write
$P_k = P(E_k, \dot{S}_{k,E_k}, N_k)$, $\rho_k : P_{k+1} \twoheadrightarrow P_k$
and $c_{\mathrm{univ},k} = c_{\mathrm{univ}}(E_k, S_k, N_k)$.

First we need to go back to the construction of a resolution of $P$ by pro-tori
in \cite[Lemma 3.5.1]{Kalgri}.
\begin{lemm}
There exists a family of resolutions, for $k \geq 0$,
\[ 1 \rightarrow P_k \rightarrow T_k \rightarrow \overline{T}_k \rightarrow 1 \]
of $P_k$ by tori $T_k,\overline{T}_k$ defined over $F$ and split by $E_k$, and
morphisms $r_k : T_{k+1} \rightarrow T_k$ and $\overline{r}_k :
\overline{T}_{k+1} \rightarrow \overline{T}_k$, such that
\begin{enumerate}
\item For all $k \geq 0$, the diagram
\begin{equation}
\begin{tikzcd}
P_{k+1} \arrow[r] \arrow[d, "{\rho_k}"] & T_{k+1} \arrow[r] \arrow[d, "{r_k}"] & \overline{T}_{k+1} \arrow[d, "{\overline{r}_k}"] \\
P_k \arrow[r]               & T_k \arrow[r]               & \overline{T}_k
\end{tikzcd}
\end{equation}
is commutative and $r_k, \overline{r}_k$ are surjective with connected kernels.
\item Letting $Y_k = X_*(T_k)$ and $\overline{Y}_k = X_*(\overline{T}_k)$, there exists a family $(\Lambda_k)_{k \geq 0}$ where $\Lambda_k \in \overline{Y}_k[S_{k,E_k}, \dot{S}_{k,E_k}]_0^{N_{E_k/F}}$ maps to $c_{\mathrm{univ},k} \in M_k^{\vee}[\dot{S}_{k,E_k}]_0^{N_{E_k/F}}$ and $!_k(\Lambda_k) = \overline{r}_k(\Lambda_{k+1})$ in $\overline{Y}_k[S_{k+1,E_{k+1}}, \dot{S}_{k+1,E_{k+1}}]_0^{N_{E_{k+1}/F}}$.
\end{enumerate}
\end{lemm}
\begin{proof}
For $k \geq 0$ let $X_k' = \Z[\mathrm{Gal(E_k/F)}][M_k]$, so that there is a
canonical surjective map of $\Z[\mathrm{Gal}(E_k/F)]$-modules $X_k' \rightarrow
M_k$. Let $X_0 = X'_0$, and for $k \geq 0$ let $X_{k+1} = X_k \oplus
X'_{k+1}$. We have a natural surjective morphism $X_k \rightarrow M_k$, which
for $k > 0$ is obtained as the sum of $X_{k-1} \rightarrow M_{k-1}
\hookrightarrow M_k$ and $X'_k \rightarrow M_k$. Let $T_k$ be the torus over
$F$ such that $X^*(T_k) = X_k$, and let $U_k = T_k/P_k$. Compared to the
construction in \cite[Lemma 3.5.1]{Kalgri}, the only difference is that
$X'_{k+1}$ is free with basis $M_{k+1}$ instead of $M_{k+1} \smallsetminus
M_k$. Let $Y_k = X_*(T_k)$ and $\overline{Y}_k = X_*(U_k)$, so that we have an
exact sequence
\[ 0 \rightarrow Y_k \rightarrow \overline{Y}_k \rightarrow M_k^{\vee} \rightarrow 0. \]
Let $\overline{X}'_k = \ker(X'_k \rightarrow M_k)$, $Y'_k = \Hom_{\Z}(X'_k,
\Z)$ and $\overline{Y}'_k = \Hom_{\Z}(\overline{X}'_k, \Z)$ Since $X'_k$ is a
free $\Z[\mathrm{Gal}(E_k/F]$-module, using the same argument as in Proposition
\ref{p:towerglobal} we can find $\Upsilon_k \in \overline{Y}'_k[S_{k,E_k},
\dot{S}_{k,E_k}]_0^{N_{E_k/F}}$ mapping to $c_{\mathrm{univ}, k}$. For all $k
\geq 0$ we can identify $\overline{Y}_{k+1}$ with the group of $f \oplus g \in
\overline{Y}_k \oplus \overline{Y}'_{k+1}$ such that $[f] = [g]$ in
$M_k^{\vee}$. We use these identifications to construct $\Lambda_k$
inductively from $\Upsilon_k$. Let $\Lambda_0 = \Upsilon_0$, and for $k \geq
0$ let $\Lambda_{k+1} = !_k(\Lambda_k) \oplus \Upsilon_{k+1} \in
(\overline{Y}_k \oplus \overline{Y}'_{k+1})[S_{k+1,E_{k+1}},
\dot{S}_{k+1,E_{k+1}}]_0^{N_{E_{k+1}/F}}$. Thanks to the equality
$!_k(c_{\mathrm{univ},k}) = \rho_k(c_{\mathrm{univ}, k+1})$, we have that
$\Lambda_{k+1} \in \overline{Y}_{k+1}[S_{k+1,E_{k+1}},
\dot{S}_{k+1,E_{k+1}}]_0^{N_{E_{k+1}/F}}$.
\end{proof}

\begin{prop} \label{p:xiiscanclass}
The $2$-cocycle $\xi$ belongs to the canonical class in $H^2(F, P)$ defined in
\cite[\S 3.5]{Kalgri}.
\end{prop}
\begin{proof}
We use the tower of resolutions and the family $(\Lambda_k)_{k \geq 0}$
constructed in the previous lemma. We are looking for a compatible family
$(a_k, b_k)_{k \geq 0}$ where $a_k \in C^1(F, T_k)$ and $b_k \in
\overline{T}_k(\A_{\overline{F}})$ are such that $\overline{a_k} =
\mathrm{d}(b_k)$ in $C^1(A, \overline{T}_k)$ and
\[ \xi_k = \prod_{v \in S_k} \mathrm{ES}^2_{R'_{k+i,v}}(\mathrm{loc}_{k,v}(\xi_{k,v})) \times \mathrm{d}(a_k) \]
in $Z^2(\A, T_k)$, for $i \geq 0$ big enough for the corestrictions to be
well-defined.

By Lemma \ref{l:cupproductlocalglobal} and thanks to the fact that $\Lambda_k$
has support in the finite set $S_{k,E_k}$, for $i \geq 0$ big enough we have
\[ \sqrt[N_k]{\alpha'_k} \underset{E_k/F}{\sqcup} N_k \Lambda_k = \prod_{v \in
S_k} \mathrm{ES}^1_{R'_{k+i,v}} \left( \sqrt[N_k]{\alpha_{k,v}} \underset{E_{k,
\dot{v}}/F_v}{\sqcup} N_k l_{k,v}(\Lambda_k) \right) \]
as maps $\mathrm{Gal}(E_k/F) \rightarrow T_k(\A_{E_{k+i}})$. Using an argument
similar to the proof of Proposition \ref{p:klocalglobal}, we deduce
\begin{align*}
\mathrm{d} \left( \sqrt[N_k]{\alpha'_k} \underset{E_k/F}{\sqcup} N_k \Lambda_k
\right) =& \prod_{v \in S_k} \mathrm{ES}^2_{R'_{k+i,v}} \left( \mathrm{d} \left(
\sqrt[N_k]{\alpha_{k,v}} \underset{E_{k, \dot{v}}/F_v}{\sqcup} N_k
l_{k,v}(\Lambda_k) \right) \right) \\
=& \prod_{v \in S_k} \mathrm{ES}^2_{R'_{k+i,v}}(\mathrm{loc}_{k,v}(\xi_{k,v}))
\end{align*}
in $Z^2(\mathrm{Gal}(\overline{F}/F), \ker(T_k(\A_{\overline{F}}) \rightarrow
\overline{T}_k(\A_{\overline{F}})))$. This leads us to define
\[ a_k = \frac{\sqrt[N_k]{\alpha_k}}{\sqrt[N_k]{\alpha'_k}}
\underset{E_k/F}{\sqcup} N_k \Lambda_k \in C^1(\mathrm{Gal}(E_k/F),
T_k(\A_{E_{k+i}})). \]
Then
\[ \overline{a_k} = \frac{\alpha_k}{\alpha'_k} \underset{E_k/F}{\sqcup}
\Lambda_k = \mathrm{d}(b_k) \]
where $b_k = \beta_k \underset{E_k/F}{\sqcup} \Lambda_k \in
\overline{T}(\A_{E_k})$.

The fact that $\overline{r}_k(b_{k+1}) = b_k$ for all $k \geq 0$ follows
directly from \eqref{e:betacup} in Lemma \ref{l:cupproductdeltabeta}. The fact
that $r_k(a_{k+1}) = a_k$ for all $k \geq 0$ follows from
$\overline{r}_k(\Lambda_{k+1}) = !_k(\Lambda_k)$ and
\[ a_k = \frac{\sqrt[N_{k+1}]{\alpha_k}}{\sqrt[N_{k+1}]{\alpha'_k}}
\underset{E_k/F}{\sqcup} N_{k+1} \Lambda_k \]
using the same argument as in Lemmas \ref{l:cupproductglobal} and
\ref{l:cupproductdeltabeta}.
\end{proof}

\section{On ramification}
\label{s:ramification}

\subsection{A ramification property}

We deduce a ramification property for Kaletha's generalized Galois cocycles
from our explicit construction. Such a property is important to state Arthur's
multiplicity formula in \cite[\S 4.5]{Kalgri}, namely to guarantee that the
global adèlic packets $\Pi_{\varphi}$ are well-defined: see \cite[Lemma
4.5.1]{Kalgri}.

\begin{prop} \label{p:unrae}
Let $G$ be a connected reductive group over $F$, and $Z$ a finite central
subgroup defined over $F$. For any $z \in Z^1(P \rightarrow \mathcal{E}, Z
\rightarrow G)$, there exists a finite subset $S'$ of $V$ containing all
archimedean places such that for any $v \in V \smallsetminus S'$,
$\mathrm{pr}_{\dot{v}}(z \circ \mathrm{loc}_v)$ is unramified, i.e.\ inflated
from an element of $Z^1(\mathrm{Gal}(K(v)/F_v), G(\mathcal{O}(K(v))))$ for some
finite unramified extension $K(v)/F_v$.
\end{prop}
\begin{proof}
Thanks to \cite[Lemma 3.6.2]{Kalgri} it is enough to prove it in the case where
$G$ is a torus $T$. As usual let $\overline{Y} = X_*(T/Z)$. We remark that this
reduction could force us to enlarge $S'$. Let $N=\exp(Z)$. There exists $k \geq
0$ and a finite $S' \subset V$ containing all places dividing $N$ and $S_k$
such that $z$ is inflated from a unique element of $Z^1(P(E_k, \dot{S}'_{E_k},
N) \rightarrow \mathcal{E}_k(S', N), Z \rightarrow T(\mathcal{O}_{S'}))$, which
we also denote by $z$. By \cite[Proposition 3.7.8]{Kalgri} we can assume that
$z = \iota_k(S', N)(\Lambda)$ for some $\Lambda \in \overline{Y}[S'_{E_k},
\dot{S}'_{E_k}]_0^{N_{E_k/F}}$, up to enlarging $S'$ so that Conditions 3.3.1
in \cite{Kalgri} are satisfied.

For $v \in V \smallsetminus S'$, the morphism $\mathrm{loc}_{k,v}(S', N) :
\mathcal{E}_{k,v}(N) \rightarrow \mathcal{E}_k(S', N)$ is trivial on $P(E_{k,
\dot{v}}, N)$ and so it factors through $\mathrm{Gal}(\overline{F_v}/F_v)$.
Thanks to ramification properties of $\delta_k(N)$ (see Definition
\ref{d:delta}) and by definition of $\eta_{k,v}(S', N)$ (see Definition
\ref{d:eta}), $\eta_{k,v}(S', N) : \mathrm{Gal}(\overline{F_v}/F_v) \rightarrow
P(E_k, \dot{S}'_{E_k}, N)$ factors through
$\mathrm{Gal}(F_v^{\mathrm{nr}}/F_v)$. By construction of $\sqrt[N]{\beta_k}$
in Proposition \ref{p:Nthbeta} and definition of $\kappa_v(\Lambda)$ in
Proposition \ref{p:towerlocalglobal}, $\kappa_v(\Lambda) \in
T(\mathcal{O}(F_v^{\mathrm{nr}}))$. The equality \eqref{e:lociotatower} in
Proposition \ref{p:towerlocalglobal}, which is inflated from \eqref{e:lociotak}
in Proposition \ref{p:klocalglobal}, shows that $\mathrm{pr}_{\dot{v}}(z \circ
\mathrm{loc}_v)$ is unramified.
\end{proof}

Note that it does not seem to be possible to choose $K(v) = K_v$ for some
finite extension $K/F$.

\subsection{Alternative proof}
\label{s:altproof}

As announced in the introduction to this paper, we now give an alternative proof
of Proposition \ref{p:unrae}, which relies solely on Kaletha's definition of the
canonical class, and not on constructions in the present paper.
\begin{proof}[Alternative proof of Proposition \ref{p:unrae}]
For $v \in V$ temporarily let $\xi_v \in Z^2(\mathrm{Gal}(\overline{F}/F), P_v)$
be \emph{any} element of $Z^2(\mathrm{Gal}(\overline{F_v}/F, P_v)$ representing
the class defined in \cite{Kalri}. Choose a tower of resolutions $\left( 1
\rightarrow P_k \rightarrow T_k \rightarrow U_k \rightarrow 1 \right)_{k \geq
0}$ as in \cite[Lemma 3.5.1]{Kalgri}, and following Kaletha let
$T(\overline{\A}) = \varprojlim_k T_k(\overline{\A})$ and $U(\overline{\A}) =
\varprojlim_k U_k(\overline{\A})$. Temporarily let $\xi$ be \emph{any} element
of $Z^2(\mathrm{Gal}(\overline{F}/F), P)$ representing the canonical class
defined in \cite[\S 3.5]{Kalgri}. Of course the $2$-cocycles constructed in
this paper are examples of elements of these cohomology classes, but we want to
emphasize that the present proof does not require constructions in previous
sections.

Let $R'_v = (R'_{k,v})_{k \geq 0}$. The Eckmann-Shapiro maps, for $k,i \geq0$,
\[ \mathrm{ES}^2_{R'_{k+i,v}} : C^2(E_{k+i,\dot{v}} / F_v, P_k) \rightarrow
C^2(E_{k+i}/F, P_k(E_{k+i} \otimes_F F_v) ) \]
are compatible and yield a pro-finite Eckmann-Shapiro map
\[ \mathrm{ES}^2_{R'_v} : C^2(F_v, P) \rightarrow  C^2(F, P(\overline{F}
\otimes_F F_v)) \]
where as in \cite[\S 3.5]{Kalgri}, $P(\overline{F} \otimes_F F_v)):=
\varprojlim_k P_k(\overline{F} \otimes_F F_v)$.  This is explained in
\cite[Appendix B]{Kalgri}, although notations differ: our set of \emph{right}
coset representatives $R'_{k,v}$ corresponds to the image of the composition in
\cite[Lemma B.1, 1.]{Kalgri}, by mapping $r \in R'_{k,v}$ to $r^{-1}$. By
definition of the canonical class there exists $a \in C^1(\A, T)$ and $b \in
U(\overline{\A})$ such that
\[ \xi = \prod_{v \in V} \mathrm{ES}^2_{R'_v}(\mathrm{loc}_v(\xi_v)) \times
\mathrm{d}(a) \]
in $Z^2(\A, T)$ and $\overline{a} = \mathrm{d}(b)$ in $C^1(\A, U)$. In
particular for any $v \in V$ we have
\[ \mathrm{res}_v(\xi) = \mathrm{loc}_v(\xi_v) \times \mathrm{d}(a_v) \]
where $\mathrm{res}_v$ denotes restriction to $\mathrm{Gal}(
\overline{F_v}/F_v)$ and $a_v = \mathrm{pr}_{\dot{v}}( \mathrm{res}_v(a))$. This
equality holds in $Z^2(F_v, T)$, but $\xi$ and $\mathrm{loc}_v(\xi_v)$ both take
values in $P$. Let $b_v = \mathrm{pr}_{\dot{v}}(b)$, and choose a lift
$\tilde{b}_v$ of $b_v$ in $T(\overline{F_v})$. This is possible thanks to the
surjectivity of all maps $P_{k+1} \rightarrow P_k$, by a simple diagram chasing
argument (or more conceptually using vanishing of $\varprojlim^1_k P_k$). Let
$a'_v = a_v / \mathrm{d}(\tilde{b}_v)$. Then $a'_v \in C^1(F_v, P)$, and we have
the equality
\[ \mathrm{res}_v(\xi) = \mathrm{loc}_v(\xi_v) \times \mathrm{d}(a_v') \]
in $Z^2(F_v, P)$.

Fix $k \geq 0$. For $v \in V$ denote by $a_{k,v}$ (resp.\ $b_{k,v}$,
$\tilde{b}_{k,v}$, $a'_{k,v}$) the image of $a_v$ (resp.\ $b_v$, $\tilde{b}_v$,
$a'_v$) in $C^1(F_v, T_k)$ (resp.\ $U_k(\overline{F_v})$, $T_k(\overline{F_v})$,
$C^1(F_v), P_k)$). Let us check that there is a finite set $S'$ of places of $F$
such that for all $v \not\in S'$, $a'_{k,v} \in C^1(F_v, P_k)$ is unramified.
There exists a finite set $S' \supset S_k$ and a finite extension $K$ of $E_k$
unramified outside $S'_{E_k}$ such that $a_k \in C^1(K/F, T_k(\A_K)_{S'})$
and $b_k \in U_k(\A_K)_{S'}$ where $T_k(\A_K)_{S'}$ is defined as $X_*(T_k)
\otimes_{\Z} I(K,S')$. So for $v \not\in S'$, $a_{k,v} \in
C^1(K_{\dot{v}}/F_v, T_k( \mathcal{O}( K_{\dot{v}} ) ))$ is unramified. For any
$v \in V$, $\tilde{b}_{k,v} \in T_k( \mathcal{O}( K_{\dot{v}} )^{(N_k)})$ where
$\mathcal{O}( K_{\dot{v}} )^{(N_k)}$ is the finite étale extension of
$\mathcal{O}( K_{\dot{v}} )$ obtained by adjoining all $N_k$-th roots of
elements in $\mathcal{O}( K_{\dot{v}} )^{\times}$. Here we used that $S_k$
contains all places above $N_k$. We conclude that for $v \not\in S'$, $a'_{k,v}
\in C^1(\mathrm{Gal}(\mathcal{O}( K_{\dot{v}} )^{(N_k)} / \mathcal{O}(F_v)),
P_k)$ and
\[ \mathrm{res}_v(\xi_k) = \mathrm{d}(a'_{k,v}) \]
in $Z^2(F_v, P_k)$, where $\xi_k$ is $\xi$ composed with the surjection $P
\rightarrow P_k$. This easily implies Proposition \ref{p:unrae}.
\end{proof}

Note that the fact that for a fixed $k$, $\mathrm{res}_v(\xi_k)$ is the
coboundary of an unramified $1$-cochain for almost all $v \in V$ is
straighforward from the definition. What the proof above shows is that the
cochain $a'_{k,v}$ coming from ``infinite level'', which is unique up
multiplication by a $1$-coboundary, is unramified for almost all $v \in V$.

\subsection{A non-canonical class failing the ramification property}
\label{s:noncanclass}

\begin{prop} \label{p:examplebadclass}
Assume that $N_1 = 2$ and that $S_1$ is big enough so that $P_1$ is
non-trivial. Then there exists $\xi^{\mathrm{bad}} \in Z^2(F, P)$ which
coincides with the canonical class in $\varprojlim_k H^2(F, P_k)$ and such that
for infinitely many places $v$ of $F$, the $1$-cochain $a_v \in C^1( F_v, P)$
such that $\mathrm{res}_v(\xi^{\mathrm{bad}}) = \mathrm{loc}_v(\xi_v)
\mathrm{d}(a_v)$ is such that its image $a_{1,v} \in C^1(F_v, P_1)$ is
ramified.
\end{prop}
Note that $a_v$ is unique up to a $1$-coboundary by \cite[Proposition
3.4.5]{Kalgri}, and so the property ``$a_{1,v}$ is unramified'' is well-defined
at all places $v \in V \smallsetminus S_1$.
\begin{proof}
Fix a tower of resolutions $(T_k \rightarrow U_k)_{k \geq 0}$ of $P_k$ by tori
as in \cite[\S 3.5]{Kalgri}, and denote by $\pi_k$ the morphism $(T_{k+1}
\rightarrow U_{k+1}) \rightarrow (T_k \rightarrow U_k)$. Recall that the
following short sequences are exact:
\[ 1 \rightarrow \varprojlim_k {}^1 H^1(F, P_k) \rightarrow H^2(F, P)
\rightarrow \varprojlim_k H^2(F, P_k) \rightarrow 1 \]
\[ 1 \rightarrow \varprojlim_k {}^1 H^1(\A, T_k \rightarrow U_k) \rightarrow
H^2(\A, T \rightarrow U) \rightarrow \varprojlim_k H^2(\A, T_k \rightarrow U_k)
\rightarrow 1 \]
and by \cite[Lemma 3.5.3]{Kalgri} the natural morphism
\begin{equation} \label{e:projlim1} \varprojlim_k {}^1 H^1(F,P_k) \rightarrow
\varprojlim_k {}^1 H^1(\A, T_k \rightarrow U_k) \end{equation}
is an isomorphism. So let us first define a non-trivial element of
$\varprojlim_k {}^1 H^1(\A, T_k \rightarrow U_k)$. Choose, for any $k \geq 1$, a
place $v_k \in V \smallsetminus S_1$ such that $E_k/F$ is split above $v_k$
and the $v_k$'s are distinct. For any $k \geq 1$, the tori $T_k$, $U_k$, $T_1$
and $U_1$ are split over $F_{v_k}$, and the surjective morphism of tori $U_k
\rightarrow U_1$ splits over $F_{v_k}$ since it has connected kernel.
Therefore
\[ H^1(F_{v_k}, P_k) = H^1(F_{v_k}, T_k \rightarrow U_k) \simeq
U_k(F_{v_k})/T_k(F_{v_k}) \]
maps onto
\[ H^1(F_{v_k}, P_1) = H^1(F_{v_k}, T_1 \rightarrow U_1) \simeq
U_1(F_{v_k})/T_1(F_{v_k}). \]
Since we have assumed $N_1=2$, over $F_{v_k}$ the multiplicative group $P_1$ is
isomorphic to $\mu_2^r$ for some $r > 1$. For each $k \geq 1$ let $c_k^{(v_k)}
\in Z^1(F_{v_k}, P_k) \subset Z^1(F_{v_k}, T_k \rightarrow U_k)$ be such that
its image in $H^1(F_{v_k}, P_1)$ is ramified. Recall that $H^1(\A, T_k
\rightarrow U_k)$ decomposes as a restricted direct product \cite[Lemma
C.1.B]{KS}, in particular $H^1(F_{v_k}, T_k \rightarrow U_k)$ is a factor of
$H^1(\A, T_k \rightarrow U_k)$. Let $c_k \in Z^1(\A, T_k \rightarrow U_k)$ be an
element of the class defined by $c_k^{(v_k)}$ in $H^1(\A, T_k \rightarrow U_k)$.
Then $(c_k)_{k \geq 0}$ defines an element of $\varprojlim_k^1 H^1(\A, T_k
\rightarrow U_k)$, whose image in $H^2(\A, T \rightarrow U)$ is the class of the
convergent product $\prod_{k \geq 0} \mathrm{d}(\tilde{c}_k)$, where
$\tilde{c}_k \in C^1(\A, T \rightarrow U)$ is any lift of $c_k$.

By surjectivity of \eqref{e:projlim1}, there exists a family $(b_k)_{k \geq 0}$
with $b_k \in Z^1(\A, T_k \rightarrow U_k)$ such that for every $k \geq 1$, the
class of $c_k b_k / \pi_k(b_{k+1})$ belongs to the image of $H^1(F, T_k
\rightarrow U_k) \rightarrow H^1(\A, T_k \rightarrow U_k)$. Up to multiplying
$c_k$ by an element of $B^1(\A, T_k \rightarrow U_k)$, we can assume that for
every $k \geq 1$, $c_k b_k / \pi_k(b_{k+1}) \in Z^1(F, T_k \rightarrow U_k)$.
Choose lifts $\tilde{b}_k \in C^1(\A, T \rightarrow U)$ of $b_k$. Then we can
choose the lifts $\tilde{c}_k$ so that for every $k \geq 0$, $\tilde{c}_k
\tilde{b}_k / \tilde{b}_{k+1} \in C^1(F, T \rightarrow U)$. Let \[ q = \prod_{k
\geq 0} \mathrm{d} \left( \tilde{c}_k \frac{ \tilde{b}_k }{ \tilde{b}_{k+1} }
\right) \in Z^2(F, T \rightarrow U). \] In $Z^2(\A, T \rightarrow U)$, $q$
factors as $\mathrm{d}(\tilde{b}_0) \times \prod_{k \geq 0} \mathrm{d}(
\tilde{c}_k )$. Moreover $q$ defines a class in $H^2(F, T \rightarrow U) =
H^2(F, P)$. Fix $a^{(1)} \in C^1(F, T \rightarrow U)$ such that $q \times
\mathrm{d}(a^{(1)}) \in Z^2(F, P)$.

Let $\xi^{\mathrm{bad}} = \xi \times q \times \mathrm{d}(a^{(1)})$ in $Z^2(F,
P)$, where $\xi \in Z^2(F, P)$ belongs to the canonical class. For any $v \in
V$, there exists $a_v \in C^1(F_v, P)$ such that $\mathrm{res}_v(\xi) =
\mathrm{loc}_v(\xi_v) \times \mathrm{d}(a_v)$. We know a priori that for every
place $v$, $\mathrm{res}_v(q)$ is the trivial class in $H^2(F_v, P)$. The point
of the diagonal construction above is that we can write $\mathrm{res}_v(q)$ more
explicitly as a coboundary. Namely, there exists $a^{(2)} \in C^1(\A, T
\rightarrow U)$ such that for any place $v$,
\[ \mathrm{pr}_{\dot{v}}( \mathrm{res}_v(q) ) = \begin{cases}
\mathrm{d}(a^{(2)}_v) & \text{ if } v \not\in \{v_k | k \geq 1 \}, \\
\mathrm{d}(a^{(2)}_v \times c_k^{(v)}) & \text{ if } v = v_k \end{cases} \]
where $a^{(2)}_v = \mathrm{pr}_{\dot{v}}( \mathrm{res}_v( a^{(2)} ))$. Since
$\xi$ belongs to the canonical class, as in the alternative proof in section
\ref{s:altproof} there exists $a^{(3)} \in C^1(\A, T \rightarrow U)$ such that
for any place $v$, $\mathrm{res}_v(\xi) = \xi_v \times \mathrm{d}(a^{(3)}_v)$.
Therefore, letting $a = a^{(1)} a^{(2)} a^{(3)} \in C^1(\A, T \rightarrow U)$,
we have that for every place $v$,
\[ \mathrm{res}_v(\xi^{\mathrm{bad}}) / \xi_v = \begin{cases} \mathrm{d}(a_v) &
\text{ if } v \not\in \{ v_k | k \geq 1 \} \\ \mathrm{d}(a_v \times c_k^{(v)} )
& \text{ if } v = v_k. \end{cases} \]
By the same argument as in section \ref{s:altproof}, in this equality we can
replace $a_v \in C^1(F_v, T \rightarrow U)$ by $a_v' \in C^1(F_v, P)$, and for
almost all places $v$ the image $a'_{1,v}$ of $a'_v$ in $C^1(F_v, P_1)$ is
unramified. We conclude that for almost all $k \geq 1$,
$\mathrm{res}_{v_k}(\xi^{\mathrm{bad}}) / \xi_{v_k}$ is the coboundary of an
element of $C^1(F_{v_k}, P)$ whose image in $C^1(F_{v_k}, P_1)$ is ramified.
\end{proof}

This example shows that for \cite[Lemma 4.5.1]{Kalgri}, it is important to use
the canonical class to form adèlic packets $\Pi_{\varphi}$, otherwise it could
a priori happen that at infinitely many $v$ places of $F$, the local
representation $\pi_v$ which is the base point of the packet attached to
$\varphi_v$, is not unramified for the canonical hyperspecial maximal compact
subgroup. This seems to be the only place in \cite{Kalgri} where taking $\xi$
to be an arbitrary lift in $H^2(F, P)$ of the canonical element of
$\varprojlim_k H^2(F, P_k)$, rather than the canonical class, could be
problematic.

Note that there are examples of parameters $\varphi$ such that $\varphi_v$ is
unramified endoscopic for infinitely many $v$: by \cite{Elkiesss1} for any
elliptic curve $E$ over $\Q$, the conjectural associated Langlands parameter
$L_{\Q} \rightarrow \mathrm{PGL}_2(\C)$ is such that for infinitely many
unramified primes $p$, $\varphi(\mathrm{Frob}_p)$ is conjugated to
$\mathrm{diag}(-1,1)$. Unconditionally, for any inner form $\widetilde{G}$ of
$\mathrm{GL}_2$ over $\Q$ for which $E$ is relevant, there is a unique
automorphic cuspidal representation for $\widetilde{G}$ corresponding to $E$ (by
modularity of $E$ and Jacquet-Langlands). Restricting to the derived subgroup
$G$ of $\widetilde{G}$, we obtain a space of cuspidal automorphic forms for $G$
of infinite length. Interestingly, the algorithm in \cite{Elkiesss1} uses primes
which do \emph{not} split in certain quadratic extensions of $\Q$, while the
counter-example in \ref{p:examplebadclass} is constructed using primes split in
arbitrarily large extensions of the base field.

\section{Effective localization}
\label{s:effloc}

We conclude by explaining how the constructive proof of the existence of a
family of ``local-global compatibility'' cochains $(\beta_k)_{k \geq 0}$ at the
end of section \ref{s:Tatecocycles} allows one to explicitly compute all
localizations of a global rigidifying datum, as promised in the introduction to
this article.

\subsection{A general procedure}

Let $G^*$ be a quasi-split connected reductive group over $F$. Fix a global
Whittaker datum $\mathfrak{w}$ of $G^*$, i.e.\ choose a Borel subgroup $B^*$ of
$G$ defined over $F$, let $U$ be the unipotent radical of $B^*$, let $\chi$ be a
generic unitary character of $U(\A)/U(F)$, and let $\mathfrak{w}$ be the
$G(F)$-conjugacy class of $(B^*, \chi)$.  Let $T$ a maximal torus of $G^*$
defined over $F$, and $E$ a finite Galois extension of $F$ splitting $E$. Let
$S$ be a finite set of places of $F$ such that
\begin{enumerate}
\item $S$ contains all archimedean places of $F$ and all places of $F$ which
ramify in $E$, and the (always injective) morphism
$I(E,S)/\mathcal{O}(E,S)^{\times} \rightarrow C(E)$ is surjective (i.e.\
$\mathrm{Pic}(\mathcal{O}(E,S))=1$).
\item $G^*$ admits a reductive model $\underline{G^*}$ over $\mathcal{O}(F,S)$
in the sense of \cite[Exposé XIX, Définition 2.7]{SGA3-III} such that the
schematic closure $\underline{T}$ of $T$ in $\underline{G^*}$, which is a flat
group scheme over $\mathcal{O}(F,S)$ since this ring is Dedekind, is a torus in
the sense of \cite[Exposé IX, Définition 1.3]{SGA3-II}.
\item For any $v \not\in S$, the Whittaker datum $\mathfrak{w}$ is compatible
with the $G^*(F_v)$-conjugacy class of the hyperspecial maximal compact subgroup
$\underline{G}^*(\mathcal{O}(F_v))$, in the sense of \cite{CasSha}.
\end{enumerate}
Let $Z$ be a finite central subgroup of $G$, $N=\exp(Z)$ and $\overline{T} =
T/Z$. Let $\underline{Z}$ be the schematic closure of $Z$ in $\underline{T}$
(or $\underline{G}$), then $\underline{Z}$ is a group scheme of multiplicative
type over $\mathcal{O}(F,S)$. Moreover $\underline{\overline{T}} :=
\underline{T} / \underline{Z}$ is a maximal torus of the reductive group scheme
$\underline{G^*} / \underline{Z}$ (see \cite[Exposé XXII, Corollaire
4.3.2]{SGA3-III}). Let $\dot{S}_E$ be a set of representatives for the action
of $\mathrm{Gal}(E/F)$ on $S_E$. Finally, choose $\Lambda \in \overline{Y}[S_E,
\dot{S}_E]_0^{N_{E/F}}$. If
\[ \alpha_{E/F} \in Z^2\left(\mathrm{Gal}(E/F), \mathrm{Hom}(\Z[S_E]_0,
\mathcal{O}(E,S))^{\times}\right) \]
is any Tate cocycle (as in \cite{Tatecohotori}), then taking the cup-product of
$\alpha_{E/F}$ with $\Lambda$ yields 
\begin{equation} \label{e:zbarforeff}
\overline{z} \in Z^1(\mathrm{Gal}(\mathcal{O}(E,S)/\mathcal{O}(F,S)),
\underline{\overline{T}}(\mathcal{O}(E,S)))
\end{equation}
i.e.\ a \v{C}ech cocycle for the étale sheaf $\underline{\overline{T}}$ and the
covering $\mathrm{Spec}(\mathcal{O}(E,S)) \rightarrow
\mathrm{Spec}(\mathcal{O}(F,S))$. In particular we obtain a reductive group
$\underline{G}$ over $\mathcal{O}(F,S)$ by twisting $\underline{G^*}$ with the
image $\overline{\overline{z}}$ of $\overline{z}$ in
\[ Z^1 \left( \mathrm{Gal}(\mathcal{O}(E,S)/\mathcal{O}(F,S)),
\underline{G^*}_{\mathrm{ad}}(\mathcal{O}(E,S)) \right). \]
This realizes the generic fiber $G$ of $\underline{G}$ as an inner twist
$(\Xi, \overline{\overline{z}})$ of $G^*$.
\begin{rema}
The fact that any connected reductive group $G$ over $F$ arises in this way is a
consequence of \cite[Lemmas A.1 and 3.6.1]{Kalgri}.

More directly, that is without making use of \cite[Lemma A.1]{Kalgri},
Steinberg's theorem on rational conjugacy classes in quasi-split semisimple
simply connected algebraic groups (\cite{Steinberg}) implies that if we start
with a reductive group $G$ and a maximal torus $T$ of $G$, then it can be
realized as an inner twist $(G^*, \Xi, \overline{\overline{z}})$ with
$\overline{\overline{z}}$ taking values in
$\Xi^{-1}(T_{\mathrm{ad}}(\overline{F}))$.
\end{rema}

We now use the constructive proof of Theorem \ref{t:constructionTate} at the
end of section \ref{s:Tatecocycles}. Let $E_1 = E$ and $S_1 = S$ and choose a
finite Galois extension $E_2$ of $F$ which is totally complex and such that for
every $v \in S$ non-archimedean,
\[ N_{E_2/E}\left( \prod_{w | v} \mathcal{O}(E_{2,w})^{\times} \right) \]
is contained in the subgroup of $N$-th powers in $\prod_{w | v}
\mathcal{O}(E_w)^{\times}$. Finally, let $E_3$ be any finite Galois extension of
$F$ containing the Hilbert class field of $E_2$. Choose global fundamental
classes $\overline{\alpha}_1$, $\overline{\alpha}_2$, $\overline{\alpha}_3$ such
that $\overline{\alpha}_k = \mathrm{AW}^2_k(\overline{\alpha}_{k+1})$ for $k \in
\{1,2\}$ and $\overline{\alpha}_3$ is normalized, i.e.\
$\overline{\alpha}_3(1,1)=1$. Fix finite sets of places $S_3 \supset S_2 \supset
S$ as in section \ref{s:notation}. For each $v \in S_3$ fix a place $\dot{v}_3
\in S_{E_3}$. Choose local fundamental classes $\alpha_{k,v}$ for $v \in S$ and
$k \in \{1,2,3\}$. Choose sets of representatives $(R_{k,v})_{1 \leq k \leq 3, v
\in S}$ as in section \ref{s:locfundcocyc}, or rather, choose their image
$\overline{R}_{k,v}$ in $\mathrm{Gal}(E_3/F)$. These families $(S_k)_{k \leq
3}$, $(\overline{\alpha}_k)_{k \leq 3}$, $(\alpha_{k,v})_{k \leq 3, v \in S}$,
$(\overline{R}_{k,v})_{k \leq 3, v \in S}$ can be extended to $k \geq 0$ and $v
\in V$, as explained in sections \ref{s:globfundcocy}, \ref{s:locfundcocyc} and
\ref{s:Tatecocycles}. Moreover $\{ \dot{v}_3 \}_{v \in S}$ can be lifted and
extended to yield $\dot{V}$ as in section \ref{s:notation}.

Now choose $\overline{\beta}^{(0)}_3 : \mathrm{Gal}(E_3/F)
\rightarrow \mathrm{Maps}(S_{E_3}, C(E_3))$ such that
$\mathrm{d}(\overline{\beta}^{(0)}_3) = \overline{\alpha}_3 /
\overline{\alpha_3'}$. Choose $\beta^{(1)}_2 : \mathrm{Gal}(E_2/F) \rightarrow
\mathrm{Maps}(S_{E_2}, I(E_2, S_2))$ lifting
$\mathrm{AWES}_2^1(\overline{\beta}^{(0)}_3)$ such that $\beta^{(1)}_2(1) = 1$
and $\beta^{(2)}_1 := \mathrm{AWES}_1^1 \left( \beta^{(1)}_2 \right)$ takes
values in $\mathrm{Maps}(S_{E_1}, I(E, S))$. Let $\alpha_1 = \alpha'_1 \times
\mathrm{d}(\beta^{(2)}_1)$.  At the end of section \ref{s:Tatecocycles} we
constructed a family $(\beta_k)_{k \geq 0}$ such that there exists $\zeta_2 \in
\mathrm{Maps}(S_{E_2}, \widehat{\mathcal{O}(E_2)}^{\times})$ satisfying
$\beta_2|_{S_{E_2}} = \beta^{(1)}_2 \times \mathrm{d}(\zeta_2)$ and so
$\beta_1|_{S_E} = \mathrm{AWES}^1_1(\beta_2) = \beta^{(2)}_1 \times \mathrm{d}
(x)$ where $x=\mathrm{AWES}^0_1(\zeta_2)$ is a map
\[ S_E \rightarrow N_{E_2/E}\left( \widehat{\mathcal{O}(E_2)}^{\times} \right). \]
In particular for every non-archimedean $v \in S$ there exists a map $y_v : S_E
\rightarrow \prod_{w|v} \mathcal{O}(E_w)^{\times}$ such that $y_v^N =
\mathrm{pr}_v(x)$. For $v \in S$ archimedean, simply let $y_v=1$.
Recall that $N = \exp(Z)$.
Going back to the construction of $N'$-th roots in Propositions
\ref{p:Nthlocal}, \ref{p:NthTate} and \ref{p:Nthbeta}, we see that
for any choice of $N$-th root $\sqrt[N]{\beta_1^{(2)}} : \mathrm{Gal}(E/F)
\rightarrow \mathrm{Maps}(S_E, \mathcal{I}(E, S \cup N))$, we can choose the
$N$-th root $\sqrt[N]{\beta_1}$ so that for all $v \in S$,
\[ \mathrm{pr}_v\left(\sqrt[N]{\beta_1}\right)|_{S_E} =
\mathrm{pr}_v\left(\sqrt[N]{\beta^{(2)}_1}\right) \times
\mathrm{d}\left(y_v\right). \]

If $\alpha_1$ is chosen to form $\overline{z}$ in \eqref{e:zbarforeff}, the
generic fiber $G$ of $\underline{G}$ is endowed with a global rigidifying datum
$(G^*, \Xi, z, \mathfrak{w})$ where $z = \iota(\Lambda)$. For $v \in V$, the
localization of this rigidifying datum at $v$ is $(G^*_{F_v}, \Xi_v, z_v,
\mathfrak{w}_v)$ where $\Xi_v = \Xi_{\overline{F_v}}$ and $z_v =
\mathrm{pr}_{\dot{v}}(z \circ \mathrm{loc}_v)$.

Let $z'_v = \iota_v(l_v(\Lambda))$ and fix a rigid inner twist $(G'_v, \Xi'_v)$
of $G^*_{F_v}$ by $z'_v$, which is well-defined up to conjugation by
$G'_v(F_v)$ (see \cite[Fact 5.1]{Kalri}).  We now compare the rigid inner twists
$(G_{F_v}, \Xi_v)$ and $(G'_v, \Xi'_v)$ of $G^*_{F_v}$. Recall (Proposition
\ref{p:towerlocalglobal}) that
\[ \mathrm{pr}_{\dot{v}}(z \circ \mathrm{loc}_v) = \iota_v(l_v(\Lambda)) \times
\mathrm{d}(\kappa_v(\Lambda)) \]
where $\kappa_v(\Lambda) = \mathrm{pr}_{\dot{v}} \left( \sqrt[N]{\beta_1}
\right) \underset{E/F}{\sqcup} N \Lambda \in T(\overline{F_v})$. Therefore we
have an isomorphism of rigid inner twists of $G^*_{F_v}$
\[ \left( f_v, \kappa_v(\Lambda) \right) : (G_{F_v}, \Xi_v, z_v)
\xrightarrow{\sim} (G'_v, \Xi'_v, z'_v) \]
where $f_v$ is obtained from $\Xi'_v \circ \mathrm{Ad}(\kappa_v(\Lambda)) \circ
\Xi_v^{-1}$ by Galois descent. Thus $f_v : G_{F_v} \simeq G^*_{F_v}$ identifies
the rigidifying datum $(G^*_{F_v}, \Xi_v, z_v, \mathfrak{w}_v)$ for $G_{F_v}$
with the rigidifying datum $(G^*_{F_v}, \Xi'_v, z'_v, \mathfrak{w}_v)$ for
$G'_v$.
\begin{itemize}
\item For $v \in V \smallsetminus S$, $l_v(\Lambda_v)=0$ and we can simply take
$G'_v = G^*_{F_v}$ and $\Xi'_v = \mathrm{Id}$. In particular $G_{F_v}$ is
quasi-split and we can simply take as rigidifying datum the pull-back
$f_v^*(\mathfrak{w}_v)$ of the Whittaker datum $\mathfrak{w}_v$. The image
$\overline{\kappa}_v(\Lambda)$ of $\kappa_v(\Lambda)$ in
$\overline{T}(\overline{F_v})$ equals
\[ \mathrm{pr}_{\dot{v}} (\beta_1) \underset{E/F}{\sqcup} \Lambda \in
\underline{\overline{T}}(\mathcal{O}(E_{\dot{v}})) \]
and so $\mathrm{Ad}(\kappa_v(\Lambda))$ is an automorphism of the reductive
group scheme $\underline{G^*}_{\mathcal{O}(E_{\dot{v}})}$. Since $\Xi_v$ is
obtained as the generic fiber of an isomorphism
$\underline{G^*}_{\mathcal{O}(E_{\dot{v}})} \simeq
\underline{G}_{\mathcal{O}(E_{\dot{v}})}$, we see that $f_v$ descends from an
isomorphism $\underline{G}_{\mathcal{O}(E_{\dot{v}})} \simeq
\underline{G^*}_{\mathcal{O}(E_{\dot{v}})}$ and so $f_v$ can be extended to an
isomorphism of reductive models $\underline{G}_{\mathcal{O}(F_v)} \simeq
\underline{G^*}_{\mathcal{O}(F_v)}$. This shows that $f_v^*(\mathfrak{w}_v)$ is
compatible with the $G(F_v)$-conjugacy class of hyperspecial maximal compact
subgroups represented by $\underline{G}(\mathcal{O}(F_v))$. Note that this holds
even for $v \not\in S$ dividing $N$.
\item For $v \in S$, one can compute the element $\kappa_v(\Lambda)$ up to an
element of $T(F_v)$, since
\[ \mathrm{d}\left(y_v\right) \underset{E/F}{\sqcup} N \Lambda = N_{E/F} \left(
y_v \underset{E/F}{\sqcup} N \Lambda \right) \in T(F_v) \]
and so $\mathrm{d}(\kappa_v(\Lambda)) = \mathrm{d} (\kappa_v'(\Lambda))$ where
\[ \kappa_v'(\Lambda) = \mathrm{pr}_{\dot{v}} \left( \sqrt[N]{\beta_1^{(2)}}
\right) \underset{E/F}{\sqcup} N \Lambda \]
is computable. Thus $(f_v, \kappa'_v(\Lambda))$ is also an isomorphism of rigid
inner twists of $G^*_{F_v}$.  Note that to compute $f_v$ it is enough to
compute the image of $\kappa_v'(\Lambda)$ in $\overline{T}(\overline{F_v})$,
i.e.\
\[ \mathrm{pr}_{\dot{v}} \left( \beta_1^{(2)} \right) \underset{E/F}{\sqcup}
\Lambda \in \overline{T}(E_{\dot{v}}) \]
and so in practice it is not necessary to compute an $N$-th rooth of
$\beta_1^{(2)}$.
\end{itemize}

\subsection{A simple example}
\label{s:example}

Let us illustrate this on a simple example, where almost no computation of
cocycles is needed.

\subsubsection{Definition of the group $G$}

Let $F = \Q(s)$ with $s^2=3$.
Let $D$ be a quaternion algebra over $F$ such that $D$ is definite at both real
places of $F$, and split at all non-archimedean places of $F$.  Let $N_D \in
\mathrm{Sym}^2(D^*)$ be the reduced norm, and $G$ the reductive group scheme
over $F$ defined by
\[ G(R) = \left\{ x \in R \otimes_F D \,|\,N_D(x) = 1 \text{ in } R \right\} \]
for any $F$-algebra $R$.

\subsubsection{A reductive model of $G$}

The class group of $F$ is trivial, and the narrow class group of $F$ is
$\Z/2\Z$, corresponding to the totally complex and everywhere unramified
extension $E = F(\zeta)$ of $F$, where $\zeta^2 - s \zeta + 1 = 0$ ($\zeta$ is a
primitive $12$-th root of unity). The class group of $E$ is also trivial. Write
$\sigma$ for the non-trivial $\mathcal{O}(F)$-automorphism of $\mathcal{O}(E)$.
Let $S$ be the set of real places of $F$, so that $S= \{v_+, v_-\}$ where the
image of $s$ in $F_{v_+} = \R$ is positive. We still denote by $v_+$, $v_-$ the
unique complex places of $E$ above $v_+$, $v_-$. The group
$\mathcal{O}(E)^{\times}$ is generated by $\zeta$ and $\zeta-1$, which has
infinite order. The group $\mathcal{O}(F)^{\times}$ is generated by $-1$ and
$2-s = N_{E/F}(\zeta-1)$, which has infinite order.

Let $\underline{G^*} = \mathrm{SL}_2$ over $\mathcal{O}(F)$ and let
$\underline{T} \subset \underline{G^*}$ be the torus defined by
\[ \underline{T}(R) = \left\{ \begin{pmatrix} x & -y \\ y & x + sy \end{pmatrix}
\, \middle| \, x,y \in R,\, x^2+sxy+y^2=1 \right\} \]
for any $\mathcal{O}(F)$-algebra $R$. Then $\underline{T}$ splits over
$\mathcal{O}(E)$. Let $\underline{Z} \simeq \mu_2 $ be the center of
$\underline{G^*}$ and $\underline{\overline{T}} = \underline{T} /
\underline{Z}$. The element $(x=s, y=-2) \in \underline{T}(\mathcal{O}_F)$ maps
to the unique element of order $2$ in $\overline{T}(F)$, and so we have a
$1$-cocyle
\[\overline{z} : \sigma \mapsto \overline{(x=s, y=-2)} \in
\mathrm{PGL}_2(\mathcal{O}(F)). \]
Since $\mathrm{PGL}_2$ is also the automorphism group of the matrix algebra
$\mathrm{M}_2$, we obtain an Azumaya algebra $\mathcal{O}(D)$ over
$\mathcal{O}(F)$ by twisting $\mathrm{M}_2(\mathcal{O}(F))$ using
$\overline{z}$. Explicitly, it has basis $(1, Z, I, ZI)$ over $\mathcal{O}(F)$,
where
\[ Z = \begin{pmatrix} 0 & -1 \\ 1 & s \end{pmatrix},\ I = \begin{pmatrix} 0 & 2
\zeta - s \\ 2 \zeta - s & 0 \end{pmatrix}. \]
We have $Z^{12}=1$ and $I^2=-1$. Let $D = F \otimes_{\mathcal{O}(F)}
\mathcal{O}(D)$. Let $\underline{G}$ be the inner twist of
$\underline{G^*}$ by $\overline{z}$, so that
\[ \underline{G}(R) = \left\{ x \in R \otimes_{\mathcal{O}(F)} \mathcal{O}(D)
\middle| N_D(x) =1 \right\} \]
for any $\mathcal{O}(F)$-algebra $R$.

\subsubsection{The group $G$ as a rigid inner twist}

If we identify $Y = X_*(T)$ with $\Z$,
then $\overline{Y} = X_*(\overline{T})$ is identified with $\frac{1}{2}\Z$. Let
$\Lambda \in \overline{Y}[\dot{S}_E]_0^{N_{E/F}}$ be defined by $\Lambda(v_+) =
1/2$ and $\Lambda(v_-) =-1/2$. An easy computation shows that one can choose the
Tate cocycle $\alpha_1$ for $E/F$ such that
\[ \alpha_1(\sigma,\sigma)(v_+)/\alpha_1(\sigma,\sigma)(v_-) = -1 \]
and so $\overline{z} = \alpha_1 \underset{E/F}{\sqcup} \Lambda$. Using $z =
\iota(\Lambda)$, we obtain a realization of $G$ as a rigid inner twist $(\Xi,
z)$ of $G^*$.

\subsubsection{Choice(s) of Whittaker data}

Let $\psi$ be the unitary character of $\A_{\Q}/\Q$ such that $\psi_{\infty}(x)
= \exp(2i \pi x)$, so that for every prime $p$ we have $\ker(\psi_p) = \Z_p$.
Fortunately the different ideal of $F/\Q$ is principal, generated by $2s$, and
so for any choice of sign the global Whittaker datum $\mathfrak{w}$ for $G^*$
\begin{equation} \label{e:defwhittaker}
\begin{pmatrix} 1 & x \\ 0 & 1 \end{pmatrix} \in U(\A_F) \mapsto
\psi(\pm \mathrm{Tr}_{F/\Q}(x/(2s)))
\end{equation}
is compatible with the model $\underline{G^*}_{\mathcal{O}(F_v)}$ at
\emph{every} finite place $v$ of $F$. Therefore the global rigidifying datum
$\mathcal{D} = (G^*, \Xi, z, \mathfrak{w})$ for $G$ is such that for any finite
place $v$ of $F$, the localization $\mathcal{D}_v$ is unramified and compatible
with the $G(F_v)$-conjugacy class of hyperspecial maximal compact subgroups
$\underline{G}(\mathcal{O}(F_v))$.

\subsubsection{Real places}

At any real place $v$ of $F$, we could compute explicit coboundaries
expressing local-global compatibility, but this is not necessary since the
parametrization of Arthur-Langlands packets for the compact groups $G(F_v)
\simeq \mathrm{SU}(2)$ is simply determined by the Whittaker datum
$\mathfrak{w}_v$ and the \emph{cohomology class} of $z_v$ in $H^1(P_v
\rightarrow \mathcal{E}, Z \rightarrow T)$ (see \cite[\S 5.6]{Kalri} and
\cite[\S 3.2]{TaiMult}), which only depends on $l_v(\Lambda)$. This
simplification is particular to anisotropic real groups, for which Langlands
packets have at most one element.

In order to formulate the local Langlands correspondence at each real place $v$
of $F$ it is necessary to identify an algebraic closure of the base field $F_v$,
occurring in the definition of the Weil group $W_{F_v}$, with the coefficient
field $\C$. We have natural algebraic closures $E_{v_+}$ and $E_{v_-}$ of
$F_{v_+}$ and $F_{v_-}$. Choose $\tau_+ : \zeta \mapsto \exp(2 i \pi /12)$
(resp.\ $\tau_- : \zeta \mapsto \exp(5 \times 2 i \pi / 12)$) identifying
$E_{v_+}$ (resp.\ $E_{v_-}$) with $\C$. There is a natural identification
$\theta_+$ (resp.\ $\theta_-$) of $G^*_{F_{v_+}}$ (resp.\ $G^*_{F_{v_-}}$) with
the usual split group $\mathrm{SL}_2$ over $\R$, compatibly with the canonical
isomorphisms $F_{v_+} = \R$ and $F_{v_-} = \R$.  Let $\sqrt{3}$ be the positive
square root of $3$ in $\R$, so that $\tau_+(s) = \sqrt{3}$ and $\tau_-(s) =
-\sqrt{3}$. In particular for any choice of sign in \eqref{e:defwhittaker}, the
Whittaker data $(\theta_+)_*(\mathfrak{w}_{v_+})$ and
$(\theta_-)_*(\mathfrak{w}_{v_-})$ differ. Associated to $\mathfrak{w}_+$ is a
Borel subgroup $B_+$ of $G^*_{F_{v_+}} \times_{F_{v_+}} \C$ containing
$T_{F_{v_+}} \times_{F_{v_+}} \C$ (see \cite{TaiMult}), corresponding to the
generic discrete series representations of $G^*(F_{v_+})$. Using $\tau_+$ we see
$B_+$ as a Borel subgroup of $G^*_{E_{v_+}}$, and since $T$ is defined over $F$
and split over $E$ we see that $B_+$ comes from a well-defined Borel subgroup of
$G^*_E$ containing $T_E$, which we still denote by $B_+$. Similarly, we have a
Borel subgroup $B_-$ of $G^*_E$ containing $T_E$. Up to changing the sign in
\eqref{e:defwhittaker}, we can assume that $B_+$ is such that the unique root of
$T_E$ in $B_+$ is $\alpha_+ : (x,y) \mapsto (x+ \zeta y)^2$. Let us determine
$B_-$ using $\theta_+$ and $\theta_-$. For this we need to conjugate
$\theta_+(T_{F_{v_+}})$ and $\theta_-(T_{F_{v_-}})$ by an element of
$\mathrm{SL}_2(\R)$. The matrix
\[ g = \begin{pmatrix} 1 & -\sqrt{3} \\ 0 & 1 \end{pmatrix} \in
\mathrm{SL}_2(\R) \]
conjugates $\theta_-(T_{F_{v_-}})$ into $\theta_+(T_{F_{v_+}})$, mapping
$\theta_-(x,y)$ to $\theta_+(x-\sqrt{3}y,y)$. Since
$(\theta_+)_*(\mathfrak{w}_+)$ and $(\theta_-)_*(\mathfrak{w}_-)$ differ, the
root $\alpha_-$ of $T_E$ in $B_-$ is \emph{not} equal to
\[ (\tau_-)^{-1} \circ \tau_+ \circ \alpha_+ \circ \tau_+^{-1} \circ
(\theta_+)_{\C}^{-1} \circ \mathrm{Ad}(g) \circ (\theta_-)_{\C} \circ \tau_-, \]
which equals $\alpha_+$. Therefore $\alpha_- \neq \alpha_+$ and $B_- \neq B_+$.
Note that other choices for $\tau_+, \tau_-$ would lead to other Borel
subgroups, and some choices would give equal Borel subgroups.

Let us now consider Arthur-Langlands packets of unitary representations of
$G(F_{v_+})$ and $G(F_{v_-})$. We refer to \cite[\S 3.2.2]{TaiMult} for the
parametrization of ``cohomological'' Arthur-Langlands packets for inner forms of
symplectic or special orthogonal groups, following Shelstad, Adams-Johnson and
Kaletha. The present case is much simpler. Note also that since $G(F_{v_+})$ and
$G(F_{v_-})$ are anisotropic, any non-empty Arthur-Langlands packet is
``cohomological'', i.e.\ is a packet of Adams-Johnson representations.  For $v
\in \{v_+,v_-\}$ there is only one Arthur-Langlands parameter
\[ W_{F_v} \times \mathrm{SL}_2(\C) \rightarrow {}^L G \]
which is non-trivial on $\mathrm{SL}_2(\C)$ and yields a non-empty packet,
namely the principal representation
\[ \mathrm{SL}_2(\C) \rightarrow \widehat{G} \simeq \mathrm{PGL}_2(\C), \]
with corresponding packet containing the trivial representation with
multiplicity one.
Any other Arthur-Langlands parameter yielding a non-empty packet of
representations is tempered and discrete, and so up to conjugation by
$\widehat{G}$ it is of the form
\begin{alignat*}{2}
\varphi_{k_+} : W_{F_{v_+}} & \longrightarrow & \mathrm{PGL}_2(\C) \\
z \in E_{v_+}^{\times} & \longmapsto & \begin{pmatrix} \tau_+(z/\bar{z})^{k_++1}
& 0 \\ 0 & 1 \end{pmatrix} \\ j & \longmapsto & \begin{pmatrix} 0 & 1 \\ 1 & 0
\end{pmatrix}
\end{alignat*}
for some $k_+ \in \Z_{\geq 0}$, and similarly discrete tempered parameters for
$G_{F_{v_-}}$ are parametrized by integers $k_- \geq 0$, using $\tau_-$. Above
$j$ is any element of $W_{F_{v_+}} \smallsetminus E_{v_+}^{\times}$ such that
$j^2=-1$. Note that we have put $\varphi_{k_+}$ in dominant form for the
upper-triangular Borel subgroup $\mathcal{B}$ of $\widehat{G}$. Using $B_+$ we
have an identification between the group $\mathcal{T}$ of diagonal matrices in
$\mathrm{PGL}_2(\C)$ and $\widehat{T} = X^*(T) \otimes_{\Z} \C^{\times}$. So we
can identify $l_{v_+}(\Lambda) = \Lambda(v_+) \in X_*(\overline{T})^{N_{E/F}}$
with an element of $X^*(\overline{\mathcal{T}})$, where $\overline{\mathcal{T}}$
is the preimage of $\mathcal{T}$ in $\widehat{\overline{G}} =
\mathrm{SL}_2(\C)$. The preimage $\mathcal{S}_{\varphi_{k_+}}^+$ of
$\mathcal{S}_{\varphi_{k_+}} = \mathrm{Cent}(\varphi_{k_+}, \widehat{G})$ in
$\widehat{\overline{G}}$ has $4$ elements and is generated by
\[ \begin{pmatrix} i & 0 \\ 0 & -i \end{pmatrix} \in \overline{\mathcal{T}}. \]
The class of $l_{v_+}(\Lambda)$ modulo $(1-\sigma)X_*(T)$ defines a character
of $\mathcal{S}_{\varphi_{k_+}}^+$. There is a unique element $\pi_{v_+, k_+}$
in the Arthur-Langlands packet attached to (the $\widehat{G}$-conjugacy class
of) $\varphi_{k_+}$, that is the unique irreducible representation of
$G(F_{v_+})$ in dimension $k_++1$. The character $\langle \cdot, \pi_{v_+, k_+}
\rangle$ of $\mathcal{S}_{\varphi_{k_+}}^+$ is the one defined by
$l_v(\Lambda)$.

Similarly, each discrete series L-packet for $G_{F_{v_-}}$ has a unique element
$\pi_{v_-, k_-}$, and a character $\langle \cdot, \pi_{v_-, k_-} \rangle$ of
$\mathcal{S}_{\varphi_{k_-}}^+$ coming from the character $l_{v_-}(\Lambda) =
\Lambda(v_-)$ of $\overline{\mathcal{T}}$. Note that since $B_-$ and $B_+$
differ and $\Lambda(v_-) = -\Lambda(v_+)$, the characters of
$\overline{\mathcal{T}}$ corresponding to $\Lambda(v_+)$ and $\Lambda(v_-)$ are
equal.

\subsubsection{Automorphic representations}

To lighten notation we let $K = \underline{G}( \widehat{ \mathcal{O}(F) })$. We
can now formulate precisely the endoscopic decomposition of the space of $G(\R
\otimes_{\Q} F)$-finite functions on $G(F) \backslash G(\A_F) / K$, with
commuting actions of $G(\R \otimes_{\Q} F)$ and of the Hecke algebra in level
$K$. Let $V_+$ (resp.\ $V_-$) be the irreducible representation of $G(F_{v_+})$
(resp.\ $G(F_{v_-})$) of dimension $k_+ + 1$ (resp.\ $k_-+1$). Recall
\cite{GrossAlg} that we can cut out the $V_+ \otimes V_-$-isotypical subspace
inside the space of all automorphic forms for $G$, and define the space $M_{k_+,
k_-}(K)$ of automorphic forms of weight $(k_+, k_-)$ and level $K$ as the space
of $G(F)$-equivariant functions
\[ G(\mathbb{A}_{F,f}) / K \rightarrow V_+ \otimes V_- , \]
which is a finite-dimensional vector space over $\mathbb{C}$ endowed with a
semi-simple action of the commutative Hecke algebra in level $K$. Moreover it is
easy to check that $M_{k_+,k_-}(K)$ has a natural $E$-structure.

The automorphic multiplicity formula for $\mathrm{SL}_2$ and its inner forms was
proved in \cite{LabLan}, although at the time there was no general definition of
transfer factors, let alone Kaletha's normalization of transfer factors for
inner forms. Formally we can use the main result of \cite{TaiMult}, but of
course a careful reading of \cite{LabLan} and a comparison of transfer factors
with the later definition in \cite{LanShe} and \cite{Kalri}, \cite{Kalgri}
should give a more direct proof. In the present case, automorphic
representations for $G$ in level $K$ fall into three categories:
\begin{itemize}
\item the trivial representation,
\item representations corresponding to self-dual automorphic cuspidal
representations of $\mathrm{PGL}_3/F$ which are algebraic regular at both
infinite places and unramified at all finite places,
\item representations ``automorphically induced'' from certain algebraic Hecke
characters for $E$.
\end{itemize}
The multiplicity formula is non-trivial only in the third case. Making it
explicit allows one to enumerate representations in the (most interesting)
second case.

\subsubsection{Global endoscopic parameters}

Let $\chi : C(E) \rightarrow \C^{\times}$ be a continuous unitary character
which is trivial on $C(F) = C(E)^{\mathrm{Gal}(E/F)}$. In particular,
$\chi^{\sigma} = \chi^{-1}$. Using $\chi$ we can form the parameter
\begin{alignat*}{2}
\varphi_{\chi} : W_{E/F} & \longrightarrow & \mathrm{PGL}_2(\C) \\
z \in C(E) & \longmapsto & \begin{pmatrix} \chi(z) & 0
\\ 0 & 1 \end{pmatrix} \\
\tilde{\sigma} & \longmapsto & \begin{pmatrix} 0 & 1 \\ 1 & 0 \end{pmatrix}
\end{alignat*}
where $\tilde{\sigma} \in W_{E/F}$ is any lift of $\sigma \in
\mathrm{Gal}(E/F)$. The parameters $\varphi_{\chi}$ and $\varphi_{\chi^{-1}}$
are conjugated by $\mathrm{PGL}_2(\C)$. We only consider characters $\chi$ such
that the restriction of $\varphi_{\chi}$ to the Weil groups at both real places
of $F$ are discrete, i.e.\ we impose that $\chi_{v_+} =
\chi|_{E_{v_+}^{\times}}$ and $\chi_{v_-} = \chi|_{E_{v_-}^{\times}}$ are
non-trivial. Therefore there are $a_+,a_- \in \Z \smallsetminus \{0\}$ such that
\[ \chi_{v_+}(z) = \tau_+(z/\bar{z})^{a_+},\ \ \chi_{v_-}(z) =
\tau_-(z/\bar{z})^{a_-}. \]
Moreover we impose that $\chi$ is everywhere unramified, i.e.\ at every finite
place $w$ of $E$, $\chi_w$ is trivial on $\mathcal{O}(E_w)^{\times}$. Since $E$
has class number $1$ the map
\[ E_{v_+}^{\times} \times E_{v_-}^{\times} \times \prod_{w \text{ finite}}
\mathcal{O}(E_w)^{\times} \rightarrow C(E) \]
is surjective, and its kernel is $\mathcal{O}(E)^{\times}$. Thus for $a_+, a_-
\in \Z \smallsetminus \{0\}$ there is at most one everywhere unramified $\chi$
as above, and there exists one if and only if $\chi_{v_+} \times \chi_{v_-}$ is
trivial on $\mathcal{O}(E)^{\times}$, which is generated by $\zeta$ and
$\zeta-1$. A simple computation shows that this is equivalent to
\[ a_+ + 5a_- = 0 \mod 12. \]
For such a character $\chi$, at a finite place $w$ of $E$ we have:
\begin{itemize}
\item If $w$ is fixed by $\sigma$ (inert case), then there is a uniformizer
$\varpi_w \in \mathcal{O}(F)$, and so $\chi_w$ is trivial.
\item If $w$ is not fixed by $\sigma$ (split case), then if $\varpi_w \in
\mathcal{O}(E)$ is a uniformizer, we have
\[ \chi_w(\varpi_w) = \chi_{v_+}(\varpi_w)^{-1} \chi_{v_-}(\varpi_w)^{-1}. \]
\end{itemize}
This concludes the description of all endoscopic global parameters for $G$ which
are discrete at both real places and unramified at all finite places. They are
parametrized by pairs $(a_+, a_-) \in (\Z \smallsetminus \{0\})^2$ such that
$a_++5a_- = 0 \mod 12$, modulo $(a_+,a_-) \sim (-a_+, -a_-)$.

Let $\chi$ be a character as above.  Then the centralizer
$\mathcal{S}_{\varphi_{\chi}}$ of $\varphi_{\chi}$ is
\[ \left\{ \begin{pmatrix} \pm 1 & 0 \\ 0 & 1 \end{pmatrix} \right\} \subset
\mathcal{T} \]
and so it coincides with the local centralizers at $v_+, v_-$. Up to replacing
$\chi$ by $\chi^{-1}$, we are in exactly one of the following cases:
\begin{itemize}
\item $a_+>0$ and $a_->0$, i.e.\ $\chi_{v_+}(z) = \tau_+(z/\bar{z})^{k_++1}$ for
$k_+ \geq 0$ and $\chi_{v_-}(z) = \tau_-(z/\bar{z})^{k_-+1}$ for $k_- \geq 0$.
Then $\langle \cdot, \pi_{v_+} \rangle \times \langle \cdot, \pi_{v_-} \rangle$
is the non-trivial character of $\mathcal{S}_{\varphi_{\chi}}$.
\item $a_+>0$ and $a_-<0$, i.e.\ $\chi_{v_+}(z) = \tau_+(z/\bar{z})^{k_++1}$ for
$k_+ \geq 0$ and $\chi_{v_-}(z) = \tau_-(z/\bar{z})^{-k_--1}$ for $k_- \geq 0$.
Then $\langle \cdot, \pi_{v_+} \rangle \times \langle \cdot, \pi_{v_-} \rangle$
is the trivial character of $\mathcal{S}_{\varphi_{\chi}}$.
\end{itemize}

Thus in weight $(k_+, k_-)$ and level $\underline{G}(\widehat{\mathcal{O}(F)})$,
there is at most one endoscopic automorphic representation, and there is one if
and only if
\begin{equation} \label{e:exmultformula} (k_++1) - 5 (k_-+1) = 0 \mod 12.
\end{equation}
In low weight, we have computed Hecke operators for small primes and verified
this condition.

\subsubsection{Comments}

The class number
\[ \mathrm{card} \left( G(F) \backslash G(\A_{F,f}) / \underline{G}(
\widehat{\mathcal{O}(F)}) \right) = 1 \]
as one can check when computing a Hecke operator at any finite place, by strong
approximation. Note that $\underline{G}$ is \emph{not} the only reductive model
of $G$, even up to the action of $G_{\mathrm{ad}}(F)$. By splitting the Azumaya
algebra $\mathcal{O}(D)$ modulo $(2) = (s-1)^2$, we can compute an
$(s-1)$-Kneser neighbour of $\mathcal{O}(D)$, that is another maximal order
$\mathcal{O}'(D)$ of $D$, having basis over $\mathcal{O}(F)$
\[ 1, Z+sI, (1-s)(s+ZI), (1-s)^{-1}(1+I+sZI). \]
It gives rise to a second model $\underline{G}'$ of $G$, which is not isomorphic
to $\underline{G}$ since one can compute using reduction theory that
$\underline{G}(\mathcal{O}_F)$ is a dihedral group of order $24$ (generated by
$Z$ and $I$, with $IZI^{-1} = Z^{-1}$), whereas $\underline{G}'(\mathcal{O}(F))$
is isomorphic to $\mathrm{SL}_2(\mathbb{F}_3)$ (an isomorphism is given by
reduction modulo $s$). One can also check that the class number
\[ \mathrm{card} \left( G_{\mathrm{ad}}(F) \backslash G_{\mathrm{ad}}(\A_{F,f})
/ \underline{G}_{\mathrm{ad}}(\widehat{\mathcal{O}(F)}) \right) = 2, \]
and so $\underline{G}$ and $\underline{G}'$ are up to isomorphism the only two
reductive models of $G$ over $\mathcal{O}(F)$. So we have two distinct notions
of ``level one'' for automorphic representations for $G$, and although the
relevant Arthur-Langlands parameters are the same in both cases, the automorphic
multiplicities differ. More precisely, any algebraic Hecke character $\chi$ for
$E$ as above contributes an automorphic representation for $G$ either in level
$\underline{G}(\widehat{\mathcal{O}(F)})$ or in level
$\underline{G}'(\widehat{\mathcal{O}(F)})$.
%\[ I' = Z^9,\ Z'= -\frac{1+s}{2}+Z+\frac{1-s}{2}ZI+\frac{1-s}{2}ZI \]
%and $J' := Z'I'(Z')^{-1}$ gives $K' := I'J' = - J'I' = Z'J'(Z')^{-1}$.

\subsubsection{Higher rank}

Alternatively, one could explicitly compute the geometric transfer factors
defined in \cite{LabLan} for $G$ and the endoscopic group $H$ isomorphic to the
unique anisotropic torus over $F$ of dimension $1$ which is split by $E$.
Although one would lose the interpretation in terms of characters of
centralizers of Langlands parameters, this would probably lead to a proof that
the multiplicity formula for $G$ in level
$\underline{G}(\widehat{\mathcal{O}(F)})$ reduces to \eqref{e:exmultformula}.

Note however that the approach in the present paper generalizes easily to higher
rank. For example, using the embedding $(\mathrm{SL}_2)^n \hookrightarrow
\mathrm{Sp}_{2n}$, it is easy to generalize the above example to the case where
$G$ is the inner form of $G^* = \mathrm{Sp}_{2n}$ over $F$ which is definite
(i.e.\ $G(F \otimes_{\Q} \R)$ is compact) and split at all finite places. This
does not require additional computation, and so one can make explicit Arthur's
multiplicity formula (also known in this case, see \cite{TaiMult}) in ``level
one''. Moreover, using also \emph{pure} inner forms of quasi-split special
orthogonal groups, namely definite special orthogonal groups obtained using
copies of $(x,y) \mapsto x^2 + s xy + y^2$ and (in odd dimension) $x \mapsto
x^2$, it is possible to carry out the same inductive strategy as in
\cite{Taidimtrace}, but using definite groups as in \cite{ChRe}, which makes
explicit computations much simpler. Therefore the above example makes it
possible to explictly compute automorphic cuspidal self-dual representations for
general linear groups over $F$ which are unramified at all finite places and
algebraic regular at both real places.

\newpage

\bibliographystyle{amsalpha}
\bibliography{explicitgri}

\end{document}